\documentclass[letterpaper,11pt,reqno]{amsart}
\usepackage[margin=1in]{geometry}
\usepackage[T1]{fontenc}
\usepackage[utf8]{inputenc}
\usepackage{lmodern}
\usepackage{amsthm,amsmath,amsfonts,mathtools,amssymb}
\usepackage{graphicx,verbatim,subfigure,graphicx,yhmath,multirow}
\usepackage[dvipsnames]{xcolor}
\usepackage[foot]{amsaddr}
\usepackage[colorlinks=true,%
    breaklinks,%
    linkcolor=blue,%
    anchorcolor=blue,%
    citecolor=blue,%
    urlcolor=blue]{hyperref}
\definecolor{shadecolor}{gray}{0.8}
\usepackage{mathrsfs,xfrac,amsbsy,latexsym,dsfont,bm} 
\usepackage[textsize=tiny]{todonotes}

\usepackage{enumitem}
\newenvironment{enumeratei}{\begin{enumerate}[label=\textup{\roman*.},leftmargin=2em,itemsep=1ex]}{\end{enumerate}}
\newenvironment{enumeratea}{\begin{enumerate}[label=\textup{\alph*)},leftmargin=2em,itemsep=1ex]}{\end{enumerate}}

\newtheorem{thm}{Theorem}[section]

\newtheorem{lem}[thm]{Lemma}

\newtheorem{cor}[thm]{Corollary}

\newtheorem{prop}[thm]{Proposition}

\newtheorem{defn}[thm]{Definition}

\theoremstyle{remark}
\newtheorem{rem}[thm]{Remark}

\newtheorem{ass}[thm]{Assumption}

\numberwithin{equation}{section}

\renewcommand{\le}{\leqslant} 
\renewcommand{\ge}{\geqslant} 
\newcommand{\eset}{\varnothing}

\newcommand{\ind}{\mathds{1}}
\newcommand{\eps}{\varepsilon}
\newcommand{\To}{\longrightarrow}
\newcommand{\norm}[1]{\left\lVert#1\right\rVert}

\newcommand{\abs}[1]{\left\vert#1\right\vert}

\newcommand{\ie}{\emph{i.e., }}

\renewcommand{\iff}{\Leftrightarrow} 
\newcommand{\tod}{\stackrel{\mathrm{(d)}}{\longrightarrow}}
\newcommand{\toP}{\stackrel{\pr}{\to}} 
\newcommand{\equald}{\stackrel{\mathrm{d}}{=}}

\let\ga=\alpha \let\gb=\beta  \let\gd=\delta 
     \let\gl=\lambda        \let\go=\omega   \let\gs=\sigma \let\gt=\tau 
  
   \let\gL=\Lambda

\newcommand{\cB}{\mathcal{B}}\newcommand{\cC}{\mathcal{C}}
\newcommand{\cD}{\mathcal{D}}\newcommand{\cE}{\mathcal{E}}\newcommand{\cF}{\mathcal{F}}
\newcommand{\cH}{\mathcal{H}}

\newcommand{\cU}{\mathcal{U}}

\providecommand{\1}{\mathds{1}}

\newcommand{\mvgs}{\boldsymbol{\sigma}}

\newcommand{\dN}{\mathds{N}}

\newcommand{\dR}{\mathds{R}}

\newcommand{\dZ}{\mathds{Z}} 

\DeclareMathOperator{\E}{\mathds{E}}
\DeclareMathOperator{\pr}{\mathds{P}}
\DeclareMathOperator{\N}{N}
\DeclareMathOperator{\sgn}{sgn}
\DeclareMathOperator{\var}{Var}

\newcommand{\ceil}[1]{\lceil #1 \rceil}
\newcommand{\floor}[1]{\lfloor #1 \rfloor}

\let\hat\widehat
\let\tilde\widetilde

\usepackage{microtype}
\usepackage{tikz}
\usetikzlibrary{arrows.meta,decorations.pathreplacing,calc}
\allowdisplaybreaks[2]

\graphicspath{{figs/}}
\newcommand{\hub}{\mathscr{H}}       
\newcommand{\Nn}{N_n}        
\newcommand{\Dl}{D_{\ell}}   
\newcommand{\rhox}{\rho}     
\newcommand{\Hloc}{\cH_{\mathrm{loc}}}
\newcommand{\slv}{\lambda}   
\newcommand{\hop}{\mathsf{H}} 
\newcommand{\Ares}{\Delta}   

\newcommand{\cut}{\mathrm{cut}}

\begin{document}
\title[Scaling limit for linear 1--D LRFPP]{Scaling limit and universality for one-dimensional long-range first-passage metric in the linear-growth regime}
\author[Chatterjee]{Shirshendu Chatterjee$^\star$}
\author[Dey]{Partha S.~Dey$^\dag$}
\date{}
\address{$^\star$Department of Mathematics, City University of New York, New York, NY, USA}
\address{$^\dag$Department of Mathematics, University of Illinois Urbana--Champaign, Urbana, Illinois 61801}
\email{$^\star$shirshendu@ccny.cuny.edu, $^\dag$psdey@illinois.edu}

\begin{abstract}
We prove universal Gaussian and stable fluctuation limits for one-dimensional long-range first-passage percolation with edge times $\abs{x-y}^{\ga}\go_{x,y}$.
The positive i.i.d.~weights satisfy $\pr(\go\le t)=t^{\theta+o(1)}$ as $t\downarrow0$ and $\E(\go^\zeta)<\infty$ for some $\theta,\zeta>0$.
For $\ga>1+2/\theta$, the passage time $T(0,n)$ has a positive deterministic linear growth rate $\mu$.
For $\ga>1+3/\theta$, we establish a Gaussian central limit theorem with linear asymptotic variance, a functional version, and limit theorems for geodesic hop and cut point counts.
For $1+2/\theta<\ga<1+3/\theta$, assuming a regularly varying lower tail, we obtain a totally left-skewed stable limit with index $\gb=\theta(\ga-1)-1$ and an explicit scale, together with convergence to a spectrally negative stable L\'evy process in the Skorokhod $M_1$ topology.
The proof constructs i.i.d.~blocks between hubs through which every crossing geodesic must pass.
A spread-out comparison metric makes this regeneration available for every $\zeta>0$, even when the edge law has infinite mean. The key estimate identifies the exact left-tail asymptotic of the centered block reward, generated by a single fast long edge. A monotone within-block approximation yields the functional stable limit.
\end{abstract}


\maketitle
\setcounter{tocdepth}{1}\tableofcontents

\section{Introduction}\label{sec:intro}

\subsection{Model}\label{ssec:model}
Let $\{\go_{x,y}: x<y,\ x,y\in\dZ\}$ be i.i.d.~copies of a random variable $\go$ with distribution function $F(t)=\pr(\go\le t)$.
Define $\go_{y,x}=\go_{x,y}$ and assign the edge $e=\{x,y\}$ the passage time
\begin{align}\label{eq:edge-weight}
	t_e=t_{x,y}:=\abs{x-y}^{\ga}\go_{x,y},\quad \ga>0.
\end{align}
We write $d_e:=\abs{x-y}$ for the length of $e$.
For a finite path $\pi=(x_0,x_1,\ldots,x_m)$, let $t(\pi)=\sum_{i=1}^m t_{x_{i-1},x_i}$, and define the first-passage time
\begin{align}\label{eq:T}
	T(x,y):=\inf_{\pi:x\to y}t(\pi),\quad T_n:=T(0,n).
\end{align}
When all vertices of $\pi$ are required to lie in a set $I$ we write $T_I(x,y)$ for the restricted passage time.
The graph is the complete graph on $\dZ$, so every vertex has infinitely many neighbors.
In the regime considered here $T(x,y)$ is attained by a finite path (Lemma~\ref{lem:pivotal}).

\begin{ass}\label{ass:law}
	The law of $\go$ is non-degenerate and satisfies the following.
	\begin{enumeratei}
		\item $\pr(0<\go<\infty)=1$.
		\item For some $\theta\in(0,\infty)$,
		\begin{align}\label{eq:lower-tail}
			\lim_{t\downarrow0}\frac{\log F(t)}{\log t}=\theta.
		\end{align}
		\item $\E(\go^{\zeta})<\infty$ for some $\zeta>0$.
	\end{enumeratei}
\end{ass}

Condition~\eqref{eq:lower-tail} means that for every $\eps>0$ there are $t_\eps,c_\eps,C_\eps>0$ with
\begin{align}\label{eq:two-sided-F}
	c_\eps t^{\theta+\eps}\le F(t)\le C_\eps t^{\theta-\eps},\quad 0<t<t_\eps.
\end{align}
Only the upper bound and the fact that $0\in\operatorname{supp}\omega$ are used for the Gaussian theorem. The stable theorem needs the more precise regular-variation hypothesis of Assumption~\ref{ass:rv} below.
The exponential law of~\cite{CD16} has $\theta=1$.
The law of $\go=U^{1/\theta}$ with $U$ uniform has exactly $F(t)=t^{\theta}$ on $[0,1]$.
A gamma distribution with shape parameter $\theta$ and a Weibull law with shape parameter $\theta$ satisfy~\eqref{eq:lower-tail}.
The upper-tail parameter $\zeta$ can be arbitrarily small, so $\go$ may have infinite mean and infinite variance.

	The passage time $\abs{x-y}^{\ga}\go_{x,y}$ combines two features.
	The deterministic factor $\abs{x-y}^{\ga}$ with $\ga>1/\theta$ makes a direct long jump more expensive than the same distance covered in short steps of typical cost.
	The metric collapses when $\ga<1/\theta$, see~\cite{CD16,CDK26}.
	The random factor $\go_{x,y}$ is the disorder.
	Its lower tail decides how often a long edge is cheap enough to compete with a linear route,
	and a sufficiently large fixed family of alternative local routes controls the effect of its upper tail,
		which is the reason no moment of $\go$ beyond $\E(\go^\zeta)<\infty$ is needed.
	The complete graph on $\dZ$ is the natural one-dimensional setting in which every pair of sites interacts,
		as in L\'evy-flight dispersal or long-range epidemic models.
	The factors $1/\theta$ in the thresholds suggest a formal effective-dimension analogy (see \S\ref{ssec:heuristics}).
We work throughout in the linear regime
\begin{align}\label{eq:linear-regime}
	\ga>1+2/\theta.
\end{align}
The heuristic reason for this threshold, and for the second threshold $\ga=1+3/\theta$,
	is explained in \S\ref{ssec:heuristics}.
In short, an edge of length $d$ can compete with a route of linear cost only if $\go\lesssim d^{1-\ga}$,
	an event of probability $d^{-\theta(\ga-1)+o(1)}$.
About $d$ such edges cross a given cut.
A positive density of cuts uncrossed by such edges needs $\sum_d d\cdot d^{-\theta(\ga-1)}<\infty$,
	whereas a finite second moment for the spacing between such cuts needs $\sum_d d^2\cdot d^{-\theta(\ga-1)}<\infty$.

\begin{rem}[Conjecture in higher dimension]
			A formal higher-dimensional count using $n^d$ candidate endpoints
and a nearest neighbor first-passage-time fluctuation exponent $\chi_d>0$ suggests that
$$\ga_c=1+\frac d\theta\left(1+\frac1\chi_d\right).$$ That is, for the long-range first-passage percolation model on $\dZ^{d}$ with $t_{x,y}=\norm{x-y}^{\alpha}\omega_{x,y}$  and $\omega$ satisfying Assumption~\ref{ass:law}, the fluctuation behavior should be universal for $\ga>\ga_c$. Here $\chi_{d}$ is the conjectured passage-time fluctuation exponent  for nearest-neighbor first-passage percolation  on $\dZ^{d}$. 
This is only a heuristic,  the geometry of the relevant endpoints and geodesics must also be controlled to make it rigorous. In this article, we verify the $d=1$ dimensional case, where the linearity threshold is $\ga>1+2/\theta$, $\chi_{1}=1/2$ and $\ga_{c}=1+3/\theta$.  
	\end{rem}

\subsection{Main results}\label{ssec:main}
Our first result is a law of large numbers and a Gaussian central limit theorem above the second threshold, valid for every $\zeta>0$.

\begin{thm}[Gaussian fluctuations]\label{thm:main}
	Suppose Assumption~\ref{ass:law} holds and
	\begin{align}\label{eq:alpha-condition}
		\ga>1+3/\theta.
	\end{align}
	Then there are constants $\mu\in(0,\infty)$ and $\gs^2\in(0,\infty)$ such that, as $n\to\infty$,
	\begin{align}
		& T_n/n\To\mu\text{ a.s.~and in }L^1,\label{eq:LLN-main}                 \\
		& (T_n-\mu n)/\sqrt{n}\tod \N(0,\gs^2),\label{eq:CLT-main}               \\
		& \E T_n/n\To\mu \text{ and } \var(T_n)/n\To\gs^2.\label{eq:mom-main}
	\end{align}
\end{thm}

By~\eqref{eq:CLT-main},~\eqref{eq:mom-main} and Slutsky's theorem, $(T_n-\E T_n)/\sqrt{\var(T_n)}\tod \N(0,1)$.
No first or second moment of $\go$ is required.
The constants have a regenerative description.
Let $\cdots<\hub_{-1}<\hub_0\le0<\hub_1<\hub_2<\cdots$ be the hubs constructed in Section~\ref{sec:hubs}.
Define
\begin{align}\label{eq:LG}
	L_i:=\hub_{i+1}-\hub_i,\quad G_i:=T(\hub_i,\hub_{i+1}),\quad i\in\dZ.
\end{align}
The pairs $(L_i,G_i)$, $i\ne0$, are i.i.d. We write $(L,G)$ for a pair with their common law and denote expectation and variance under this Palm block law by $\E^0$ and $\var^0$.
Then
\begin{align}\label{eq:constants}
	\mu={\E^0 G}/{\E^0 L},\quad \gs^2={\var^0(G-\mu L)}/{\E^0 L}.
\end{align}

The same proof gives a process-level statement and limit theorems for two geometric quantities of the geodesic.
Let $D[0,\infty)$ carry the Skorokhod $J_1$ topology. Define a vertex $v$ to be a \emph{cut point} for the random geometry, if there exists a finite (random) integer $k\ge1$ such that $v$ belongs to all geodesics from $u$ to $w$ for $u<v-k, w>v+k$. The hubs constructed in this paper are cut points (Lemma~\ref{lem:pivotal}).
Write $\cut_n$ for the number of cut points in $(0,n]$,
	and let $\hop_n$ be the smallest number of edges of a geodesic from $0$ to $n$.
A geodesic between consecutive hubs is a simple path inside their block,
	so it has at most as many edges as the block has vertices.

\begin{cor}[Functional CLT]\label{cor:fclt}
	Under the assumptions of Theorem~\ref{thm:main},
	\begin{align*}
		\big( (T(0,\floor{nt})-\mu nt)/\sqrt n\big)_{t\ge0}
		\tod \big(\gs B(t)\big)_{t\ge0}
	\end{align*}
	in $D[0,\infty)$, where $B$ is a standard Brownian motion.
\end{cor}

\begin{cor}[Hop and cut point counts]\label{cor:count}
	Under the assumptions of Theorem~\ref{thm:main},
		there are constants $\nu_H,\nu_{\cut}\in(0,1]$ and $\gs_H^2,\gs_{\cut}^2\in[0,\infty)$ such that
	\begin{align*}
		 \frac{\hop_n}{n}\To\nu_H,\quad \frac{\cut_n}{n}\To \nu_{\cut},\text{ a.s.},\quad \frac{\hop_n-\nu_Hn}{\sqrt n}\tod \N(0,\gs_H^2),\quad
		\frac{\cut_n-\nu_{\cut}n}{\sqrt n}\tod \N(0,\gs_{\cut}^2).
	\end{align*}
Let $h$ be the smallest number of edges of a geodesic between the endpoints of a typical Palm block, and let $c$ be the number of cut points in that block, with the left endpoint excluded and the right endpoint included. Then
\begin{align*}
 \nu_H&=\frac{\E^0h}{\E^0L},\ 
 \nu_{\cut}=\frac{\E^0c}{\E^0L},\ 
 \gs_H^2=\frac{\var^0(h-\nu_HL)}{\E^0L},\ 
 \gs_{\cut}^2=\frac{\var^0(c-\nu_{\cut}L)}{\E^0L}.
\end{align*}
\end{cor}
The constants $\mu,\gs^2,\nu_H,\nu_{\cut},\gs_H^2$ and $\gs_{\cut}^2$ are intrinsic. Their values are uniquely determined by the laws of the passage time. Their regenerative formulas may be computed using any hub construction satisfying the proved hypotheses.
Our second result concerns the intermediate window $1+2/\theta<\ga<1+3/\theta$, where hubs still have positive density but their spacing has infinite variance.
We need the lower tail of $\go$ to be regularly varying rather than merely logarithmically equivalent to a power.

\begin{ass}[Regular variation at zero]\label{ass:rv}
	$F(t)=t^{\theta}\,\slv(t)$ for $0<t\le1$, where $\slv$ is slowly varying at $0$,
		\ie~$\slv(st)/\slv(t)\to1$ as $t\downarrow0$ for every $s>0$.
\end{ass}

Define the \emph{stable index}
\begin{align}\label{eq:beta}
	\gb:=\theta(\ga-1)-1,\text{ so that } \gb\in(1,2)\iff 1+2/{\theta}<\ga<1+3/{\theta}.
\end{align}
For $x>0$, let
\begin{align}\label{eq:rho}
	\rhox(x):=\sum_{d\ge1}F\big((\mu d-x)_+\,d^{-\ga}\big)
\end{align}
be the expected number of edges with left endpoint at a given vertex whose passage time undercuts the linear cost $\mu d$ of their length $d$ by at least $x$.
Here $\mu$ is the time constant of~\eqref{eq:LLN-main},
	which exists throughout the linear-growth regime (Proposition~\ref{prop:lln}).
Under Assumption~\ref{ass:rv}, $\rhox$ is regularly varying at infinity with index $-\gb$ (Lemma~\ref{lem:rho}), and
\begin{align}\label{eq:rho-asymp}
	\rhox(x)\sim \mu^{\theta\ga-1}\,\mathrm{B}(\theta+1,\gb)\;x^{-\gb}\,\slv(x^{1-\ga}),\quad x\to\infty,
\end{align}
where $\mathrm{B}(\cdot,\cdot)$ is Euler's Beta function.
Let $(a_n)$ be any sequence with $$n\rhox(a_n)\to1.$$ Any two such sequences are asymptotically equivalent.
Then $a_n=n^{1/\gb}\slv^\ast(n)$ for a slowly varying $\slv^\ast$.
If $F(t)\sim c_Ft^{\theta}$ as $t\downarrow0$,
	then $a_n\sim\big(\mu^{\theta\ga-1}\mathrm{B}(\theta+1,\gb)\,c_F\,n\big)^{1/\gb}$.

\begin{thm}[Stable fluctuations]\label{thm:stable}
	Suppose Assumptions~\ref{ass:law} and~\ref{ass:rv} hold, and $1+2/\theta<\ga<1+3/\theta$.
	Then, with $\mu$ as in~\eqref{eq:LLN-main} and $a_n$ as above,
	\begin{align}\label{eq:stable-main}
		\frac{T_n-\mu n}{a_n}\tod S_{\gb}\big(\gs_\gb,-1,0\big),\quad
		\text{where }\ \gs_\gb^{\gb}:=\Gamma(1-\gb)\cos({\pi\gb}/2)\in(0,\infty).
	\end{align}
	Here $S_{\gb}(\gs_\gb,-1,0)$ denotes the $\gb$-stable law with characteristic function
	\begin{align}\label{eq:stable-cf}
		t\mapsto\exp\big\{-\gs_\gb^{\gb}\abs{t}^{\gb}\big(1+i\,\sgn(t)\tan(\pi\gb/2)\big)\big\},
	\end{align}
	that is, the stable law with index $\gb$, scale $\gs_\gb$,
		skewness parameter $-1$ and shift $0$ in the parametrization of~\cite{ST94}.
	Moreover $\E T_n-\mu n=o(a_n)$, so~\eqref{eq:stable-main} also holds with $\E T_n$ in place of $\mu n$.
\end{thm}

The limit is totally skewed to the left.
It is the law at time $1$ of the spectrally negative L\'evy process with L\'evy measure $\gb\abs{y}^{-\gb-1}\1_{\{y<0\}}\,dy$ and no Gaussian part.
The sign of the skewness has a transparent meaning.
Fluctuations of $T_n$ below its mean are produced by single anomalously fast long edges,
	each of which saves a time of the order of its length.
Deviations above the mean require the environment to be slow along a whole stretch and are much less likely.
In the language of the renewal structure,
	the reward $G-\mu L$ of a hub block has a regularly varying left tail with index $\gb$ and a lighter right tail.
Establishing this \emph{tail equivalence} for the reward,
	rather than merely a tail exponent for $L$, is the main technical step of the proof (Section~\ref{sec:tails}).
	
\begin{thm}[Functional stable limit]\label{thm:stable-functional}
	Under the assumptions of Theorem~\ref{thm:stable},
	\begin{align}\label{eq:stable-functional}
		\big((T(0,\floor{nt})-\mu nt)/{a_n}\big)_{t\ge0}\tod \big(X_{\gb}(t)\big)_{t\ge0}
	\end{align}
	in $D[0,\infty)$ equipped with the Skorokhod $M_1$ topology,
		where $X_\gb$ is the spectrally negative $\gb$-stable L\'evy process with no Gaussian part,
		L\'evy measure $\gb\abs{y}^{-\gb-1}\1_{\{y<0\}}\,dy$ and $\E X_\gb(1)=0$, so that $X_\gb(1)\equald S_\gb(\gs_\gb,-1,0)$.
\end{thm}

\begin{rem}[The boundary]\label{rem:boundary}
	The Gaussian theorem and the stable theorem cover the open parameter ranges on either side of $\ga=1+3/\theta$ within $\ga>1+2/\theta$. At $\ga=1+3/\theta$, one expects Gaussian behavior, possibly with a non-standard normalization depending on the slowly varying function $\slv$.
\end{rem}

\subsection{Heuristics: fast long edges and the two thresholds}\label{ssec:heuristics}

The mechanism behind Theorems~\ref{thm:main} and~\ref{thm:stable} can be summarized in a few lines.
Suppose that the passage time grows linearly, $T(x,y)\approx\mu\abs{x-y}$.
An edge $e=\{x,y\}$ of length $d=\abs{x-y}$ is useful to a geodesic only if $t_e=d^{\ga}\go_e$ is at most of order $\mu d$, \ie\ $\go_e\lesssim d^{1-\ga}$.
By~\eqref{eq:lower-tail} this event has probability $F(d^{1-\ga})=d^{-\theta(\ga-1)+o(1)}$.
Such an edge saves a time $\mu d-t_e$ of order $d$,
	and there are about $d$ edges of length $d$ crossing any given point of the line.
Hence
\begin{align}\label{eq:heur-density}
	\sum_{d\ge1}d\cdot d^{-\theta(\ga-1)}<\infty\iff \ga>1+2/\theta
\end{align}
is the condition under which a given point has positive probability of being crossed by no useful long edges.

Points that are crossed by no useful edge, and at which a linear-cost local route exists on both sides, are \emph{hubs}.
Every geodesic passing from their left to their right must go through them (Lemma~\ref{lem:pivotal}).
Between two consecutive hubs the passage time is a function of the edges inside the block only.
The cut factorization and localization proved in Proposition~\ref{prop:regeneration} make the marked blocks i.i.d.; $T_n$ is then, up to boundary terms, a sum of a renewal number of i.i.d.~rewards.
This gives the law of large numbers in the whole linear-growth regime and reduces the fluctuation question to the tails of the block length $L$ and of the block reward $G-\mu L$.

The block containing a useful edge of length $d$ has length at least $d$.
Heuristically, the expected number of useful edges with a fixed left endpoint and length exceeding $x$ is
$\sum_{d>x}d^{-\theta(\ga-1)}=x^{-\gb+o(1)}$, where $\gb=\theta(\ga-1)-1$.
The additional factor $d$ in the count of edges crossing a fixed point instead gives the stationary recurrence exponent $1-\gb$.
These are different distributions: the block seen from a typical vertex is size-biased, whereas $L$ has the Palm law of a typical block.
Theorem~\ref{thm:momhub} in Appendix~\ref{app:nn} establishes the corresponding moment threshold for the nearest-neighbor construction under its stated upper-tail assumptions, away from critical equality.
The series
\begin{align}\label{eq:heur-variance}
 \sum_{d\ge1}d^2 d^{-\theta(\ga-1)}
\end{align}
converges for $\ga>1+3/\theta$ and diverges for $\ga<1+3/\theta$.
This predicts the finite second-moment threshold. 
Above it, the classical renewal-reward central limit theorem applies.
Below it, $\gb\in(1,2)$, and the sum of $n/\E L$ i.i.d.~rewards whose left tail is regularly varying with index $\gb$ lies in the domain of attraction of a totally left-skewed $\gb$-stable law, scaled by $n^{1/\gb}$.
The scale $a_n$ in Theorem~\ref{thm:stable} is the level at which the expected number of edges in $[0,n]$ saving more than $a_n$ equals one, $n\rhox(a_n)=1$.
The Beta-function constant in~\eqref{eq:rho-asymp} comes from integrating $F((\mu d-x)d^{-\ga})$ over the lengths $d>x/\mu$. In Proposition~\ref{prop:zero-speed}, we present a self-contained proof of zero-speed below $\alpha<1+2/\theta$, which also follows from the results in~\cite{CDK26}.

\subsection{Main idea and contribution}\label{ssec:method}
The proof has three layers, developed in Sections~\ref{sec:spread}--\ref{sec:clt}.
\begin{enumeratei}
	\item \emph{A finite-range comparison metric with good upper-tail integrability.} A nearest-neighbor path from $0$ to $n$ has passage time $\sum_{i<n}\go_{i,i+1}$,
			which is useless when $\E\go=\infty$.
		We replace it by FPP on the $\ell$-spread-out line, which uses all edges of length at most $\ell$.
		A unit cut of that graph contains $\ell$ groups of edge-disjoint routes,
			and the minimum over them has a finite moment of order $p$ as soon as $\zeta\ell>p$ (Lemma~\ref{lem:amplify}).
		The explicit family of routes is taken from Dey and Kim~\cite{DK26},
			where this amplification phenomenon and its sharp cut exponent are the central theme.
		Here it is used as an auxiliary device inside an infinite-range model.
		The range $\ell$ depends on $\zeta$ but stays fixed as $n\to\infty$.
	\item \emph{Hubs.} A hub is a cell boundary at which the spread-out metric is linearly controlled on both sides,
			a finite local ``portal'' event forces all short crossings through the boundary,
			and no anomalously fast long edge crosses the boundary (Definition~\ref{def:hub}).
		The long-edge part of the hub event has positive probability precisely because $\ga>1+2/\theta$.
	\item \emph{Regeneration.} An exploration of the fast long edges,
			combined with the renovating events supplied by the finite-range metric,
			produces i.i.d.~hub blocks (Proposition~\ref{prop:regeneration}).
		The block length has a finite second moment when $\ga>1+3/\theta$,
			and an upper tail bounded on the regularly varying scale of index $-\gb$ when $1+2/\theta<\ga<1+3/\theta$.
		The Gaussian theorem then follows from the renewal-reward CLT and Anscombe's theorem~\cite{Ans52,Gut09,GW93}.
		The stable theorem follows from the single-big-jump tail asymptotics of Section~\ref{sec:tails} together with the stable limit theorem for randomly indexed sums.
\end{enumeratei}

Regeneration decompositions along pivotal vertices have a long history in one-dimensional percolation and interacting particle systems,
	going back at least to the limit theorems of Cox and Durrett~\cite{CD81}.
Regenerative methods for essentially one-dimensional FPP were used by Ahlberg~\cite{Ahlberg15},
	and skeleton constructions for long-range last-passage percolation appear in \cite{FMS12}.
Events that depend on both the past and the future of the environment, like our hub events,
	are naturally handled by extended renovation theory~\cite{FK03}.
We state and prove the specialized regeneration argument needed here, so no general renovation theorem is required.
The use of $M_1$ to capture monotone clusters of stable fluctuations has classical precedents in Avram and Taqqu~\cite{AT92} and Basrak, Krizmani\'c and Segers~\cite{BKS12}.
Heavy-tailed renewal-reward mechanisms also occur in the network-traffic models of Mikosch, Resnick, Rootz\'en and Stegeman~\cite{MRRS02}.
The single-big-jump principle is developed in Denisov, Dieker and Shneer~\cite{DDS08} and Foss, Korshunov and Zachary~\cite{FKZ13}. Our problem is to establish the precise block-reward version for this metric.
Hambly and Martin~\cite{HM07} obtain big-weight-driven limits in two-dimensional last-passage percolation, a related but different mechanism.
The main contributions of this paper are
\begin{enumeratei}
	\item the explicit two-sided pivotal event whose probability is controlled by elementary large deviation and fluctuation estimates for the edge-weight law,
			rather than by a percolation-theoretic renormalization,
	\item the use of a spread-out comparison metric to allow an arbitrarily small positive upper-tail moment, and
	\item the exact tail asymptotic for the block reward that yields the stable limit with an explicit scale.
\end{enumeratei}

\subsection{Literature review}
First-passage percolation (FPP) is a model for random paths through a disordered medium.
Each edge of a graph carries a random traversal time, and one studies the induced random metric.
Introduced by Hammersley and Welsh~\cite{HW65} as a model for the spread of a fluid or an infection through a random medium,
	it has become a central object in probability and statistical physics, see~\cite{K86,SW78,GK12,ADH17} for background.
Subadditivity, through the ergodic theorems of Kingman and Liggett~\cite{L85},
	gives a deterministic linear growth rate (the \emph{time constant}) and a shape theorem in great generality.
Fluctuations are far more delicate.
On $\dZ^2$ the passage time is expected to display Kardar--Parisi--Zhang scaling~\cite{KS91}.
For ordinary noncritical nearest-neighbor FPP on $\dZ^2$, a general fluctuation limit remains open. Gaussian fluctuations of order $\sqrt n$ are known in thin-cylinder and essentially one-dimensional settings~\cite{CD09,Ahlberg15}.
In an essentially one-dimensional model one expects instead that \emph{regeneration},
	rather than KPZ geometry, governs the fluctuations~\cite{Ahlberg15}.
The difficulty in the model studied here is that the underlying graph is complete.
Every vertex has infinitely many neighbors,
	and the edge law is allowed to be so heavy-tailed that it need not have a first moment.

Long-range percolation on $\dZ^d$, in which an edge between $x$ and $y$ is retained independently with probability decaying polynomially in $\abs{x-y}$,
	has been studied since the 1980s as a tractable interpolation between nearest-neighbor percolation and mean-field percolation on the complete graph.
On $\dZ$ the one-dimensional case already exhibits rich critical behavior.
The regimes $s<2$, $s=2$ and $s>2$ differ qualitatively for connection probabilities decaying as $\abs{x-y}^{-s}$~\cite{NS86,S99}; the existence of an infinite cluster also depends on the short-edge probabilities and on the coefficient of the tail,
	and the boundary case $s=2$ displays a further,
	discontinuous transition governed by the constant in front of the polynomial~\cite{AN86,IN88}.
Geometric quantities of the resulting random graph, such as the diameter, the chemical distance,
	the mixing time, and heat-kernel and resistance estimates,
	have been analyzed in~\cite{BB01,B02,B04,B04a,CGS02,BBY08,misumi2008estimates,CS09,T10,biskup2011graph,DS13}.
\emph{Scale-free percolation}, in which each vertex carries an i.i.d.~weight and the presence of an edge depends on both the product of the weights and the distance~\cite{DHH13},
	interpolates between long-range percolation and inhomogeneous random graphs and is a canonical spatial model for scale-free networks.

The long-range \emph{first-passage} analogue assigns a passage time growing with distance to every pair of vertices. Earlier epidemic formulations appear in~\cite{M72}. The lattice growth regimes relevant here were analyzed by the authors in~\cite{CD16}.
On $\dZ^d$, with edge times $\abs{x-y}^{\ga}$ times i.i.d.~exponential variables,
	four growth regimes for the ball of radius $t$ appear as $\ga$ increases:
	instantaneous percolation ($\ga<d$), stretched-exponential growth ($d<\ga<2d$),
	superlinear polynomial growth ($2d<\ga<2d+1$), and linear growth ($\ga>2d+1$).
The last regime recovers the linear growth behavior of nearest-neighbor FPP.

The present article extends the one-dimensional linear-growth regime of~\cite{CD16} from exponential edge weights to a general,
	possibly heavy-tailed, law for $\go$.
Writing $\theta$ for the exponent governing the lower tail of $\go$ near $0$ (so that $\theta=1$ covers the exponential case),
	the linear-growth threshold becomes $\ga>1+2/\theta$.
At $\theta=1$, it reduces to $\ga>3=2d+1$ as in~\cite{CD16}.
The finer Gaussian-versus-stable dichotomy that we identify at $\ga=1+3/\theta$ reduces at $\theta=1$ to $\ga=4$.

A finite complete-graph version of long-range FPP,
	extending the mean-field results of Janson~\cite{J99},
	has recently been analyzed by van der Hofstad and Lodewijks~\cite{HL24}.
Together with~\cite{CD16} this is a close antecedent of the present work.
The paper~\cite{CD16} studies growth regimes on the infinite lattice, whereas~\cite{HL24} studies the finite torus in the instantaneous regime.
	Neither establishes the fluctuation limits claimed here for the linear-growth regime.
The microscopic regime of FPP on the spread-out line is analyzed by Dey and Kim~\cite{DK26}. That paper proves Gaussian and stable limits using pivotal blocks. However, its stable fluctuations arise from large upper-tail edge times and are totally right-skewed. The stable fluctuations here arise from small lower-tail weights on long edges and are totally left-skewed. Our comparison metric uses their route idea, while our infinite-range cut constraints and reward-tail calculation require separate arguments.

Degree-dependent FPP on spatial random graphs is studied by Komj\'athy, Lapinskas, Lengler and Schaller~\cite{KLLS24,KLLS26}, where transmission penalties depend on endpoint degrees and spatial distance enters the connection probabilities.
Related recent work includes the spread-out cycle model of Dey and Kim~\cite{DKcycles26} and the random vertex environment of Chatterjee, Dey and Kim~\cite{CDK26}.
These results concern different geometries or growth regimes and do not substitute for the fluctuation and within-block estimates proved here.
Long-range \emph{last}-passage percolation on the line,
	whose skeleton structure is close in spirit to the hubs used here, was studied by Foss, Martin and Schmidt~\cite{FMS12}.

For broader context, first-passage percolation on random graphs with mean-field geometry (the complete graph,
	the configuration model, inhomogeneous random graphs) has an extensive literature,
	see~\cite{J99,HHM01,B08,BHH10} and the references therein.
Competition versions of first-passage percolation have been studied on the lattice in~\cite{GM05,H05,B10},
	and a continuum analogue of FPP with linear growth is analyzed in~\cite{GM08}.
Long-range dispersal kernels of the type considered here arise naturally in spatial epidemic and population-genetics models,
	where the passage time is the time for an infection or a favorable allele to spread over a given distance,
	see~\cite{M72,FM04,CMM06,MMC11,SRC11} for a sample of that parallel literature,~\cite{MN00,CNSW00,SCB02,DJ07} for epidemics and percolation on small-world and scale-free networks,
	and~\cite{GB86,JOWH98,HH98} for the physics of L\'evy-flight spreading.
Gossip processes on the complete graph and on small worlds~\cite{A10,CD11,BR12} provide yet another mean-field point of comparison.

\subsection{Notation}\label{ssec:notation}
We use $\dZ,\dR,\dR_+$ to denote the integers, the reals and the non-negative reals.
$F$ is the distribution function of $\go$, $\bar F=1-F$,
	and $\theta,\zeta,\ga$ are as in Assumption~\ref{ass:law} and~\eqref{eq:edge-weight}.
$T(x,y)$, $T_I(x,y)$ and $T_n$ are as in~\eqref{eq:T}.
The spread-out range $\ell$, the cell size $B$, the cell envelopes $Y_m$,
	the linear-control constants $a$ and $C_0$,
	the portal constant $b$ and the forbidden-edge probabilities $p_d$ are introduced in Section~\ref{sec:spread}.
$(\hub_i)_{i\in\dZ}$ is the two-sided stationary sequence of hubs with $\hub_0\le0<\hub_1$,
	$(L_i,G_i)$ are the block lengths and rewards~\eqref{eq:LG}, $(L,G)$ their common law for $i\ne0$,
	$Z_i:=G_i-\mu L_i$, and $\Nn=\max\{i:\hub_i\le n\}$.
$\gb=\theta(\ga-1)-1$ is the stable index, $\rhox$ the tail function~\eqref{eq:rho}, and $a_n$ the stable scaling.
We write $\E(X;A):=\E(X\1_A)$,
	$X\equald Y$ for equality in law, $\tod$ for convergence in distribution,
	$\toP$ for convergence in probability,
	$\N(0,\gs^2)$ for the centered Gaussian law with variance $\gs^2$,
	$f\sim g$ for $f/g\to1$, and $f\asymp g$ if $f/g$ is bounded above and below by positive constants. $C,C',c,\ldots$ denote positive finite constants whose value may change from line to line.

\subsection{Organization}\label{ssec:organization}
The paper is organized as follows. 
Section~\ref{sec:spread} constructs the spread-out comparison metric,
	proves the moment amplification lemma, and introduces the local linear-control, portal and forbidden-edge events.
Section~\ref{sec:hubs} defines hubs, proves the pivotal property,
	and establishes the regeneration structure together with the moment bounds on the block length.
Section~\ref{sec:hubs} also proves the law of large numbers in the whole linear-growth regime; Section~\ref{sec:clt} proves Theorem~\ref{thm:main} with its corollaries.
Section~\ref{sec:tails} proves the tail asymptotics of $\rhox$ and the tail equivalence for the block reward,
	and Section~\ref{sec:stable} completes the proof of Theorem~\ref{thm:stable} and proves its functional version,
	Theorem~\ref{thm:stable-functional}.
Section~\ref{sec:sim} contains results from some numerical experiments.
Appendix~\ref{app:nn} presents the nearest-neighbor hub construction, valid when $\E\go<\infty$,
	together with sharp moment thresholds for its hub gaps under the additional upper-tail moment assumptions in Theorem~\ref{thm:momhub}, away from critical equality.

\subsection*{Declaration on the use of AI}
AI tools were used as an assistive tool for grammar and language editing, consistency checks of notation and the identification of potentially relevant references.
All mathematical statements, proofs, references, and final verification were written and reviewed by the authors, who take full responsibility for the content of the manuscript.

\subsection*{Acknowledgements} The authors thank Arijit Chakrabarty for providing references for the stable law convergence and Daecheol Kim for many helpful discussions.

\section{Spread-out local control}\label{sec:spread}

Throughout this section we only assume Assumption~\ref{ass:law} and the linear-growth condition $\ga>1+2/\theta$.
The stronger condition~\eqref{eq:alpha-condition} enters only in Lemma~\ref{lem:summability}.(ii).

\subsection{Moment amplification by edge-disjoint routes}\label{ssec:amplify}
We choose the moment parameter $p$ before choosing the auxiliary range $\ell$.
In the Gaussian regime, fix $\eps_F\in(0,\theta)$ with
$e_F:=(\theta-\eps_F)(\ga-1)>3$, put $\delta_F:=e_F-3$, and set
\[
 \kappa_G:=\frac{1+\max\{3/4,\,2/(2+\delta_F)\}}2,\quad
 \delta_G:=\frac12\min\{3\kappa_G-2,\,\kappa_G(2+\delta_F)-2\}>0.
\]
Take $p>\max\{10,2+2/\delta_G\}$.
In the stable regime take $p$ as in~\eqref{eq:p-stable}.
For the law of large numbers alone any $p\ge10$ suffices.
These choices depend only on $\ga,\theta$ and remove any dependence of the moment choice
on a decay exponent established later.
Since $\zeta>0$ we may choose an integer $\ell\ge2$ so large that
\begin{align}\label{eq:ell-choice}
	\zeta\ell>p.
\end{align}
Let $\dZ^{(\ell)}$ be the graph on $\dZ$ in which $x$ and $y$ are adjacent when $1\le\abs{x-y}\le\ell$,
	with the same edge times~\eqref{eq:edge-weight}.
Its passage metric restricted to paths in an interval $I$ is denoted by $\Dl^I$, and the unrestricted one by $\Dl$.
Clearly $T\le\Dl\le\Dl^I$.

Let $K\ge1$ be an integer, to be fixed in Lemma~\ref{lem:summability}, and define $B:=4\ell K$.
For $m\in\dZ$ let $I_m:=[mB,(m+1)B]\cap\dZ$ be the $m$-th \emph{cell}.
Define the cell envelope
\begin{align}\label{eq:Y-cell}
	Y_m:=\max_{u,v\in I_m}\Dl^{I_m}(u,v).
\end{align}
The variables $(Y_m)_{m\in\dZ}$ are i.i.d., because adjacent cells share a vertex but no edge.

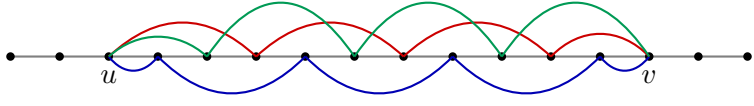
\begin{figure}[htbp]
	\centering
	\begin{tikzpicture}[x=0.65cm,y=0.5cm,>=Stealth]
		\draw[thick,black!50] (0,0) -- (15,0);
		\foreach \x in {0,...,15} {\fill (\x,0) circle (1.6pt);}
		\node[below] at (2,-0.1) {$u$};
		\node[below] at (13,-0.1) {$v$};
		\draw[red!80!black,thick] (2,0) .. controls (3,1.2) and (4,1.2) .. (5,0);
		\draw[red!80!black,thick] (5,0) .. controls (6,1.2) and (7,1.2) .. (8,0);
		\draw[red!80!black,thick] (8,0) .. controls (9,1.2) and (10,1.2) .. (11,0);
		\draw[red!80!black,thick] (11,0) .. controls (11.6,0.8) and (12.4,0.8) .. (13,0);
		\draw[blue!70!black,thick] (2,0) .. controls (2.3,-0.5) and (2.7,-0.5) .. (3,0);
		\draw[blue!70!black,thick] (3,0) .. controls (4,-1.3) and (5,-1.3) .. (6,0);
		\draw[blue!70!black,thick] (6,0) .. controls (7,-1.3) and (8,-1.3) .. (9,0);
		\draw[blue!70!black,thick] (9,0) .. controls (10,-1.3) and (11,-1.3) .. (12,0);
		\draw[blue!70!black,thick] (12,0) .. controls (12.3,-0.5) and (12.7,-0.5) .. (13,0);
		\draw[ForestGreen,thick] (2,0) .. controls (2.6,0.7) and (3.4,0.7) .. (4,0);
		\draw[ForestGreen,thick] (4,0) .. controls (5,1.9) and (6,1.9) .. (7,0);
		\draw[ForestGreen,thick] (7,0) .. controls (8,1.9) and (9,1.9) .. (10,0);
		\draw[ForestGreen,thick] (10,0) .. controls (11,1.9) and (12,1.9) .. (13,0);
	\end{tikzpicture}
	\caption{The $\ell$ edge-disjoint lane routes from $u$ to $v$ used in Lemma~\ref{lem:amplify} for $\ell=3$: lane $u\!+\!3\dZ$ (red), lane $u\!+\!1\!+\!3\dZ$ (blue), lane $u\!+\!2\!+\!3\dZ$ (green).}\label{fig:lanes}
\end{figure}

\begin{lem}[Moment amplification]\label{lem:amplify}
	\begin{enumeratei}
		\item For all $u<v$ in $I_0$ there are $\ell$ edge-disjoint paths from $u$ to $v$ in $\dZ^{(\ell)}[I_0]$,
				each with at most $(v-u)/\ell+2$ edges.
		\item With $\ell$ chosen as in~\eqref{eq:ell-choice}, $\E(Y_0^p)<\infty$.
		\item If $u<v$ with $u\in I_i$ and $v\in I_j$,
		there is an $\ell$-spread-out path from $u$ to $v$ contained in $\bigl[\, \lfloor u/B\rfloor B,(\lfloor v/B\rfloor+1)B\, \bigr]$, whose passage time is at most
		\begin{align}\label{eq:path-envelope}
			Y_i+Y_{i+1}+\cdots+Y_j.
		\end{align}
	\end{enumeratei}
\end{lem}

\begin{proof}
	(i) Suppose first that $v-u\ge\ell$.
	The edges of length exactly $\ell$ split into $\ell$ disjoint \emph{lanes},
		the residue classes $r+\ell\dZ$, $0\le r<\ell$, see Figure~\ref{fig:lanes}.
	For each $r$, let $s_r\in\{0,\ldots,\ell-1\}$ be the shift with $u+s_r\equiv r\pmod\ell$.
	Let $v-t_r$ be the largest vertex of the lane not exceeding $v$. So $t_r\in\{0,\ldots,\ell-1\}$. Consider the path
	\begin{align*}
		u\ \to\ u+s_r\ \to\ u+s_r+\ell\ \to\ \cdots\ \to\ v-t_r\ \to\ v,
	\end{align*}
	where a shift $0$ means that the corresponding edge is absent.
	Lane edges have length $\ell$ and lie in different residue classes, so they are distinct across lanes.
	The entry edges $\{u,u+s_r\}$ have distinct lengths $s_r<\ell$,
		the exit edges $\{v-t_r,v\}$ have distinct lengths $t_r<\ell$, and neither can equal a lane edge.
	An entry edge equals an exit edge only if $v-u=s_r=t_r<\ell$, which is excluded.
	Each path has at most $(v-u)/\ell+2$ edges and stays in $[u,v]$.

	If $v-u<\ell$, use the direct edge $\{u,v\}$ together with the two-edge paths $u\to w\to v$ for $w\in[v-\ell,u+\ell]\cap I_0\setminus\{u,v\}$.
	Distinct $w$ give distinct edges, and no such edge equals $\{u,v\}$.
	Since $I_0$ has length $B\ge4\ell$,
		the window $[v-\ell,u+\ell]\cap I_0$ contains at least $\ell+1$ vertices,
		so there are at least $\ell-1$ admissible $w$ and in total at least $\ell$ edge-disjoint paths of at most $2$ edges.

	(ii) Fix $u<v$ in $I_0$ and let $S_1,\ldots,S_\ell$ be the passage times of the paths in (i).
	They are independent because their edge sets are disjoint, and $\Dl^{I_0}(u,v)\le\min_rS_r$.
	Each path has at most $q:=B/\ell+2$ edges of length at most $\ell$.
	Using $(x_1+\cdots+x_q)^{\zeta}\le C_{q,\zeta}\sum_ix_i^{\zeta}$, valid for every $\zeta>0$,
	\begin{align*}
		\E(S_r^{\zeta})\le C_{q,\zeta}\,q\,\ell^{\ga\zeta}\,\E(\go^{\zeta})<\infty,\quad r=1,2,\ldots,\ell.
	\end{align*}
	Hence $\pr(S_r>t)\le Ct^{-\zeta}$ by Markov's inequality, and independence gives
	\begin{align*}
		\pr\big(\Dl^{I_0}(u,v)>t\big)\le\prod_{r=1}^{\ell}\pr(S_r>t)\le Ct^{-\zeta\ell}.
	\end{align*}
	Since $p<\zeta\ell$, $\E(\Dl^{I_0}(u,v)^p)<\infty$.
	There are at most $B(B+1)/2$ pairs $u<v$ in $I_0$, and $Y_0^p\le\sum_{u<v}\Dl^{I_0}(u,v)^p$, so $\E(Y_0^p)<\infty$.

	(iii) If $u$ and $v$ are in the same cell, use its restricted metric directly. Otherwise connect $u$ to the right endpoint of $I_i$ inside $I_i$,
		cross each intervening cell between its endpoints, and connect the left endpoint of $I_j$ to $v$ inside $I_j$.
	The contribution of cell $I_r$ is at most $Y_r$, and the concatenation stays in $\bigl[\,\lfloor u/B\rfloor B,(\lfloor v/B\rfloor+1)B\,\bigr]$.
\end{proof}

\begin{rem}\label{rem:nn}
	When $\E\go<\infty$, one may take $\ell=1$, $B=1$ and $Y_m=\go_{m,m+1}$,
		and, after omitting the portal condition, one obtains the nearest-neighbor hub of Appendix~\ref{app:nn},
		for which sharp moment thresholds are available (Theorem~\ref{thm:momhub}).
	When $\E(\go^2)=\infty$ the nearest-neighbor walk has no second moment,
		and when $\E\go=\infty$ the nearest-neighbor linear-control events have probability zero.
	This is why the spread-out envelope is needed.
\end{rem}

\subsection{Linear control, portal, and forbidden edges}\label{ssec:events}
By Lemma~\ref{lem:amplify}.(iii) applied to sub-cells of length $4\ell$,
	$Y_0$ is bounded by a sum of $K$ i.i.d.~copies of the envelope of a cell of length $4\ell$.
Hence
\begin{align}\label{eq:EY-linear}
	\E(Y_0)\le C_1K=C_1B/(4\ell),
\end{align}
with $C_1=C_1(\ell)$ independent of $K$.
Define
\begin{align}\label{eq:a-C0}
	a:=2\E(Y_0),\quad C_0:=a/B\le C_1/(2\ell),
\end{align}
so that $C_0$ is bounded uniformly in $K$.

Let $S_n:=\sum_{q=0}^{n-1}(Y_q-a)$ for $n\ge0$ and $S_{-n}:=-\sum_{q=-n}^{-1}(Y_q-a)$,
	a two-sided random walk with i.i.d.~increments of mean $-\E(Y_0)<0$ and finite $p$-th moment.
For a cell boundary $x=mB$ define the two-sided linear-control events
\begin{align}
	A_m^+ & :=\bigcap_{r\ge1}\Big\{\sum_{q=m}^{m+r-1}Y_q\le ar\Big\}=\{S_{m+r}\le S_m\text{ for all }r\ge1\},\label{eq:Aplus}    \\
	A_m^- & :=\bigcap_{r\ge1}\Big\{\sum_{q=m-r}^{m-1}Y_q\le ar\Big\}=\{S_{m-r}\ge S_m\text{ for all }r\ge1\}.\label{eq:Aminus}
\end{align}
In words, $m$ is a time at which the walk $S$ satisfies $S_{m-r}\ge S_m\ge S_{m+r}$ for all $r\ge1$ (a two-sided ladder epoch).
Since $S$ has negative drift, $\pi_0:=\pr(A_0^+)=\pr(A_0^-)>0$.
On $A_m^-\cap A_m^+$, Lemma~\ref{lem:amplify}.(iii) gives, for all $u<x<v$,
\begin{align}\label{eq:linear-local}
	\Dl(u,x)\le C_0(x-u+B),\quad \Dl(x,v)\le C_0(v-x+B),
\end{align}
because the path from $u$ to $x$ uses the cells $I_i,\ldots,I_{m-1}$ with $i=\floor{u/B}$,
	whose number $m-i$ is less than $(x-u)/B+1$, and their envelopes sum to at most $a(m-i)$ on $A_m^-$.

We also need a local event that forces short crossings through the cell boundary.
Since $0$ lies in the support of $\go$ by Assumption~\ref{ass:law}.(ii) and $\go$ is non-degenerate, there are $b>0$ and $s>b$ with
\begin{align}\label{eq:b-s}
	\pr(\go\ge s)>0\text{ and } \pr(Y_{0}<b/2)>0,
\end{align}
where the second condition holds for every $b>0$ because $Y_0$ is a function of finitely many edges,
	each of which is small with positive probability.
We also take $b<a$.
Define the \emph{portal} event
\begin{align}\label{eq:portal}
	Q_m:=\{Y_{m-1}<b/2\}\cap\{Y_m<b/2\}\cap Q_m^\times,\quad
	Q_m^\times:=\bigcap_{e:\,d_e\le B,\ e\text{ strictly crosses }mB}\{t_e>b\}.
\end{align}
The events $(Q_m^\times)_{m\in\dZ}$ are i.i.d.:
	no edge of length at most $B$ strictly crosses two distinct cell boundaries.
They are independent of $(Y_q)_{q\in\dZ}$, because an edge that strictly crosses a boundary lies in no cell.
Each has probability at least $\prod_{d\le B}\pr(\go\ge s)^{d-1}>0$, since $t_e>b$ is implied by $d^\ga\go_e\ge \go_e\ge s>b$.

\begin{lem}[Local portal]\label{lem:portal}
	On $Q_m$, every path that uses only edges of length at most $B$ and passes from the left of $mB$ to the right of $mB$ without visiting $mB$ can be made strictly cheaper by routing it through $mB$.
\end{lem}

\begin{proof}
	Such a path contains an edge $e=\{y,z\}$ with $y<mB<z$ and $z-y\le B$.
	Then $y\in I_{m-1}$ and $z\in I_m$, so $\Dl(y,mB)+\Dl(mB,z)\le Y_{m-1}+Y_m<b<t_e$.
	Replacing $e$ by a spread-out route from $y$ to $mB$ followed by one from $mB$ to $z$ strictly lowers the passage time.
\end{proof}

For $d>B$, call an edge $e$ of length $d$ \emph{forbidden} if
\begin{align}\label{eq:forbidden}
	t_e\le C_0(d+2B),
\end{align}
and let
\begin{align}\label{eq:pd}
	p_d:=\pr(e\text{ is forbidden})=F\big(C_0(d+2B)d^{-\ga}\big),\quad d>B.
\end{align}
Note that the number of edges of length $d$ strictly crossing a given point is $d-1$.

\begin{lem}[Long-edge summability]\label{lem:summability}
	\begin{enumeratei}
		\item If $\ga>1+2/\theta$, then for every $\eps>0$ there is $C_\eps$ with $p_d\le C_\eps d^{-(\theta-\eps)(\ga-1)}$ for all $d>B$, uniformly in $K$.
			In particular $K$ can be chosen so large that $p_d<1$ for all $d>B$ and
		\begin{align}\label{eq:K-choice}
			c_B:=\sum_{d>B}d\,p_d\le\frac18.
		\end{align}
		We fix such a $K$ from now on.
		\item If $\ga>1+3/\theta$, then in addition $\sum_{d>B}d^2p_d<\infty$,
				and there is $\gd>0$ with $\sum_{d\ge r}d\,p_d\le Cr^{-1-\gd}$ for all $r\ge1$.
		\item If Assumption~\ref{ass:rv} holds,
				then $p_d\sim C_0^{\theta}\,d^{-\theta(\ga-1)}\slv(d^{1-\ga})$ is regularly varying in $d$ with index $-\theta(\ga-1)=-(\gb+1)$.
			If moreover $\gb\in(1,2)$, then as $r\to\infty$,
		\begin{align}\label{eq:karamata}
			\sum_{d\ge r}p_d\sim\frac{C_0^{\theta}}{\gb}\,r^{-\gb}\slv(r^{1-\ga}),\
			\sum_{d\ge r}d\,p_d\sim\frac{C_0^{\theta}}{\gb-1}\,r^{1-\gb}\slv(r^{1-\ga}),\
			\sum_{B<d<r}d^2p_d\sim\frac{C_0^{\theta}}{2-\gb}\,r^{2-\gb}\slv(r^{1-\ga}).
		\end{align}
	\end{enumeratei}
\end{lem}

\begin{proof}
	The argument of $F$ in~\eqref{eq:pd} is at most $3C_0d^{1-\ga}\le C d^{1-\ga}$,
		with $C$ independent of $K$ by~\eqref{eq:a-C0}.
	Choose $\eps>0$ so small that $(\theta-\eps)(\ga-1)>2$ (resp.\ $>3$).
	The upper bound in~\eqref{eq:two-sided-F} gives $p_d\le C_\eps d^{-(\theta-\eps)(\ga-1)}$,
		whose exponent exceeds $2$ (resp.\ $3$).
	This proves (i), since the tails of a convergent series tend to $0$ and $p_d\to0$.
	For (ii), the exponent $e:=(\theta-\eps)(\ga-1)$ exceeds $3$,
		and $\sum_{d\ge r}d^{1-e}\le Cr^{2-e}$, so $\gd:=e-3$ works.
	Part (iii) follows from $F(t)=t^{\theta}\slv(t)$ and the uniform convergence theorem for slowly varying functions~\cite[Theorem~1.2.1]{BGT87},
		since $C_0(d+2B)d^{-\ga}\sim C_0d^{1-\ga}$.
	The asymptotics~\eqref{eq:karamata} are Karamata's theorem~\cite[Proposition~1.5.8, 1.5.10]{BGT87}.
\end{proof}

By~\eqref{eq:K-choice}, the expected number of forbidden edges strictly crossing a given point is at most $c_B\le1/8$,
	so a given cell boundary is crossed by no forbidden edge with probability at least $7/8$.
This is the quantitative form of~\eqref{eq:heur-density}.
\section{Hubs and regeneration}\label{sec:hubs}

\subsection{Definition and the pivotal property}
\begin{defn}[Hub]\label{def:hub}
	A cell boundary $x=mB$ is a \emph{hub} if
	\begin{enumeratei}
		\item the linear-control event $A_m^-\cap A_m^+$ occurs,
		\item the portal event $Q_m$ occurs, and
		\item no forbidden edge strictly crosses $x$.
	\end{enumeratei}
	We write $H_m$ for the event that $mB$ is a hub,
		$\Hloc:=\{m: A_m^-\cap A_m^+\cap Q_m\}$ for the set of \emph{locally good} boundaries,
		$\cD$ for the set of forbidden edges,
		and $\cU:=\{m: \text{no } e\in\cD \text{ strictly crosses } mB\}$ for the set of \emph{uncovered} boundaries,
		so that the hub set is $\Hloc\cap\cU$.
\end{defn}

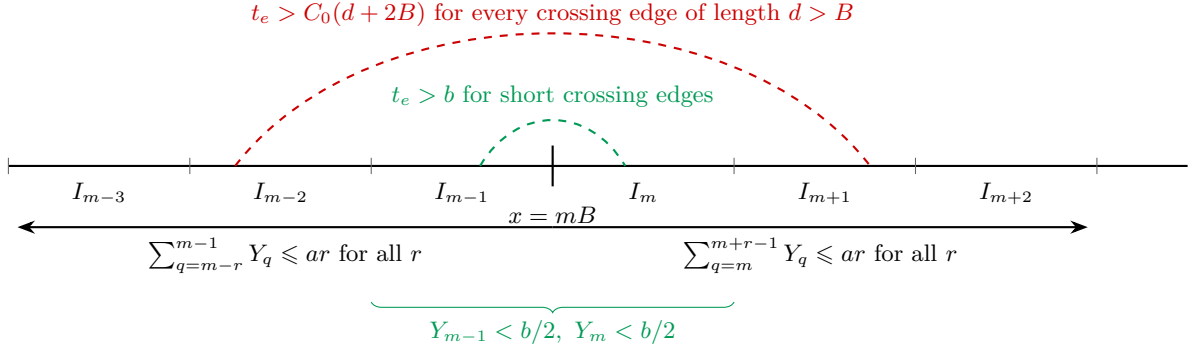
\begin{figure}[htbp]
	\centering
	\begin{tikzpicture}[x=1.2cm,y=.9cm,>=Stealth]
		\draw[thick] (0,0) -- (13,0);
		\foreach \x in {0,2,4,6,8,10,12} {\draw[black!60] (\x,-0.12) -- (\x,0.12);}
		\node[below,font=\footnotesize] at (1,-0.1) {$I_{m-3}$};
		\node[below,font=\footnotesize] at (3,-0.1) {$I_{m-2}$};
		\node[below,font=\footnotesize] at (5,-0.1) {$I_{m-1}$};
		\node[below,font=\footnotesize] at (7,-0.1) {$I_{m}$};
		\node[below,font=\footnotesize] at (9,-0.1) {$I_{m+1}$};
		\node[below,font=\footnotesize] at (11,-0.1) {$I_{m+2}$};
		\draw[thick] (6,-0.3) -- (6,0.3);
		\node[below,font=\footnotesize] at (6,-0.45) {$x=mB$};
		\draw[thick,-{Stealth}] (6,-0.9) -- (0.1,-0.9) node[midway,below,font=\footnotesize] {$\sum_{q=m-r}^{m-1}Y_q\le ar$ for all $r$};
		\draw[thick,-{Stealth}] (6,-0.9) -- (11.9,-0.9) node[midway,below,font=\footnotesize] {$\sum_{q=m}^{m+r-1}Y_q\le ar$ for all $r$};
		\draw[ForestGreen,decorate,decoration={brace,amplitude=4pt,mirror}] (4,-2.0) -- (8,-2.0) node[midway,below=3pt,font=\footnotesize] {$Y_{m-1}<b/2,\ Y_m<b/2$};
		\draw[thick,dashed,ForestGreen] (5.2,0) .. controls (5.6,0.9) and (6.4,0.9) .. (6.8,0);
		\node[ForestGreen,font=\footnotesize,above] at (6,0.72) {$t_e>b$ for short crossing edges};
		\draw[red!80!black,thick,dashed] (2.5,0) .. controls (4,2.6) and (8,2.6) .. (9.5,0);
		\node[red!80!black,font=\footnotesize,above] at (6,1.9) {$t_e>C_0(d+2B)$ for every crossing edge of length $d>B$};
	\end{tikzpicture}
	\caption{The three ingredients of a hub at $x=mB$.
		Linear control bounds the spread-out metric on both sides.
		The portal (green) makes the two adjacent cells cheap and every short crossing edge expensive.
		No forbidden edge (red) crosses $x$.}\label{fig:hub}
\end{figure}

The use of the spread-out envelope in (i) and (ii) is the essential change from the nearest-neighbor hub of Appendix~\ref{app:nn}.
It makes the construction available even when $\E\go=\infty$.
Observe that $\Hloc$ is a function of the edges of length at most $B$,
	whereas $\cD$ and $\cU$ are functions of the edges of length greater than $B$.
Hence $\Hloc$ and $\cU$ are independent.
Moreover $\cD$ is a Bernoulli field: each edge of length $d>B$ belongs to $\cD$ independently with probability $p_d$.

\begin{lem}[Pivotal property]\label{lem:pivotal}
	Let $x$ be a hub.
	\begin{enumeratei}
		\item\label{pivot:i} If $u<x<v$, every path from $u$ to $v$ avoiding $x$ can be made strictly cheaper by routing it through $x$.
			In particular every geodesic from $u$ to $v$ visits $x$.
		\item\label{pivot:ii} A path from $x$ to a point $v>x$ that leaves $[x,\infty)$ can be made strictly cheaper by one that stays in $[x,v']$ for a hub $v'\ge v$, and symmetrically to the left.
		\item\label{pivot:iii} If $x_1<\cdots<x_k$ are hubs, then for $u\le x_1$ and $v\ge x_k$ with $u<v$,
		\begin{align}\label{eq:additivity-hubs}
			T(u,v)=T(u,x_1)+\sum_{j=1}^{k-1}T(x_j,x_{j+1})+T(x_k,v).
		\end{align}
	\end{enumeratei}
	Almost surely, for all $u<v$ the infimum in~\eqref{eq:T} is attained by a path with finitely many edges.
\end{lem}

\begin{proof}
\begin{enumeratei}
\item[\ref{pivot:i}]  Let $\pi$ be a path from $u$ to $v$ avoiding $x=mB$.
	It contains an edge $e=\{y,z\}$ with $y<x<z$.
	If $z-y\le B$, Lemma~\ref{lem:portal} applies.
	If $d:=z-y>B$, then $e$ is not forbidden, so $t_e>C_0(d+2B)$.
	On $A_m^-\cap A_m^+$, inequality~\eqref{eq:linear-local} provides an $\ell$-spread-out walk from $y$ to $z$ through $x$ of passage time at most $C_0(x-y+B)+C_0(z-x+B)=C_0(d+2B)<t_e$.
	Replacing $e$ by this walk, and deleting loops, strictly decreases the passage time.

	\item[\ref{pivot:ii}] If a path from $x$ to $v>x$ visits a vertex $y<x$,
		consider its last visit to $y'<x$ and the following edge $\{y',z\}$ with $z\ge x$.
	If $z=x$, delete the loop at $x$.
	Otherwise replace the edge by a walk through $x$ as in (i) and delete the loop.
	Both operations strictly decrease the cost. An excursion to the right of a hub $v'\ge v$ can be removed in the same way, by the symmetric argument at $v'$. 

	\item[\ref{pivot:iii}] By (i), the infimum defining $T(u,v)$ may be restricted to paths visiting $x_1,\ldots,x_k$ in this order,
		and by (ii) the segment between $x_j$ and $x_{j+1}$ may be restricted to $[x_j,x_{j+1}]$.
	This gives $\ge$ in~\eqref{eq:additivity-hubs}. Concatenating paths gives $\le$.
	\end{enumeratei}
	For the last claim, Lemma~\ref{lem:cut}.(i) below and the ergodicity of the environment show that a.s.~there are hubs $x_-<u$ and $x_+>v$.
	A path from $u$ to $v$ that visits a vertex to the left of $x_-$ contains an edge $\{z',y''\}$ with $z'\ge x_->y''$ preceding its first visit to $(-\infty,x_-)$ and an edge $\{y',z\}$ with $y'<x_-\le z$ following its last visit to $(-\infty,x_-)$.
	Replacing each of these two edges, when it strictly crosses $x_-$,
		by a walk through $x_-$ as in (i), and then deleting the resulting loop at $x_-$, strictly decreases the cost.
	The same applies at $x_+$.
	Hence the infimum may be restricted to simple paths inside $[x_-,x_+]$,
		of which there are finitely many, and it is attained.
\end{proof}

\subsection{The cut structure of the hub event}
For $x\in\dZ$ write $\cE^-_x$, $\cE^+_x$ and $\cE^\times_x$ for the sets of edges with both endpoints in $(-\infty,x]$,
	with both endpoints in $[x,\infty)$, and strictly crossing $x$, respectively.
These three sets partition the edge set.
We write $\cF(\cE)$ for the $\gs$-field generated by the edge variables in $\cE$.
The following lemma is the reason hubs regenerate.

\begin{lem}[Cut structure]\label{lem:cut}
	\begin{enumeratei}
		\item $H_m=H_m^-\cap H_m^\times\cap H_m^+$, where
		\begin{align*}
			H_m^- :=A_m^-\cap\{Y_{m-1}<b/2\} &\in\cF(\cE^-_{mB}),\quad
			H_m^+ :=A_m^+\cap\{Y_{m}<b/2\}\in\cF(\cE^+_{mB}),\text{ and }\\
			H_m^\times &:=Q_m^\times\cap\{m\in\cU\}\in\cF(\cE^\times_{mB}).
		\end{align*}
		The three events are independent,
		\[
			\pr(H_m^-)=\pr(H_m^+)>0\text{ and }\pr(H_m^\times)\ge\frac78\prod_{d\le B}\pr(\go\ge s)^{d-1}>0.
		\]
		\item (Localization) For $m<m'$ there is an event $H_{m'}^{(m)}\in\cF(\cE^+_{mB})$ with $H_m\cap H_{m'}=H_m\cap H_{m'}^{(m)}$,
				and symmetrically an event in $\cF(\cE^-_{m'B})$ playing the same role for $m<m'$ on $H_{m'}$.
		\item (Two-sided localization) For $m<m''<m'$ there is an event $H_{m''}^{(m,m')}\in\cF(\cE^+_{mB}\cap\cE^-_{m'B})$ with $H_m\cap H_{m'}\cap H_{m''}=H_m\cap H_{m'}\cap H_{m''}^{(m,m')}$.
			Moreover, there is an event $F^{(m,m')}\in\cF(\cE^+_{mB}\cap\cE^-_{m'B})$ with
		\begin{align}\label{eq:two-hubs}
			H_m\cap H_{m'}=H_m^-\cap H_m^\times\cap F^{(m,m')}\cap H_{m'}^\times\cap H_{m'}^+ .
		\end{align}
	\end{enumeratei}
\end{lem}

\begin{proof}
	\begin{enumeratei}
	\item The decomposition is immediate from the definitions, since $Y_q$ is a function of the edges inside the cell $I_q$.
	Independence holds because the three edge sets are disjoint and the edge variables are independent.
	$H_m^-$ is the intersection of two events that are decreasing functions of the edge weights,
		each of positive probability (the first because a random walk with negative drift stays below its starting point forever with positive probability,
		the second by~\eqref{eq:b-s}), so $\pr(H_m^-)>0$ by Harris' inequality for product measures.
	$H_m^\times$ is an intersection of independent conditions on the crossing edges,
		and the union bound with~\eqref{eq:K-choice} gives $\pr(m\in\cU)\ge1-\sum_{d>B}(d-1)p_d\ge7/8$.

	\item Fix $m<m'$ and work on $H_m$.
	On $A_m^-$, for $r>m'-m$,
	\begin{align*}
		\sum_{q=m'-r}^{m'-1}Y_q=\sum_{q=m'-r}^{m-1}Y_q+\sum_{q=m}^{m'-1}Y_q\le a(r-(m'-m))+\sum_{q=m}^{m'-1}Y_q,
	\end{align*}
	so $A_{m'}^-$ is implied by its finite part $\bigcap_{r\le m'-m}\{\sum_{q=m'-r}^{m'-1}Y_q\le ar\}$, which lies in $\cF(\cE^+_{mB})$.
	The events $A_{m'}^+$, $\{Y_{m'-1}<b/2\}$,
		$\{Y_{m'}<b/2\}$ and $Q_{m'}^\times$ lie in $\cF(\cE^+_{mB})$ because the cells $I_{m'-1},I_{m'}$ and every edge of length at most $B$ strictly crossing $m'B$ are contained in $[mB,\infty)$.
	Finally, on $\{m\in\cU\}$ a forbidden edge crossing $m'B$ cannot cross $mB$,
		so $\{m'\in\cU\}$ coincides on $H_m$ with the event that no forbidden edge in $\cE^+_{mB}$ crosses $m'B$.
	The symmetric statement is proved in the same way.

	\item Fix $m<m''<m'$ and work on $H_m\cap H_{m'}$.
	By the computation in (ii), $A^-_{m''}$ is implied by its finite part $\bigcap_{r\le m''-m}\{\sum_{q=m''-r}^{m''-1}Y_q\le ar\}$,
		and symmetrically, on $A^+_{m'}$, for $r>m'-m''$ one has $$\sum_{q=m''}^{m''+r-1}Y_q\le\sum_{q=m''}^{m'-1}Y_q+a(r-(m'-m'')),$$
		so $A^+_{m''}$ is implied by its finite part $\bigcap_{r\le m'-m''}\{\sum_{q=m''}^{m''+r-1}Y_q\le ar\}$.
	Both finite parts, the envelope conditions $Y_{m''-1},Y_{m''}<b/2$ and $Q^\times_{m''}$ only involve edges inside $[mB,m'B]$.
	On $\{m\in\cU\}\cap\{m'\in\cU\}$ a forbidden edge crossing $m''B$ crosses neither $mB$ nor $m'B$,
		hence lies in $[mB,m'B]$, so $\{m''\in\cU\}$ coincides with the event that no forbidden edge inside $[mB,m'B]$ crosses $m''B$.
	This gives $H_{m''}^{(m,m')}$.
	For~\eqref{eq:two-hubs}, write $$
	H_m\cap H_{m'}=H_m^-\cap H_m^\times\cap(H_m^+\cap H_{m'}^-)\cap H_{m'}^\times\cap H_{m'}^+.
	$$
	\end{enumeratei}
	\noindent On $A^+_{m'}$ the event $A_m^+$ is implied by its finite part $\bigcap_{r\le m'-m}\{\sum_{q=m}^{m+r-1}Y_q\le ar\}$,
		and on $A^-_m$ the event $A^-_{m'}$ is implied by its finite part, by the same computation.
	Hence $F^{(m,m')}:=$ (finite part of $A_m^+$) $\cap\{Y_m<b/2\}\cap$ (finite part of $A_{m'}^-$) $\cap\{Y_{m'-1}<b/2\}$, which lies in $\cF(\cE^+_{mB}\cap\cE^-_{m'B})$, satisfies~\eqref{eq:two-hubs}.
\end{proof}

\subsection{Locally good boundaries}
Before treating the hubs we record the renewal structure of $\Hloc=\{m:A_m^-\cap A_m^+\cap Q_m\}$,
	which depends only on the i.i.d.~sequence $(Y_q)$ and the independent i.i.d.~marks $(Q_q^\times)$.
Let $\pr^{\mathrm{loc}}:=\pr(\cdot\mid0\in\Hloc)$ and let $\gL:=\min\{m\ge1:m\in\Hloc\}$ be the gap.

\begin{lem}[Locally good boundaries]\label{lem:loc}
	$\Hloc$ is a stationary renewal set on $\dZ$ of positive intensity.
	Under $\pr^{\mathrm{loc}}$, $\E(\gL^{r})<\infty$ for every $r<(p-2)/2$.
	Consequently, for every such $r$,
	\begin{align}\label{eq:loc-gap}
		\pr\big(\Hloc\cap(0,u]=\eset\big)\le Cu^{1-r},\quad u\ge1.
	\end{align}
	Here $\Hloc$ is indexed by cell numbers.
	The same bound holds, with a different constant,
		for the set $B\Hloc$ of locally good boundaries and intervals of $u$ vertices,
		and we use both forms without further comment.
\end{lem}

\begin{proof}
	The cut factorization and localization of Lemma~\ref{lem:cut}, with the long-edge condition omitted,
		imply the renewal property by the argument of Proposition~\ref{prop:regeneration}.(i).
	Positivity follows from Lemma~\ref{lem:cut}.(i).
	This part does not use any gap moment bound.

	For an integer $h\ge1$, let $J_m^{(h)}$ be the event obtained from $\{m\in\Hloc\}$ by imposing the two linear-control inequalities only for $1\le r\le h$,
		retaining the entire portal condition.
	Define $\pi_{\rm loc}:=\pr(0\in\Hloc)>0$.
	Then $\pr(J_m^{(h)})\ge\pi_{\rm loc}$.
	Since $a=2\E Y_0$ and $\E Y_0^p<\infty$,
		the polynomial large-deviation bound for sums of independent variables~\cite{Nag79} gives
	\begin{align*}
		\pr\Big(\sum_{q=1}^rY_q>ar\Big)\le C r^{1-p},\quad r\ge1.
	\end{align*}
	A union bound over the omitted inequalities therefore yields
	\begin{align*}
		\pr\big(J_m^{(h)}\setminus\{m\in\Hloc\}\big)\le C\sum_{r>h}r^{1-p}\le Ch^{2-p}.
	\end{align*}

	The event $J_m^{(h)}$ depends only on $Y_{m-h},\ldots,Y_{m+h-1}$ and the independent mark $Q_m^\times$.
	Choose $k\asymp u/h$ deterministic trial points in $(0,u]$ at mutual distances at least $2h+1$.
	Their truncated events are independent.
	If none is locally good, either all truncated events fail or at least one is a false positive.
	Consequently
	\begin{align*}
		\pr\big(\Hloc\cap(0,u]=\eset\big)
		\le (1-\pi_{\rm loc})^k+Ckh^{2-p}
		\le C\exp(-cu/h)+Cu h^{1-p}.
	\end{align*}

	For any $1<r<p-1$, choose $r<r'<p-1$ and $\kappa\in(r'/(p-1),1)$, and take $h=\floor{u^\kappa}$.
	The last display is $O(u^{1-r'})$.
	The stationary--Palm identity for the cell-index renewal process now gives $\E^{\rm loc}\gL^r<\infty$:
		for $r>1$, finiteness is equivalent to summability of $u^{r-2}\pr(\Hloc\cap(0,u]=\eset)$.
	Moments of order at most one follow from the finite Palm mean.
	This proves, in particular, all the claimed moments $r<(p-2)/2$ and~\eqref{eq:loc-gap}.
\end{proof}

\subsection{Regeneration}
Let $U$ be uniform on $\{0,\ldots,B-1\}$ and independent of the edge environment.
Repeat the preceding cell construction on the translated grid $U+B\dZ$ (equivalently,
	replace each $I_m$ by $[U+mB,U+(m+1)B]\cap\dZ$).
This harmless augmentation makes the hub process stationary under every integer translation, not only translations by $B$.

Let $\cH:=\{U+mB:H_m\}$ be the hub set.
It is a stationary ergodic point process on $\dZ$.
Its intensity is $\pr(H_0)/B>0$ by Lemma~\ref{lem:cut}.(i),
	so a.s.~$\cH$ is unbounded in both directions, and we may enumerate it as $\cdots<\hub_{-1}<\hub_0\le0<\hub_1<\cdots$.
The Palm law $\pr^0$ is the law conditioned on a hub at the origin. Equivalently,
	condition on $\{U=0\}$ and $H_0$, and then use the origin as the distinguished hub.
Under $\pr^0$ we write $0=\hub_0<\hub_1<\hub_2<\cdots$ for the hubs in $[0,\infty)$,
	$L_i:=\hub_{i+1}-\hub_i$ and $G_i:=T(\hub_i,\hub_{i+1})$,
	and we let $\cB_i:=(\go_e:e\subset[\hub_i,\hub_{i+1}])$ be the block environment.

\begin{prop}[Regeneration]\label{prop:regeneration}
	Assume Assumption~\ref{ass:law} and $\ga>1+2/\theta$.
	\begin{enumeratei}
		\item Under $\pr^0$ the marked blocks $(L_i,\cB_i)_{i\ge0}$ are i.i.d.,
				and $(L_i,\cB_i)_{i\ge0}$ is independent of $\cF(\cE_0^-)$.
			The stationary hub process is a stationary renewal process whose gap law is that of $L:=L_0$ under $\pr^0$, with
		\begin{align}\label{eq:EL}
			\E^0L={B}/{\pr(H_0)}<\infty.
		\end{align}
		In particular, in the stationary enumeration the pairs $(L_i,G_i)$, $i\ne0$,
			are i.i.d.~with the $\pr^0$-law of $(L_0,G_0)$.
		\item $G_i$ is a function of $\cB_i$, and
		\begin{align}\label{eq:G-bound}
			0\le G_i\le C_0L_i.
		\end{align}
		In particular $\E^0G<\infty$.
		More generally, for every hub $x$ and $u<x<v$, $T(u,x)\le C_0(x-u+B)$ and $T(x,v)\le C_0(v-x+B)$.
		\item If $\ga>1+3/\theta$, then $\E^0(L^2)<\infty$ and hence $\E^0(G^2)<\infty$.
	\end{enumeratei}
\end{prop}

\begin{proof}
	(i) By Lemma~\ref{lem:cut}.(i), $\pr^0=\pr(\cdot\mid H_0^-)\otimes\pr(\cdot\mid H_0^\times)\otimes\pr(\cdot\mid H_0^+)$ on $\cF(\cE_0^-)\otimes\cF(\cE_0^\times)\otimes\cF(\cE_0^+)$.
	By Lemma~\ref{lem:cut}.(ii), on $H_0$ the hubs in $(0,\infty)$ are determined by the edges in $\cE_0^+$.
	Hence $(L_i,\cB_i)_{i\ge0}$ is a function of the edges in $\cE_0^+$ and is independent of $\cF(\cE^-_0)$.

	Fix $m'\ge1$ and consider the event $H_0\cap\{\hub_1=m'B\}=H_0\cap H_{m'}\cap\{H_{m''}\text{ fails for }0<m''<m'\}$.
	By Lemma~\ref{lem:cut}.(iii), on $H_0\cap H_{m'}$ each $H_{m''}$ with $0<m''<m'$ coincides with an event in $\cF(\cE_0^+\cap\cE^-_{m'B})$,
		and $H_0\cap H_{m'}=H_0^-\cap H_0^\times\cap F^{(0,m')}\cap H_{m'}^\times\cap H_{m'}^+$ with $F^{(0,m')}\in\cF(\cE_0^+\cap\cE^-_{m'B})$.
	Hence
	\begin{align*}
		H_0\cap\{\hub_1=m'B\}=H_0^-\cap H_0^\times\cap F_{m'}\cap H_{m'}^\times\cap H_{m'}^+,\quad F_{m'}\in\cF\big(\cE_0^+\cap\cE^-_{m'B}\big),
	\end{align*}
	where $F_{m'}$ is the intersection of $F^{(0,m')}$ with the complements of the localized intermediate hub events.
	The block environment $\cB_0$ consists of the edges in $\cE_0^+\cap\cE^-_{m'B}$,
		and the edges in $\cE^+_{m'B}$ determine the later blocks.
	Let $A\in\cF(\cE_0^+\cap\cE^-_{m'B})$ and $A'\in\cF(\cE^+_{m'B})$,
		and write $\gt_{m'B}A'$ for the translate of $A'$ by $m'B$.

	The crossing sets $\cE_0^\times$ and $\cE_{m'B}^\times$ are \emph{not} disjoint:
		an edge with one endpoint left of $0$ and the other right of $m'B$ belongs to both.
	Their union is, however, disjoint from $\cE_0^-$, $\cE_0^+\cap\cE^-_{m'B}$ and $\cE^+_{m'B}$.
	Thus the two crossing conditions must be kept as one joint factor.
	Independence of the remaining edge families and translation invariance give
	\begin{align*}
		\pr\big(H_0,\hub_1=m'B,A,\gt_{m'B}A'\big)&=\pr(H_0^-)\,\pr(F_{m'}\cap A)\,\pr(H_0^\times\cap H_{m'}^\times)\,\pr(H_{m'}^+\cap\gt_{m'B}A')\\
		&=\pr\big(H_0,\hub_1=m'B,A\big)\,\frac{\pr(H_0^+\cap A')}{\pr(H_0^+)}.
	\end{align*}

	Dividing by $\pr(H_0)$ shows that under $\pr^0$,
		conditionally on $\hub_1$ and on the block $(L_0,\cB_0)$, the environment to the right of $\hub_1$,
		translated to the origin, has the law $\pr(\cdot\mid H_0^+)$ on $\cF(\cE_0^+)$,
		which is also the $\pr^0$-law of the environment in $\cE_0^+$.
	In particular it is independent of $(L_0,\cB_0)$, and $(L_1,\cB_1)$ has the $\pr^0$-law of $(L_0,\cB_0)$.
	Iterating gives the i.i.d.~property.

	A stationary point process whose Palm gaps are i.i.d.~is a stationary renewal process~\cite{Thorisson00},
		and its intensity is the reciprocal of the Palm mean gap, which gives~\eqref{eq:EL}.
	The last claim in (i) is the standard description of a stationary renewal process:
		the gaps not containing $0$ form two independent i.i.d.~sequences with the Palm gap law,
		independent of the gap containing $0$.

	(ii) By Lemma~\ref{lem:pivotal}.\ref{pivot:ii} a geodesic from $\hub_i$ to $\hub_{i+1}$ stays in $[\hub_i,\hub_{i+1}]$,
		so $G_i$ is a function of $\cB_i$.
	Since $\hub_i$ and $\hub_{i+1}$ are cell boundaries,
		the linear-control event $A^+$ at $\hub_i$ and Lemma~\ref{lem:amplify}.(iii) give $G_i\le\Dl(\hub_i,\hub_{i+1})\le a\,L_i/B=C_0L_i$.
	The general statement is~\eqref{eq:linear-local}.

	(iii) is Lemma~\ref{lem:second-moment} below, together with~\eqref{eq:G-bound}.
\end{proof}

\subsection{Second moment of the block length}
The stationary--Palm relation for a renewal process on $\dZ$ reads
\begin{align}\label{eq:palm-stat}
	\pr(\hub_1\ge v)=\pr\big(\cH\cap(0,v)=\eset\big)=\frac{\E^0\big((L-v+1)_+\big)}{\E^0L},\quad v\ge1,
\end{align}
so that $\E^0(L^2)<\infty$ if and only if $\sum_{u\ge1}\pr(\cH\cap(0,u]=\eset)<\infty$.
Formula~\eqref{eq:palm-stat} is exact for every integer $v\ge1$.
The next lemma is the only place where $\sum_dd^2p_d<\infty$ is used.

\begin{lem}[Gaussian regime]\label{lem:second-moment}
	If $\ga>1+3/\theta$, there is $\gd>0$ such that $\pr(\cH\cap(0,u]=\eset)\le Cu^{-1-\gd}$ for all $u\ge1$.
	Consequently $\E^0(L^2)<\infty$.
\end{lem}

\begin{proof}
	Let $\gd>0$ be as in Lemma~\ref{lem:summability}.(ii), so that $\sum_{d\ge D}d\,p_d\le CD^{-1-\gd}$.
	Take $\gd=\delta_F$, $\kappa=\kappa_G$ and $r=3$, with the choices in Section~\ref{ssec:amplify}. Then $\kappa r>2$, $\kappa(2+\gd)>2$ and $r<(p-2)/2$.
	Define $D:=\floor{u^\kappa}$.
	Let $E_u$ be the event that every interval of length $D$ contained in $(0,u]$ contains a point of $\Hloc$.
	By stationarity, a union bound over a deterministic grid of intervals of length $D/2$ and~\eqref{eq:loc-gap},
	\begin{align*}
		\pr(E_u^c)\le(u/D+1)\,CD^{1-r}\le Cu^{1-\kappa r}\le Cu^{-1-\gd'}
	\end{align*}
	for some $\gd'>0$, since $\kappa r>2$.

	On $E_u$ choose points $m_1<\cdots<m_{K'}$ of $\Hloc\cap(0,u]$ greedily:
		$m_1$ is the first point of $\Hloc$ in $(0,D]$ and $m_{i+1}$ the first point of $\Hloc$ in $(m_i+D,m_i+2D]$, as long as $m_i+2D\le u$.
	Then consecutive points are at distance more than $D$ apart and $K'\ge u/(2D)-1\ge cu^{1-\kappa}$.
	The points are functions of $\Hloc$, which is independent of the forbidden field.
	If no hub lies in $(0,u]$, then every $m_i$ is strictly crossed by a forbidden edge.
	Either some $m_i$ is crossed by a forbidden edge of length at least $D$,
		an event of probability at most $K'\sum_{d\ge D}(d-1)p_d\le CK' D^{-1-\gd}$,
		or every $m_i$ is crossed by a forbidden edge of length less than $D$.
	The latter events are independent over $i$,
		because a forbidden edge of length less than $D$ crosses at most one of the points,
		and each has probability at most $c_B\le1/8$ by~\eqref{eq:K-choice}.
	Hence
	\begin{align*}
		\pr\big(\cH\cap(0,u]=\eset\big)\le\pr(E_u^c)+Cu^{1-\kappa}u^{-\kappa(1+\gd)}+8^{-cu^{1-\kappa}}\le Cu^{-1-\gd''},
	\end{align*}
	with $\gd^{\prime\prime}=\delta_G>0$, since both polynomial exponents exceed $1+2\delta_G$.
	The last claim follows from~\eqref{eq:palm-stat} and $\sum_v\E^0((L-v+1)_+)=\E^0(L(L+1))/2$.
\end{proof}

\begin{rem}\label{rem:stable-tail-forward}
	In the stable regime $1+2/\theta<\ga<1+3/\theta$ the same proof gives $\pr(\cH\cap(0,u]=\eset)\le Cu^{1-\kappa\gb}$ for any $\kappa<1$,
		which only shows $\E^0(L^q)<\infty$ for $q<\gb$.
	The sharper upper bound on the scale $u^{-\gb}\slv(u^{1-\ga})$ is proved using the first-big-edge decomposition in Lemma~\ref{lem:L-tail}. That lemma also supplies a matching lower bound up to constants for the original Palm gap.
\end{rem}

\subsection{Law of large numbers in the linear-growth regime}
\begin{prop}[Law of large numbers]\label{prop:lln}
	Suppose Assumption~\ref{ass:law} holds and $\ga>1+2/\theta$.
	Then $\mu:=\E^0G/\E^0L\in(0,\infty)$, $T_n/n\to\mu$ a.s.~and in $L^1$, and $\E T_n/n\to\mu$.
\end{prop}

\begin{proof}
	By Proposition~\ref{prop:regeneration}, $0<\E^0L<\infty$ and $\E^0G\le C_0\E^0L<\infty$.
	Positivity of $\E^0G$ follows from $\pr(\go>0)=1$ and the fact that $G$ is the sum of the edge times of a path with at least one edge.
	Let $\Nn=\max\{i:\hub_i\le n\}$ in the stationary enumeration.
	By~\eqref{eq:additivity-hubs}, on $\{\hub_1\le n\}$,
	\begin{align}\label{eq:decomp}
		T_n=T(0,\hub_1)+\sum_{i=1}^{\Nn-1}G_i+T(\hub_{\Nn},n).
	\end{align}
	The renewal strong law gives $\Nn/n\to1/\E^0L$ a.s.,
	and the strong law for the i.i.d.~rewards gives $\frac1n\sum_{i<\Nn}G_i\to\E^0G/\E^0L$ a.s.
	By Proposition~\ref{prop:regeneration}.(ii) the two residual terms are bounded by $C_0(\hub_1+B)$ and $C_0(n-\hub_{\Nn}+B)$.
	The first is a.s.~finite.
	The second is at most $C_0(\max_{1\le i\le\Nn}L_i+B)=o(n)$ a.s., because $\E^0L<\infty$ and $\Nn=O(n)$.
	This proves a.s.~convergence.
	For $L^1$ convergence on the randomly shifted grid, we use $0\le T_n\le\Dl(0,n)\le W_n:=\sum_{m=-1}^{\ceil{n/B}}Y_m$. These cells cover $[0,n]$ for every grid phase. Conditional on the phase their envelopes are i.i.d., and $\E W_n^p\le C(n+1)^p$, so $(T_n/n)_{n\ge1}$ is uniformly integrable.
	Hence $\E T_n/n\to\mu$.
\end{proof}

The construction proves positivity when $\ga>1+2/\theta$.
The next proposition directly proves the complementary zero-speed assertion under the same edge-law assumptions.
\begin{prop}[Zero speed below the linear-growth threshold]\label{prop:zero-speed}
Assume Assumption~\ref{ass:law}, and suppose that $0<\ga<1+2/\theta$.
Then $T(0,n)/n\to0$ almost surely.
\end{prop}
\begin{proof}
Choose $p_0>1$, an integer $\ell_0\ge2$ with $\zeta\ell_0>p_0$, and cells of length $B_0=4\ell_0$ on the unshifted grid.
The route argument of Lemma~\ref{lem:amplify}, which does not use~\eqref{eq:linear-regime},
gives i.i.d.~cell envelopes $V_j$ with $\E V_0^{p_0}<\infty$.
Fix $\delta\in(0,1/4)$ and $\eps>0$, and consider all edges joining
$A_n=[0,\floor{\delta n}]\cap\dZ$ to
$C_n=[\ceil{(1-\delta)n},n]\cap\dZ$.
Their lengths are at most $n$. Independence of these distinct edges implies
\[
 \pr(\text{none has }t_e\le\eps n)
 \le (1-F(\eps n^{1-\ga}))^{|A_n||C_n|}.
\]
If $\ga>1$, choose $\eps_F>0$ with $(\theta+\eps_F)(\ga-1)<2$.
The lower bound in~\eqref{eq:two-sided-F} makes the last probability at most
$\exp(-c_{\delta,\eps} n^{2-(\theta+\eps_F)(\ga-1)})$ for large $n$.
If $\ga\le1$, use $F(\eps n^{1-\ga})\ge F(\eps)>0$ instead.
In either case the probabilities are summable. By Borel--Cantelli, almost surely
there is such a fast edge $\{u_n,v_n\}$ for every sufficiently large $n$.

Route from $0$ to $u_n$ inside the left window's cells, cross this edge, and route from $v_n$ to $n$ inside the right window's cells.
The strong law for $(V_j)$ bounds the two endpoint-route costs, divided by $n$,
by $\delta\E V_0/B_0+o(1)$ each, uniformly over the chosen endpoints.
For the right window this follows by subtracting the two partial sums up to
$n/B_0$ and $(1-\delta)n/B_0$; the finitely many boundary cells contribute $o(n)$.
Hence $\limsup_n T(0,n)/n\le 2\delta\E V_0/B_0+\eps$ a.s.
Intersect the probability-one events over rational $\delta,\eps$ and let both decrease to zero.
\end{proof}

For $y>B$, write $\cD_{\ge y}$ and $\cD_{<y}$ for forbidden edges of length at least $y$ and less than $y$, respectively.
Define two comparison metrics $T_-^y,T_+^y$ by retaining the original times when $d_e<y$ and assigning, when $d_e\ge y$,
\[
	t_{e,-}^y:=C_0(d_e+2B),\quad t_{e,+}^y:=+\infty.
\]
Both metrics depend only on the edge variables of length less than $y$.
Their time constants are denoted by $\mu_-^y,\mu_+^y$.
The original hubs remain passage-time cut points for both metrics,
	since deterministic long bridges cost at least as much as the envelope bypass bound.
The same cut and renewal arguments apply:
	the short weights remain independent with translation-invariant, length-dependent laws,
	while the other times are deterministic.

\begin{lem}\label{lem:basic}
	\begin{enumeratei}
		\item $\mu\le\E Y_0/B=C_0/2<C_0$.
		\item $0<\mu_-^y\le\mu_+^y\le C_0/2$ and $ \mu_-^y\ge\mu-C\sum_{d\ge y}d^{-7}. $
		\item $\mu\le\mu_+^y\le\mu+C\sum_{d\ge y}(dp_d+d^{-7})$.
	\end{enumeratei}
	In particular both comparison time constants tend to $\mu$ as $y\to\infty$.
\end{lem}

\begin{proof}
	We work on the unshifted cell grid under the original product law of the edge variables.
	All expectations and independence statements in this proof refer to this law.
	We first justify working with restricted paths.
	Take $n\in B\dN$ and let $V(0,n)$ be any of the three metrics restricted to $[0,n]$.
	Its time constant agrees with the unrestricted one.
	Indeed, eventually there are original hubs between 0 and $n$.
	Join 0 to the first such hub by a cell route, use the optimal block paths up to the last hub,
		and finish with another cell route.
	The first cost is a fixed finite random variable.
	The last is at most $C_0$ times the distance to the last hub; that distance is $o(n)$ almost surely.
	The block paths lie inside $[0,n]$ by confinement at both hubs.
	This proves agreement of the almost-sure limits.
	For each metric the common cell routes give
	$
		V(0,n)\le\sum_{0\le m< n/B}Y_m.
	$
	The finite $p$-th moment of $Y_0$ gives uniform integrability after division by $n$,
		hence agreement also of the limits of expectations.
	The almost-sure law of large numbers for the cell envelopes gives separately
	\[
		\max\{\mu,\mu_-^y,\mu_+^y\}\le\E Y_0/B=C_0/2.
	\]
	Existence and positivity for each metric follow from the same regenerative LLN proof as Proposition~\ref{prop:lln},
		using the original hubs and their finite mean gaps.
	Edges spanning an original hub are dominated also in the two modified metrics.

	For an edge $e=\{u,v\}$ of length $d>B$ in $[0,n]$,
		let $W_e$ be the sum of the envelopes of the cells intersecting the interior of $[u,v]$.
	It bounds the cost of a cell route inside $[0,n]$, is independent of $\go_e$, and has
	\[
		\E W_e\le C_0(d+2B)/2,\quad \E W_e^p\le Cd^p.
	\]
	The large-deviation bound for sums of independent variables~\cite{Nag79},
		followed by H\"older's inequality, gives for $p\ge10$
	\begin{align}\label{eq:envelope-excess}
		\E\big(W_e;W_e\ge C_0(d+2B)\big)
		\le Cd\,\big(d^{1-p}\big)^{1-1/p}\le Cd^{-7}.
	\end{align}
	All constants may depend on the fixed cell parameters.

	Let $\pi$ be an optimal restricted path in the original metric.
	Each of its edges satisfies $t_e\le W_e$, since otherwise it could be replaced by the cell route.
	Replacing edges with $d_e\ge y$ therefore yields
	\begin{align}\label{eq:truncation-bound}
		V_+^y(0,n)\le V(0,n)+
		\sum_{\substack{e\subset[0,n],\ d_e\ge y}}W_e\1_{\{t_e\le W_e\}}.
	\end{align}
	On $\{W_e\le C_0(d+2B)\}$ the indicator implies that $e$ is forbidden, and the expectation is bounded by $Cd p_d$.
	On its complement use~\eqref{eq:envelope-excess}.
	There are at most $n+1$ left endpoints for each length.
	Taking expectations, dividing by $n$, and passing to the limit proves (iii).

	Conversely take an optimal restricted path for $T_-^y$.
	Replace every artificial long edge by its cell route in the original metric.
	The increase is non-positive when $W_e<C_0(d+2B)$ and otherwise is at most $W_e\1_{\{W_e\ge C_0(d+2B)\}}$.
	The expected increase per possible edge is therefore at most the left-hand side of~\eqref{eq:envelope-excess}.
	Summation and passage to the limit give $\mu\le\mu_-^y+C\sum_{d\ge y}d^{-7}$.
	Monotonicity gives $\mu_-^y\le\mu_+^y$ and $\mu\le\mu_+^y$.
	This proves (i)--(ii).
\end{proof}

\section{Gaussian fluctuations: proof of Theorem~\ref{thm:main}}\label{sec:clt}

Throughout this section, we assume~\eqref{eq:alpha-condition},
	$\gd>0$ is as in Lemma~\ref{lem:second-moment}.
	We use $\gd=\delta_G$ and the choice $p>2+2/\delta_G$ made before the construction in Section~\ref{ssec:amplify}.
Let $(L_i,G_i)$ be as in~\eqref{eq:LG}.
Proposition~\ref{prop:regeneration} gives
\begin{align}\label{eq:second-moments}
	0<\E^0L<\infty,\quad \E^0(L^2)+\E^0(G^2)<\infty.
\end{align}
Define $m_L:=\E^0L$, $\mu:=\E^0G/m_L$, $Z_i:=G_i-\mu L_i$ and
\begin{align}\label{eq:sigma}
	\gs^2:={\var^0(G-\mu L)}/{m_L},
\end{align}
where the variance is computed under $\pr^0$.
By~\eqref{eq:G-bound}, $\abs{Z_i}\le(C_0+\mu)L_i$.

\begin{lem}[Non-degeneracy]\label{lem:nondegenerate}
	Under Assumption~\ref{ass:law} and $\ga>1+2/\theta$, $\var(G-\mu L)>0$ whenever it is finite.
\end{lem}

\begin{proof}
	Since $\E^0(G-\mu L)=0$, it suffices to exhibit an event of positive $\pr^0$-probability on which $L$ is constant and $G$ is not.
	Let $e^\ast:=\{1,3B-1\}$, an edge of length $3B-2>B$ which strictly crosses the boundaries $B$ and $2B$ and no other.
	Let $c:=B+\floor{B/2}$ be the midpoint of the cell $I_1$.
	Since $0$ lies in the support of $\go$,
		we can fix $0<c_1<c_2<a/(2B)$ with $\pr(c_1\le\go\le c_2)>0$, and we define $\eps_0:=c_1/2$.
	Consider the event $E$ on which
	\begin{enumeratei}
		\item $Y_{-1},Y_0,Y_2,Y_3<b/2$, $t_{0,1}<\eps_0$, $t_{3B-1,3B}<\eps_0$, and $A_{-1}^-$, $A_4^+$ hold;
		\item every nearest-neighbor edge inside $I_1$ has $\go_e\le c_2$,
				and every edge inside $I_1$ that is incident to $c$ or strictly crosses $c$ has $\go_e\ge c_1$;
		\item every edge in no cell that strictly crosses $c$ or is incident to $c$, other than $e^\ast$,
				has passage time at least $2\eps_0$,
				the crossing parts $H_0^\times$ and $H_3^\times$ of the hub events at $0$ and $3B$ hold, and $e^\ast$ is forbidden.
	\end{enumeratei}
	The three groups of conditions concern the disjoint edge sets ``inside the cells other than $I_1$'',
		``inside $I_1$'' and ``in no cell'', so they are independent.
	Group (i) consists of decreasing events of positive probability, hence has positive probability by Harris' inequality.
	Group (ii) consists of finitely many independent interval conditions, each of positive probability.
	Group (iii) imposes finitely many lower bounds on each crossing edge and one upper bound on $e^\ast$,
		and the infinite product of the corresponding probabilities is positive because $\sum_d(d-1)\big(p_d+F(2\eps_0d^{-\ga})\big)<\infty$ by Lemma~\ref{lem:summability} and $\ga\theta>2$.
	Hence $\pr(E)>0$.

	On $E$ the boundaries $0$ and $3B$ are hubs.
	Indeed, (ii) gives $Y_1\le Bc_2<a/2$, because along the nearest-neighbor path of $I_1$ each edge costs $\go_e\le c_2$.
	Together with (i) and $b<a$ this gives $S_1,\ldots,S_4<0$ and $S_{-1},\ldots,S_3>S_4$,
		so $A_0^+$ and $A_3^-$ follow from $A_4^+$, and $A_0^-$, $A_3^+$ from $A_{-1}^-$, $A_4^+$,
		as in the proof of Lemma~\ref{lem:cut}.(ii).
	The portal and crossing conditions at $0$ and $3B$ are contained in (i) and (iii).
	The boundaries $B$ and $2B$ are covered by the forbidden edge $e^\ast$.
	Hence under $\pr^0$, on $E$ the next hub after $0$ is $3B$ and $L=3B$.

	Write $G=\min(T',\,t_{e^\ast}+T'')$,
		where $T'$ is the infimum over paths from $0$ to $3B$ avoiding $e^\ast$,
		and $T''$ is the infimum over the remaining parts of paths that use $e^\ast$ once.
	Neither depends on $\go_{e^\ast}$.
	Every path avoiding $e^\ast$ either visits $c$ or uses an edge strictly crossing $c$,
		and by (ii)--(iii) every such edge costs at least $2\eps_0$, so $T'\ge2\eps_0$.
	On the other hand $T''\le t_{0,1}+t_{3B-1,3B}<2\eps_0$ by (i).
	Hence $G=(3B-2)^{\ga}\go_{e^\ast}+T''$ whenever $(3B-2)^{\ga}\go_{e^\ast}<T'-T''$,
		and on that range $G$ is strictly increasing in $\go_{e^\ast}$.
	Since $E$ depends on $\go_{e^\ast}$ only through the upper bound ``$e^\ast$ forbidden'',
		the conditional law of $\go_{e^\ast}$ given the other edges and $E$ charges every neighborhood of $0$.
	Therefore $G$ is not a.s.~constant on $E$, while $L=3B$ there.
\end{proof}

\subsection{Residual terms}
Let $\Nn=\max\{i:\hub_i\le n\}$ in the stationary enumeration,
	let $M_n:=(\Nn-1)_+$ be the number of complete blocks in $[\hub_1,n]$ on $\{\hub_1\le n\}$,
	and let $\Ares_n:=n-\hub_{\Nn}$ be the backward recurrence time.
Recall~\eqref{eq:decomp} and define
\begin{align}\label{eq:residuals}
	R_n^-:=T(0,\hub_1)-\mu\hub_1,\quad R_n^+:=T(\hub_{\Nn},n)-\mu \Ares_n .
\end{align}
By Proposition~\ref{prop:regeneration}.(ii),
	$\abs{R_n^-}\le(C_0+\mu)(\hub_1+B)$ and $\abs{R_n^+}\le(C_0+\mu)(\Ares_n+B)$.
Under the stationary, randomly shifted grid, $\Ares_n\equald\hub_1-1$ for every integer $n$.
In particular these variables are tight.
The inequality $\Ares_n\le n$ holds on $E_n:=\{\hub_1\le n\}$.

\begin{lem}\label{lem:residual}
	We have $\E(\hub_1)+\E(\Ares_n;\hub_1\le n)<\infty$, $\E(\hub_1^2;\hub_1\le n)/n\to0$ and $\E(\Ares_n^2;\hub_1\le n)/n\to0$.
	Consequently $(R_n^-+R_n^+)/\sqrt n\toP0$ and $\E\big((R_n^-+R_n^+)^2;\hub_1\le n\big)=o(n)$.
\end{lem}

\begin{proof}
	By~\eqref{eq:palm-stat}, $\E(\hub_1)=\sum_{v\ge1}\E^0((L-v+1)_+)/m_L=\E^0(L(L+1))/(2m_L)<\infty$, and the same for $\Ares_n\ind_{\hub_1\le n}$.
	Moreover, since $\Ares_n\le n$ and $\pr(\Ares_n\ge v)=\sum_{w\ge v+1}\pr^0(L\ge w)/m_L$, we have
	\begin{align*}
		\frac{\E(\Ares_n^2;\hub_1\le n)}n
		&\le\frac2n\sum_{v\le n}v\,\pr(\Ares_n\ge v)\\
		&=\frac2{nm_L}\sum_{w\ge1}\pr^0(L\ge w)\sum_{v\le\min(w,n)}v
		\le\frac1{m_L}\sum_{w\ge1}\pr^0(L\ge w)\,\frac{(\min(w,n)+1)^2}{n}.
	\end{align*}
	Each summand tends to $0$ as $n\to\infty$ and is bounded by $4(w+1)\pr^0(L\ge w)$,
		which is summable because $\E^0(L^2)<\infty$.
	Dominated convergence gives $\E(\Ares_n^2;\hub_1\le n)/n\to0$.
	The same argument applies to $\E(\hub_1^2;\hub_1\le n)$.
	The remaining claims follow from the bounds on $R_n^\pm$ and Chebyshev's inequality.
\end{proof}

\subsection{Central limit theorem}
\begin{proof}[Proof of Theorem~\ref{thm:main}]
	The law of large numbers~\eqref{eq:LLN-main} is Proposition~\ref{prop:lln}.
	The variables $(Z_i)_{i\ge1}$ are i.i.d., centered, with variance $\var^0Z=m_L\gs^2\in(0,\infty)$ by~\eqref{eq:second-moments} and Lemma~\ref{lem:nondegenerate},
		and they are independent of $\hub_1$.
	By~\eqref{eq:decomp}, on $\{\hub_1\le n\}$,
	\begin{align}\label{eq:Tn-decomp}
		T_n-\mu n=\sum_{i=1}^{M_n}Z_i+R_n^-+R_n^+ .
	\end{align}
	The renewal strong law gives $M_n/n\to1/m_L$ a.s.
	Hence Anscombe's theorem~\cite{Ans52},
		\cite[Theorem~1.3.1]{Gut09} yields
	\[
		n^{-1/2}\sum_{i\le M_n}Z_i\tod\N(0,\var(Z_1)/m_L)=\N(0,\gs^2),
	\]
	and Lemma~\ref{lem:residual} with Slutsky's theorem proves~\eqref{eq:CLT-main}.

	For the moment assertions, $N_n$ is a bounded stopping time for $\cF_i:=\gs(\hub_1,(L_j,Z_j)_{1\le j\le i})$,
		with $\{N_n=0\}=\{\hub_1>n\}\in\cF_0$.
	The positive-index blocks are independent of $\hub_1$.
	Stationarity on $\dZ$ gives $\E N_n=n/m_L$. More explicitly, on $\{\hub_1\le n\}$,
$N_n$ is the first $k\ge1$ for which $\hub_1+\sum_{j=1}^kL_j>n$;
otherwise $N_n=0$. Hence $\{N_n\le i\}\in\cF_i$ and $N_n\le\floor{n/B}+1$.
The centered partial sums form a square-integrable martingale, so optional stopping yields
	\begin{align}\label{eq:wald}
		\E\Big(\sum_{i=1}^{N_n}Z_i\Big)=0,\quad
		\E\Big[\Big(\sum_{i=1}^{N_n}Z_i\Big)^2\Big]=\var(Z_1)n/m_L.
	\end{align}
	Set $D_n:=Z_{N_n}\1_{E_n}$, assigning $D_n=0$ when $N_n=0$.
	Then $\sum_{i=1}^{M_n}Z_i=\sum_{i=1}^{N_n}Z_i-D_n$ on the whole space.
	Summing over possible positive hub locations and using the Palm block law gives
	\begin{align*}
		\E D_n^2
		& =\frac1{m_L}\sum_{s=1}^n\E^0[Z^2;L>n-s]=o(n),\text{ and }
		\E|D_n| \le\frac1{m_L}\sum_{r\ge0}\E^0[|Z|;L>r]
		=\frac{\E^0[L|Z|]}{m_L}<\infty.
	\end{align*}
	The first conclusion is Ces\`aro averaging, since $\E^0 Z^2<\infty$; the second is uniform in $n$,
		since $|Z|\le(C_0+\mu)L$ and $\E^0L^2<\infty$.
	Cauchy--Schwarz and~\eqref{eq:wald} now give
	\begin{align}\label{eq:SM-moments}
		\E\Big(\sum_{i=1}^{M_n}Z_i\Big)=O(1),\quad
		\E\Big(\Big(\sum_{i=1}^{M_n}Z_i\Big)^2\Big)=\var(Z_1)n/m_L+o(n).
	\end{align}

	On $\{\hub_1>n\}$ we use $0\le T_n\le W_n:=\sum_{m=-1}^{\ceil{n/B}}Y_m$,
		$\E(W_n^p)\le Cn^p$ and $\pr(\hub_1>n)\le Cn^{-1-\gd}$ (Lemma~\ref{lem:second-moment}), so that by H\"older's inequality
	\begin{align}\label{eq:bad-event}
		\E\big((T_n-\mu n)^2;\hub_1>n\big)\le Cn^2\pr(\hub_1>n)^{1-2/p}\le Cn^{2-(1+\gd)(1-2/p)}=o(n),
	\end{align}
	because $p>2+2/\gd$.
	In the same way,
	\begin{align*}
		\E(\abs{T_n-\mu n};\hub_1>n)\le Cn\,\pr(\hub_1>n)^{1-1/p}\le Cn^{1-(1+\gd)(1-1/p)}=o(1),
	\end{align*}
	because $p>1+1/\gd$.
	On $\{\hub_1>n\}$ the sums in~\eqref{eq:SM-moments} are empty,
		so~\eqref{eq:Tn-decomp} may be used on $\{\hub_1\le n\}$ and the two bounds just proved on its complement.
	Combining~\eqref{eq:Tn-decomp},~\eqref{eq:SM-moments},
		Lemma~\ref{lem:residual} and~\eqref{eq:bad-event} gives $\E T_n-\mu n=O(1)+O(\E(\hub_1+\Ares_n))+o(1)=O(1)$ and $\E((T_n-\mu n)^2)=\var(Z_1)n/m_L+o(n)$,
		where the cross terms are $o(n)$ by the Cauchy--Schwarz inequality.
	Hence $\var(T_n)=\gs^2n+o(n)$.
\end{proof}

\begin{proof}[Proof of Corollaries~\ref{cor:fclt} and~\ref{cor:count}]
	\textbf{Corollary~\ref{cor:fclt}.}
	Let $U_k:=\sum_{i\le k}Z_i$ and $\hub_{k+1}=\hub_1+\sum_{i\le k}L_i$.
	Donsker's theorem gives $(U_{\floor{ns}}/\sqrt n)_{s\ge0}$ a Brownian limit with variance parameter $\var^0Z$.
The uncentered renewal strong law gives
$M_{\floor{nt}}/n\to t/m_L$ uniformly on compact time intervals, almost surely.
Joint convergence with this deterministic clock and continuity of composition at continuous paths
give $(U_{M_{\floor{nt}}}/\sqrt n)_{t\ge0}\tod(\gs B(t))_{t\ge0}$.
Independence between $Z_i$ and $L_i$ is not required.
The centered clock fluctuation is not itself an increasing time change.
	The residual terms are handled uniformly on compacts:
	\[
		\sup_{t\le t_0}\abs{R^\pm_{\floor{nt}}}\le(C_0+\mu)\left(\max_{0\le i\le N_{\floor{nt_0}}+1}L_i+\hub_1+B\right)=o_{\pr}(\sqrt n)
	\]
	because $\E^0(L^2)<\infty$ implies $\max_{i\le k}L_i=o(\sqrt k)$ a.s.
For the initial segment before the first positive hub, linear control at $\hub_1$ gives
$T(0,m)\le T(0,\hub_1)+T(m,\hub_1)\le2C_0(\hub_1+B)$ for $0\le m<\hub_1$.
This fixed random bound divided by $\sqrt n$ vanishes almost surely.
The replacement of $\mu\floor{nt}$ by $\mu nt$ contributes at most $\mu/\sqrt n$ uniformly in $t$.
Thus the preceding uniform argument also covers the initial incomplete block and the rounding of the spatial coordinate.

	\smallskip\noindent\textbf{Corollary~\ref{cor:count}, hub count.}
	The a.s.~statement is the renewal strong law.
	Writing $n=\hub_1+\sum_{i=1}^{M_n}L_i+\Ares_n$ and $U'_k:=\sum_{i\le k}(L_i-m_L)$ gives
	\begin{align*}
		\frac{\Nn-n/m_L}{\sqrt n}=-\frac1{m_L}\frac{U'_{M_n}}{\sqrt n}+\frac{m_L-\hub_1-\Ares_n}{m_L\sqrt n},
	\end{align*}
	where the first term converges to $\N(0,\var(L)/m_L^3)$ by Anscombe's theorem and the second is $o_P(1)$ by Lemma~\ref{lem:residual}.
	Non-degeneracy, $\var(L)>0$, follows from Lemma~\ref{lem:nondegenerate}'s event $E$ (on which $L=3B$) together with the analogous event on which $0$ and $4B$ are consecutive hubs (replace $e^\ast$ by $\{1,4B-1\}$), which has positive probability by the same argument.

	\smallskip\noindent\textbf{Hop count.}
	For $x<y$ let $h(x,y)$ be the smallest number of edges of a geodesic from $x$ to $y$.
	If $x<z<y$ and $z$ is a hub, every geodesic from $x$ to $y$ passes through $z$ and its two halves are geodesics,
		and conversely the concatenation of two geodesics through $z$ is a geodesic.
	Hence $h(x,y)=h(x,z)+h(z,y)$.
	With $h_i:=h(\hub_i,\hub_{i+1})$, on $\{\hub_1\le n\}$,
	\begin{align*}
		\hop_n=h(0,\hub_1)+\sum_{i=1}^{M_n}h_i+h(\hub_{\Nn},n).
	\end{align*}
	A geodesic between consecutive hubs is a simple path inside the block, so $h_i\le L_i$,
		and the $h_i$ are i.i.d.~under $\pr^0$ by Proposition~\ref{prop:regeneration}.
	A geodesic from $\hub_{\Nn}$ to $n$ stays in $[\hub_{\Nn},\hub_{\Nn+1}]$ by Lemma~\ref{lem:pivotal},
		so $h(\hub_{\Nn},n)\le L_{\Nn}$, and likewise $h(0,\hub_1)\le L_0$.
	Both residuals are tight, because $L_{\Nn}$ and $L_0$ have the size-biased law with finite mean $\E^0(L^2)/m_L$.
	For the almost-sure limit, we use $L_0/n\to0$ and $L_{N_n}/n\to0$ a.s., the latter following from $\max_{i\le Cn}L_i=o(n)$ a.s. for integrable Palm gaps and $N_n=O(n)$ a.s. For the CLT the tight endpoint terms divided by $\sqrt n$ vanish in probability. The reward argument is then the same as for~\eqref{eq:LLN-main} and~\eqref{eq:CLT-main},
		with $(G_i)$ replaced by $(h_i)$ and $\mu$ by $\nu_H=\E^0(h)/m_L\in(0,1]$.

\smallskip\noindent\textbf{Cut point count.}
If $a<b$ are consecutive hubs and $v\in(a,b)$, then $v$ is a cut point in the definition of Section~\ref{ssec:main} if and only if every block geodesic from $a$ to $b$ visits $v$, \ie
$
T_{[a,b]\setminus\{v\}}(a,b)>T_{[a,b]}(a,b).
$
The strict inequality is equivalent to the stated geodesic criterion because the block is finite and all edge times are positive.

For sufficiency, every geodesic from $u<a$ to $w>b$ visits $a$ and $b$ in that order.
Its segment between them is a block geodesic by the pivotal property and confinement, and hence visits $v$.
Choosing an integer $k\ge1$ with $v-k\le a$ and $v+k\ge b$ proves the cut point property.
For necessity, choose $u$ and $w$ beyond both the defining neighborhood of $v$ and the block endpoints.
If a block geodesic from $a$ to $b$ avoided $v$, concatenate it with geodesics from $u$ to $a$ and from $b$ to $w$.
Confinement at $a,b$ keeps the two outer geodesics in $(-\infty,a]$ and $[b,\infty)$, respectively, so the concatenation avoids $v$.
Additivity at the hubs makes it a geodesic from $u$ to $w$, contradicting the cut point property.
Both restricted passage times in the criterion are determined by the block environment.
Every hub is itself a cut point.
Thus the number $c_i$ of cut points in $(\hub_i,\hub_{i+1}]$ is an i.i.d. block reward, with $1\le c_i\le L_i$.
Its two endpoint contributions are bounded by the lengths of the containing blocks. They vanish after division by $n$ a.s. and after division by $\sqrt n$ in probability, by the preceding hop-count argument.
The renewal-reward strong law and CLT therefore give the asserted formulas with $\nu_{\cut}=\E^0c/m_L$ and $\gs_{\cut}^2=\var^0(c-\nu_{\cut}L)/m_L$.
The same formulas allow zero variance for either count; strict positivity requires a separate argument, unlike the passage-time variance proved in Lemma~\ref{lem:nondegenerate}.

\end{proof}
\section{Tail asymptotics of hub blocks in the stable regime}\label{sec:tails}

\noindent\textbf{Proof roadmap.}
Lemma~\ref{lem:rho} computes the saving-tail function $\rhox$.
Lemma~\ref{lem:L-tail} gives block-length bounds, including bounds uniform in the truncation parameter $y$.
These imply the background-deviation estimates of Lemma~\ref{lem:no-big-dev}, which, together with the first-big-edge decomposition and the cycle formula, yield Proposition~\ref{prop:tail-equiv}.
Lemma~\ref{lem:iid-stable} then identifies the stable domain of attraction; renewal time change and negligible endpoint terms give Theorem~\ref{thm:stable}.
For the functional limit, Lemmas~\ref{lem:good-block} and~\ref{lem:shape} separately show that a long block typically contains one large edge and that its centered passage-time profile is approximately monotone. The $M_1$ approximation of Lemma~\ref{lem:M1-approx} then gives Theorem~\ref{thm:stable-functional}.
Uniformity in $y$ is essential because $y=\eta x$ (or $y=\eta u$) grows with the deviation scale: it is used in Lemma~\ref{lem:no-big-dev}, in the tail-sum split~\eqref{eq:A-split}, and again in Lemma~\ref{lem:good-block}, where $u=\eps a_n$ in the functional proof.

Throughout Sections~\ref{sec:tails} and~\ref{sec:stable} we assume Assumptions~\ref{ass:law} and~\ref{ass:rv} and $1+2/\theta<\ga<1+3/\theta$,
	so that $\gb=\theta(\ga-1)-1\in(1,2)$.
We fix the auxiliary parameters $\ell,K,a,b$ of Section~\ref{sec:spread}, with
\begin{align}\label{eq:p-stable}
	p>2\gb+8+\frac{2\gb}{\gb-1}
\end{align}
in~\eqref{eq:ell-choice}.
None of the constants below depends on $x$ or $u$.
Recall that $\mu=\E^0G/\E^0L$ is the time constant (Proposition~\ref{prop:lln}) and that $\mu<C_0$ (Lemma~\ref{lem:basic}).
In this section, $H_0$ denotes the event that the origin is a hub in the randomly shifted grid, and $H_0^{<y}$ denotes the corresponding $y$-hub event.
Their probabilities are vertex intensities, so $\pr(H_0)\E^0L=1$ and $\pr(H_0^{<y})\E^{0,y}L^{<y}=1$.
This differs by the factor $B$ from the probability of the cell-index event $H_m$ on a fixed grid used in~\eqref{eq:EL}.
For an edge $e=\{u_e,v_e\}$, $u_e<v_e$, of length $d_e=v_e-u_e$ we define its \emph{saving}
\begin{align}\label{eq:saving}
	\varsigma_e:=\mu d_e-t_e,
\end{align}
the amount by which $e$ undercuts the linear cost of its length.
If $\varsigma_e>0$ and $d_e>B$, then $e$ is forbidden, because $\mu<C_0$.
For $x>0$ we say that $e$ is \emph{$x$-saving} if $\varsigma_e\ge x$.
Since $\pr(\{u,u+d\}\text{ is $x$-saving})=\pr\big(\go\le(\mu d-x)_+d^{-\ga}\big)=F\big((\mu d-x)_+d^{-\ga}\big)$,
	the expected number of $x$-saving edges with a given left endpoint is exactly $\rhox(x)$.
An $x$-saving edge has length $d_e>x/\mu$, because $t_e>0$.

\begin{figure}[htbp]
	\centering
	\begin{tikzpicture}[x=1.1cm,y=.7cm,>=Stealth]
		\draw[thick] (0,0) -- (13,0);
		\foreach \x/\n in {0.5/{0}, 12.5/{L}} {\draw[thick] (\x,-0.25) -- (\x,0.25); \node[below] at (\x,-0.3) {$\n$};}
		\foreach \x in {2.2,4.1,9.6,11.2} {\draw[thick,blue!40] (\x,-0.15) -- (\x,0.15);}
		\node[below,blue!50,font=\footnotesize] at (3.15,-0.2) {$y$-hubs};
		\node[below,blue!50,font=\footnotesize] at (10.4,-0.2) {$y$-hubs};
		\draw[red!80!black,thick] (5,0) .. controls (6,2.4) and (8,2.4) .. (9,0);
		\node[red!80!black,above,font=\footnotesize] at (7,1.8) {$e_1$: the unique edge of $\cD_{\ge y}$, saving $\varsigma_{e_1}\approx x$};
		\node[below,font=\footnotesize] at (5,-0.05) {$u_1$};
		\node[below,font=\footnotesize] at (9,-0.05) {$v_1$};
		\draw[decorate,decoration={brace,amplitude=4pt,mirror}] (0.5,-0.9) -- (5,-0.9) node[midway,below=3pt,font=\footnotesize] {$\widehat T^y(0,u_1)-\mu u_1\ge-\eps x/2$};
		\draw[decorate,decoration={brace,amplitude=4pt,mirror}] (9,-0.9) -- (12.5,-0.9) node[midway,below=3pt,font=\footnotesize] {$\widehat T^y(v_1,L)-\mu(L-v_1)\ge-\eps x/2$};
		\foreach \x in {1.3,3.0,6.3,10.3,11.9} {\draw[black!45] (\x,0) .. controls (\x+0.2,0.35) and (\x+0.5,0.35) .. (\x+0.7,0);}
	\end{tikzpicture}
	\caption{The single-big-jump picture behind Proposition~\ref{prop:tail-equiv}.
		A reward $Z<-x$ is produced by one edge $e_1$ of $\cD_{\ge y}$, $y=\eta x$, whose saving is close to $x$.
		Outside $e_1$ the block is a union of $y$-blocks (short ticks) whose rewards are i.i.d.~and centered,
			and their partial sums cannot produce a deviation of order $x$.}\label{fig:bigjump}
\end{figure}
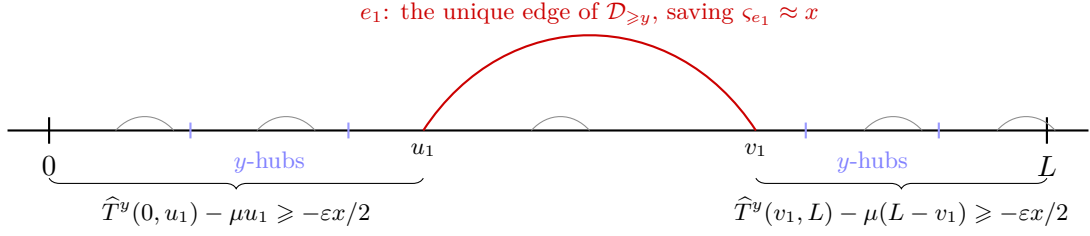

\subsection{The tail function \texorpdfstring{$\rhox$}{rho}}
\begin{lem}\label{lem:rho}
	The function $\rhox(x)=\sum_{d\ge1}F((\mu d-x)_+d^{-\ga})$ of~\eqref{eq:rho} satisfies~\eqref{eq:rho-asymp}.
	In particular $\rhox$ is regularly varying at infinity with index $-\gb$, and $\rhox(x)\asymp\sum_{d\ge x}p_d$.
	Moreover, for every $c\ge1$,
	\begin{align}\label{eq:rho-far}
		\limsup_{x\to\infty}\frac1{\rhox(x)}\sum_{d\ge cx}F\big((\mu d-x)_+d^{-\ga}\big)\le Cc^{-\gb}.
	\end{align}
\end{lem}

\begin{proof}
	Only $d>x/\mu$ contribute.
	Substituting $d=xs$ and writing $g(s):=(\mu s-1)s^{-\ga}$ for $s>1/\mu$,
	\begin{align*}
		\rhox(x)=\sum_{d>x/\mu}F\big(x^{1-\ga}g(d/x)\big)=x^{-\theta(\ga-1)}\sum_{d>x/\mu}g(d/x)^{\theta}\,\slv\big(x^{1-\ga}g(d/x)\big).
	\end{align*}
	By the uniform convergence theorem for slowly varying functions,
		$\slv(x^{1-\ga}g(s))/\slv(x^{1-\ga})\to1$ uniformly for $s$ in compact subsets of $(1/\mu,\infty)$,
		and Potter's bounds~\cite[Theorem~1.5.6]{BGT87} give, for $0<\gd<\min\{\theta,\theta-1/(\ga-1)\}$ and all large $x$,
		$\slv(x^{1-\ga}g(s))/\slv(x^{1-\ga})\le2\max(g(s)^{\gd},g(s)^{-\gd})$ uniformly in $s$.
	Since $g(s)\le\mu s^{1-\ga}$ and $g(s)\ge c(s-1/\mu)$ near $1/\mu$,
		and $\int_{1/\mu}^\infty g^{\theta\pm\gd}(s)\,ds<\infty$ for small $\gd$ because $g(s)^{\theta}\asymp s^{-(\gb+1)}$ at infinity,
		first restrict the Riemann sum to $1/\mu+h\le d/x\le R$.
On this interval uniform convergence gives the integral limit. Near $1/\mu$,
the dominating powers are bounded by $C(s-1/\mu)^{\theta-\gd}$, whose sum is
$O(h^{\theta-\gd+1}+x^{-1})$. For $s\ge R$, they are bounded by
$Cs^{-(\theta-\gd)(\ga-1)}$, whose integral and mesh sum tend to zero as $R\to\infty$.
Let $x\to\infty$, then $h\downarrow0$ and $R\to\infty$. This justifies the Riemann-sum passage and yields
	\begin{align*}
		\rhox(x)\sim x^{1-\theta(\ga-1)}\slv(x^{1-\ga})\int_{1/\mu}^\infty g(s)^{\theta}\,ds .
	\end{align*}
	The substitution $s=v/\mu$, then $v=1/w$, gives
	\begin{align*}
		\int_{1/\mu}^\infty g(s)^{\theta}\,ds=\mu^{\theta\ga-1}\int_1^\infty(v-1)^{\theta}v^{-\theta\ga}\,dv
		=\mu^{\theta\ga-1}\int_0^1(1-w)^{\theta}w^{\theta\ga-\theta-2}\,dw=\mu^{\theta\ga-1}\mathrm B(\theta+1,\gb),
	\end{align*}
	which is~\eqref{eq:rho-asymp} since $1-\theta(\ga-1)=-\gb$.
	The comparison with $\sum_{d\ge x}p_d$ follows from~\eqref{eq:karamata}.
	The same computation restricted to $d\ge cx$ gives the limit $\int_c^\infty g^{\theta}/\int_{1/\mu}^\infty g^{\theta}\le Cc^{-\gb}$, which is~\eqref{eq:rho-far}.
\end{proof}

\subsection{Independence structure}
For $y>B$, let $\mathcal F_{<y}:=\gs(U,\go_e:d_e<y)$.
The \emph{$y$-hubs} form the set $\Hloc\cap\cU_{<y}$,
	where $\cU_{<y}$ consists of boundaries crossed by no edge of $\cD_{<y}$.
This set is $\mathcal F_{<y}$-measurable and contains all original hubs.
Both comparison metrics $T_-^y$ and $T_+^y$ of Lemma~\ref{lem:basic} have these points as passage-time cut points.
For the lower metric, an artificial long edge costs $C_0(d+2B)$,
	which is at least the envelope-route bound through a $y$-hub;
	the other cases follow from the original portal and crossing arguments.

Write $\pr^{0,y}$ for the Palm law conditioned on a $y$-hub at zero, and $L^{<y}$ for its gap.
For $\diamond\in\{-,+\}$, let
\[
	G_\diamond^y:=T_\diamond^y(0,L^{<y}),\quad
	Z_\diamond^y:=G_\diamond^y-\mu_\diamond^y L^{<y}.
\]
The marked $y$-blocks are i.i.d. under $\pr^{0,y}$ by the same cut factorization:
	all random crossing constraints concern edges of length less than $y$, while longer edges have deterministic times.
Moreover $0\le G_\diamond^y\le C_0L^{<y}$ and $\E^{0,y}Z_\diamond^y=0$.
Define $H_0^{<y}:=\{0\text{ is a }y\text{-hub}\}$.
Intensity and conditioning give
\begin{align}\label{eq:y-mono}
	B\le\E^{0,y}L^{<y}\le\E^0L,\quad
	\pr^{0,y}(L^{<y}>u)
	\le\frac{\pr(\Hloc\cap\cU_{<y}\cap(0,u]=\eset)}{\pr(H_0)}.
\end{align}
There is no analogous upper bound for the original Palm gap by the no-$y$-hub event.
The entire $\mathcal F_{<y}$ field, including both background rewards, is independent of $\mathcal D_{\ge y}$.
The three sets $\Hloc,\cD_{<y},\cD_{\ge y}$ are also mutually independent,
	since the envelopes and portals use edges of length at most $B$.

Two identities for the stationary hub process are used repeatedly.
First, for a non-negative function $f$ of the augmented environment (edge weights and grid phase) and its translates $f_v:=f\circ\gt_v$,
\begin{align}\label{eq:cycle}
	\E^0\Big(\sum_{v\in[0,L)}f_v\Big)=\E^0 L\cdot \E(f)\text{ (cycle formula~\cite{Thorisson00})}.
\end{align}
Here $\gt_v$ translates both the edge field and the grid.
Randomizing the phase is essential when $f$ involves the containing hub block.

Second, the \emph{first-big-edge decomposition}.
Let $y>B$ and suppose that $0$ is a hub and that some edge of $\cD_{\ge y}$ lies in the block $[0,L]$.
Let $e_1=\{u_1,v_1\}$ be the one with the smallest left endpoint, breaking ties by its right endpoint.
Since no edge of $\cD_{\ge y}$ has left endpoint in $[0,u_1)$ and none crosses $0$,
	every point of $\Hloc\cap(0,u_1)$ is covered by an edge of $\cD_{<y}$.
Hence $0$ is a $y$-hub whose $y$-block reaches at least to $u_1$.
This event is a function of $\Hloc$ and $\cD_{<y}$ and is independent of $\{e_1\in\cD\}$.
For any event $E\in\mathcal F_{<y}$,
\begin{align}\label{eq:first-big}
	\pr\big(H_0,\ E,\ \text{some edge of }\cD_{\ge y}\text{ lies in }[0,L]\big)
	\le\sum_{u_1\ge0}\pr\big(0\text{ is a }y\text{-hub},\ L^{<y}\ge u_1,\ E\big)\sum_{d\ge y}p_d ,
\end{align}
and, taking $E$ trivial and using~\eqref{eq:palm-stat} for the $y$-hub process together with~\eqref{eq:y-mono},
\begin{align}\label{eq:first-big-sum}
	\sum_{u_1\ge0}\pr\big(0\text{ is a }y\text{-hub},\ L^{<y}\ge u_1\big)=\pr(H_0^{<y})\big(1+\E^{0,y}(L^{<y})\big)=\pr(H_0^{<y})+1\le2 .
\end{align}

\begin{lem}[Two long edges]\label{lem:two-long}
Put $s_y:=\sum_{d\ge y}p_d$. For $y>B$ and $r\ge1$,
\[
 \pr^0\big(L\le r,\ \#\{e\in\cD_{\ge y}:e\subset[0,L]\}\ge2\big)
 \le C(r+1)s_y^2.
\]
\end{lem}
\begin{proof}
Choose the first long forbidden edge $e=\{u,u+d\}$ in the block.
Its existence implies the background event
$B_u^y:=\{0\hbox{ is a }y\hbox{-hub},\ L^{<y}\ge u\}$.
If $L\le r$, the second edge has left endpoint in $[0,r]\cap\dZ$.
For each fixed pair of distinct edges, $B_u^y$ is independent of both edge indicators.
Relax the requirement that $e$ is actually first and sum over the two edges.
The resulting bound is
$\pr(H_0)^{-1}\sum_{u\ge0}\pr(B_u^y)s_y(r+1)s_y$,
which is at most $2\pr(H_0)^{-1}(r+1)s_y^2$ by~\eqref{eq:first-big-sum}.
This is an unconditional product-space calculation before conditioning on $H_0$;
independence is not being asserted under the original Palm law.
\end{proof}

\subsection{The tail of the block length}
\begin{lem}[Tail of the block length]\label{lem:L-tail}
	There is $\eta_0\in(0,1)$ such that for every $\eta\in(0,\eta_0]$ there are $C=C(\eta)$ and $\gd_0>0$ with, for all $u\ge1$,
	\begin{align}
		\pr^0(L>u)                                                                                     & \le C\,u^{-\gb}\slv(u^{1-\ga}),\label{eq:L-tail} \\
		\pr^0\big(L>u,\ \text{no edge of }\cD_{\ge\eta u}\text{ lies in }[0,L]\big) & \le C\,u^{-\gb-\gd_0}.\label{eq:L-tail-nobig}
	\end{align}
For the original Palm gap there is also $c>0$ such that
$\pr^0(L>u)\ge c\,u^{-\gb}\slv(u^{1-\ga})$ for all sufficiently large $u$.
	Only the two upper bounds~\eqref{eq:L-tail} and~\eqref{eq:L-tail-nobig} are asserted for $L^{<y}$ under $\pr^{0,y}$, uniformly for $y>B$, with the truncated field $\cD_{<y}$ in the no-big-edge event. The lower bound is for the original gap only.
\end{lem}

\begin{proof}
	\smallskip\noindent\textbf{Step 1: local density.} Fix $\kappa\in(\max\{1/\gb,(\gb+2)/(\gb+3)\},1)$ and $r\in((\gb+2)/\kappa,(p-2)/2)$,
		which is a non-empty interval by~\eqref{eq:p-stable}, and let $D:=\floor{u^{\kappa}}$.
	Let $E_u$ be the event that every interval of length $D$ contained in $(0,u]$ contains a point of $\Hloc$.
	By Lemma~\ref{lem:loc} and a union bound over $O(u/D+1)$ deterministic intervals of length $D/2$, each sliding interval of length $D$ containing one of these,
		$\pr(E_u^c)\le C(u/D+1)D^{1-r}\le Cu^{1-\kappa r}\le Cu^{-\gb-1}$.

	\smallskip\noindent\textbf{Step 2: no big edge.} On $\{H_0,\ L>u\}$ every point of $\Hloc\cap(0,u)$ is covered by a forbidden edge with left endpoint in $[0,u)$.
	On $E_u$ choose points $m_1<\cdots<m_{K'}$ of $\Hloc\cap(0,u]$ at mutual distances greater than $D$ with $K'\ge cu^{1-\kappa}$,
		as in the proof of Lemma~\ref{lem:second-moment}.
	Suppose that no edge of $\cD_{\ge\eta u}$ lies in $[0,L]$.
	Call a forbidden edge with left endpoint in $[0,u)$ \emph{medium} if its length lies in $[D,\eta u)$ and \emph{short} if its length is less than $D$.

	Let $M$ be the number of medium edges.
	It is a sum of independent Bernoulli variables with mean $\E(M)\le u\sum_{d\ge D}p_d\le Cu^{1-\kappa\gb+o(1)}$,
		and for every integer $m\ge1$, $\pr(M\ge m)\le\E(M)^m/m!\le Cu^{-m(\kappa\gb-1)+o(1)}$.
	Fix $m$ with $m(\kappa\gb-1)>\gb+1$ and then $\eta_0$ so small that $m\,(\eta_0u/D+2)\le K'/2$ for all $u\ge u_0$,
		which is possible because $\eta_0u/D+2\le2\eta_0u^{1-\kappa}+2$ and $K'\ge cu^{1-\kappa}$.
	Small $u$ are absorbed into the constant $C(\eta)$.
	If $M<m$, the medium edges cover at most $K'/2$ of the points $m_i$, and the remaining points,
		which are determined by the medium edges and $\Hloc$, must each be covered by a short edge.
	Since a short edge covers at most one of the $m_i$,
		these coverage events are independent given the medium edges, and each has probability at most $c_B\le1/8$.
	Hence
	\begin{align*}
		\pr\big(H_0,\ L>u,\ \text{no edge of }\cD_{\ge\eta u}\text{ in }[0,L]\big)\le\pr(E_u^c)+\pr(M\ge m)+8^{-K'/2}\le Cu^{-\gb-\gd_0}.
	\end{align*}
	Dividing by $\pr(H_0)$ proves~\eqref{eq:L-tail-nobig}.

	\smallskip\noindent\textbf{Step 3: first big edge.} On $\{H_0,L>u\}$ minus the event of Step 2 there is an edge of $\cD_{\ge\eta u}$ in $[0,L]$.
	By~\eqref{eq:first-big} with $y=\eta u$ and $E$ trivial, and~\eqref{eq:first-big-sum},
	\begin{align*}
		\pr(H_0,L>u)\le Cu^{-\gb-\gd_0}+2\sum_{d\ge\eta u}p_d\le Cu^{-\gb-\gd_0}+C\eta^{-\gb}u^{-\gb}\slv(u^{1-\ga})
	\end{align*}
	by~\eqref{eq:karamata}.
	Dividing by $\pr(H_0)$ proves~\eqref{eq:L-tail}.
For the lower bound, the long edges $\{0,d\}$, $d>B$, are mutually independent and independent of the hub event at zero: they belong to no cell, and none strictly crosses zero.
If one of them with $d>u$ is forbidden, it excludes every hub in $(0,d)$, so $L\ge d>u$.
Consequently
\[
 \pr^0(L>u)\ge1-\prod_{d>\max(u,B)}(1-p_d)
 \sim\sum_{d>u}p_d
 \sim\frac{C_0^\theta}{\gb}u^{-\gb}\slv(u^{1-\ga}).
\]
This proves the claimed lower bound, but does not identify an exact asymptotic constant for the gap tail.
	The upper-bound argument applies verbatim to the $y$-hub process,
		whose forbidden edges form a sub-family of $\cD$, with $\pr(H_0^{<y})\ge\pr(H_0)$ in the last step.
\end{proof}

\begin{rem}[Uniformity in the truncation level]
The constants in the two upper bounds of Lemma~\ref{lem:L-tail} can be chosen independently of $y$.
The local field and its density estimate are unchanged. Medium- and short-edge coverage uses a
subfield of the original Bernoulli field, so its probabilities only decrease.
The first-edge sum is at most two for every truncation level, and
$\pr(H_0^{<y})\ge\pr(H_0)>0$ bounds the conditioning denominator uniformly.
These facts justify applying the estimates with $y=\eta x$ growing with $x$.
\end{rem}

\subsection{Deviations without a big edge}
The following bound controls the two background metrics over a linear number of $y$-blocks.
\begin{lem}\label{lem:no-big-dev}
	Let $R\ge1$, $\eps\in(0,1]$, and set
	\begin{align}\label{eq:eta-star}
		\eta_*:=\frac{\eta_0\eps}{16C_0(1+C_0)}.
	\end{align}
	Fix $\eta\in(0,\eta_*]$, define $y=\eta x$,
		and write $S_{\diamond,k}^y:=\sum_{j<k}Z_{\diamond,j}^y$ for $\diamond\in\{-,+\}$.
	Under $\pr^{0,y}$,
	\begin{align}\label{eq:no-big-dev}
		q(x)&:=\max_{\diamond\in\{-,+\}}\pr^{0,y}\Big(\max_{k\le Rx}|S_{\diamond,k}^y|\ge\eps x/4\Big)\longrightarrow0,\notag\\
		q'(x)&:=\pr^{0,y}\Big(\max_{j\le Rx}L_j^{<y}\ge\eps x/(8C_0)\Big)\longrightarrow0.
	\end{align}
\end{lem}
\begin{proof}
	Uniformly in $y>B$, Lemma~\ref{lem:L-tail} and $|Z_\diamond^y|\le2C_0L^{<y}$ give, for all sufficiently large $s$,
	\begin{align}\label{eq:Zy-tail}
		\pr^{0,y}(|Z_\diamond^y|>s)+\pr^{0,y}(L^{<y}>s)
		\le Cs^{-\gb}\slv(s^{1-\ga}).
	\end{align}
	Fix a sign $\diamond$ and truncate each reward at absolute value $x$,
		writing $\widetilde Z_j=Z_{\diamond,j}^y\1_{\{|Z_{\diamond,j}^y|\le x\}}$.
	The probability that any of the first $\floor{Rx}+1$ rewards is changed is at most $CRx^{1-\gb}\slv(x^{1-\ga})=o(1)$.
	Since the untruncated rewards are centered, integration of~\eqref{eq:Zy-tail} gives
	\[
		|\E^{0,y}\widetilde Z_j|\le Cx^{1-\gb}\slv(x^{1-\ga}),\quad
		\var^{0,y}(\widetilde Z_j)\le Cx^{2-\gb}\slv(x^{1-\ga}).
	\]
	Thus the accumulated mean over $Rx$ terms is $o(x)$.
	Kolmogorov's maximal inequality for the independent centered truncated rewards bounds a deviation of size $\eps x/8$ by $C_{R,\eps}x^{1-\gb}\slv(x^{1-\ga})=o(1)$.
	This proves the assertion for either sign, uniformly in $y$.
	A union bound using the gap term in~\eqref{eq:Zy-tail} proves the assertion for $q'(x)$.
\end{proof}

\subsection{Tail equivalence for the block reward}
\begin{prop}[Tail equivalence]\label{prop:tail-equiv}
	Let $Z=G-\mu L$ be the reward of a typical block.
	Then, as $x\to\infty$,
	\begin{align}
		\pr^0(Z<-x) & \sim\E^0(L)\,\rhox(x),\label{eq:Z-left} \\
		\pr^0(Z>x)  & =o\big(\rhox(x)\big).\label{eq:Z-right}
	\end{align}
	In particular $Z$ has a regularly varying left tail with index $-\gb$,
		$\E^0(\abs{Z}^q)<\infty$ for $q<\gb$, and $\E^0(Z)=0$.
\end{prop}

The idea of the proof is in Figure~\ref{fig:bigjump}.
A large negative reward requires a large saving,
	and Lemma~\ref{lem:no-big-dev} says that savings do not accumulate from many moderate edges.
So a large negative reward is produced by one edge with a saving close to $x$,
	whose expected number per block is exactly $\E^0(L)\rhox(x)$ by the cycle formula.
A large positive reward would require a block that is long and slow at the same time, which is even less likely.

\begin{proof}[Proof of the lower bound in~\eqref{eq:Z-left}]
	Fix $\eps\in(0,1)$ and $J\ge1$.
	Let $M_x$ be the number of $(1+\eps)x$-saving edges in the block $[0,L]$ of the hub $0$.
	Every such edge is forbidden and has length greater than $y_1:=x/\mu$, so it lies in the block of its left endpoint.
	By the cycle formula~\eqref{eq:cycle} applied to $f:=$ the number of $(1+\eps)x$-saving edges with left endpoint at $0$,
	\begin{align}\label{eq:EM}
		\E^0(M_x)=\E^0(L)\,\rhox((1+\eps)x).
	\end{align}

	\smallskip\noindent\textbf{Long blocks.} Apply~\eqref{eq:cycle} to $f$ times the indicator that the block containing $0$ is longer than $Jx$.
	This gives $\E^0(M_x;L>Jx)=\E^0(L)\,\E\big(f;\hub_1-\hub_0>Jx\big)$,
		and $\{\hub_1-\hub_0>Jx\}$ is contained in $\{\hub_1>Jx/2\}\cup\{\hub_0\le-Jx/2\}$.
	Write $f=\sum_{d}f_d$ with $f_d$ the indicator that $\{0,d\}$ is $(1+\eps)x$-saving.
	The edge $\{0,d\}$ strictly crosses no point outside $(0,d)$ and lies in no cell,
		so the event $\{\hub_0\le-Jx/2\}$ and, for $d<Jx/4$, the event $\{\text{no hub in }(Jx/4,Jx/2]\}$ are independent of it.
	Hence, by~\eqref{eq:palm-stat},~\eqref{eq:L-tail}, Karamata's theorem and~\eqref{eq:rho-far},
	\begin{align}\label{eq:M-long}
		\E(f;\hub_1-\hub_0>Jx)&\le2\rhox(x)\,\pr\big(\text{no hub in }(0,Jx/4]\big)+\sum_{d\ge Jx/4}F\big((\mu d-x)_+d^{-\ga}\big)\notag\\
		&\le\rhox(x)\big(o(1)+CJ^{-\gb}\big).
	\end{align}

\smallskip\noindent\textbf{Two saving edges.}
We claim $\E^0(M_x(M_x-1);L\le Jx)=o(\rhox(x))$.
Put $s=s_{y_1}=\sum_{d\ge y_1}p_d$ and $q=\rhox((1+\eps)x)$.
For an ordered pair of saving edges, choose the first forbidden edge of length at least $y_1$.
If it is one member of the pair, the unconditional prefix-background calculation
in Lemma~\ref{lem:two-long} bounds the contribution by $C(Jx+1)sq$.
If it is neither member, sum over three distinct edges; the same calculation gives
$C(Jx+1)^2sq^2$.
Thus
\[
 \E^0(M_x(M_x-1);L\le Jx)
 \le C(Jx+1)sq+C(Jx+1)^2sq^2=o(\rhox(x)),
\]
because $s\asymp\rhox(x)$, $q\asymp\rhox(x)$ and $x\rhox(x)\to0$.
In these sums the prefix event uses only edges shorter than $y_1$, while the distinct long-edge
indicators are independent in the unconditioned environment.
	Since $M_x\le M_x(M_x-1)$ on $\{M_x\ge2\}$,
	\begin{align}\label{eq:M-ge1}
		\pr^0(M_x\ge1)&\ge\E^0(M_x)-\E^0(M_x;M_x\ge2)\notag\\
		&\ge\E^0(L)\rhox((1+\eps)x)-\E^0(M_x(M_x-1);L\le Jx)-\E^0(M_x;L>Jx),
	\end{align}
	and the last two terms are $\E^0(L)\rhox(x)(o(1)+CJ^{-\gb})$ by~\eqref{eq:M-long} and the claim.

	\smallskip\noindent\textbf{Position of the saving edge.} Define $y_0:=\eps x/(4C_0)-B$.
	We show that
	\begin{align}\label{eq:position}
    \begin{split}
    		\pr^0\big(M_x\ge1,\ \text{some $(1+\eps)x$-saving edge $e$ in the block has }u_e>y_0\text{ or }L-v_e>y_0\big)\\
		\quad\le\rhox(x)\big(o(1)+CJ^{-\gb}\big).
    \end{split}
	\end{align}
	The event $\{L>Jx\}$ has $\pr^0$-probability at most $CJ^{-\gb}x^{-\gb}\slv(x^{1-\ga})$ by~\eqref{eq:L-tail}.
	On $\{L\le Jx\}$, if the block contains two edges of $\cD_{\ge y_1}$,
		then by~\eqref{eq:first-big} (the second edge having left endpoint in $[0,Jx)$) the probability is at most $CJx(\sum_{d\ge y_1}p_d)^2=o(\rhox(x))$.
	Otherwise the saving edge $e$ is the unique edge of $\cD_{\ge y_1}$ in the block, hence $e=e_1$,
		and $u_e>y_0$ has probability at most
	\begin{align*}
		\pr(H_0)^{-1}\sum_{u_1>y_0}\pr\big(0\text{ is a }y_1\text{-hub},L^{<y_1}\ge u_1\big)\rhox((1+\eps)x)
		&\le C\rhox(x)\sum_{u_1>y_0}\pr^{0,y_1}(L^{<y_1}\ge u_1)\\
		&\le C\rhox(x)\,y_0^{1-\gb}\slv(y_0^{1-\ga})=o(\rhox(x)),
	\end{align*}
	by~\eqref{eq:first-big}, Lemma~\ref{lem:L-tail} for the $y_1$-hub process and Karamata's theorem.
	The event $\{L-v_e>y_0\}$ has the same probability by reflection symmetry.
	Indeed, the law $\pr$ and the definition of hubs are invariant under $x\mapsto-x$, so $\pr^0$ is invariant as well.
	Reflection maps the block $[0,\hub_1]$ to $[\hub_{-1},0]$ and the event $\{L-v_e>y_0\}$ to the event $\{u_e-\hub_{-1}>y_0\}$ for the left block.
	Under $\pr^0$ the marked blocks $[\hub_{-1},0]$ and $[0,\hub_1]$ have the same law (Proposition~\ref{prop:regeneration}.(i)),
		and the event in question is a function of the marked block,
		because by Lemma~\ref{lem:cut}.(ii) the $y_1$-hubs inside a hub block are determined by the block environment.
	This proves~\eqref{eq:position}.

	\smallskip\noindent\textbf{Conclusion.} On $\{M_x\ge1\}$ minus the event in~\eqref{eq:position},
		pick a $(1+\eps)x$-saving edge $e=\{u,v\}$ with $u\le y_0$ and $L-v\le y_0$.
	By Proposition~\ref{prop:regeneration}.(ii) applied at the hubs $0$ and $L$,
	\begin{align*}
		Z&\le\big[T(0,u)-\mu u\big]+\big[t_e-\mu d_e\big]+\big[T(v,L)-\mu(L-v)\big]\\
		&\le C_0(u+B)-\varsigma_e+C_0(L-v+B)\le\frac{\eps x}2-(1+\eps)x<-x .
	\end{align*}
	Hence, by~\eqref{eq:M-ge1} and~\eqref{eq:position},
		$\pr^0(Z<-x)\ge\E^0(L)\rhox((1+\eps)x)(1-o(1))-\E^0(L)\rhox(x)(o(1)+CJ^{-\gb})$.
	Since $\rhox((1+\eps)x)\sim(1+\eps)^{-\gb}\rhox(x)$, letting $x\to\infty$,
		then $J\to\infty$ and $\eps\downarrow0$ gives $\liminf\pr^0(Z<-x)/(\E^0(L)\rhox(x))\ge1$.
\end{proof}

\begin{proof}[Proof of the upper bound in~\eqref{eq:Z-left}]
	Fix $\eps\in(0,1/4)$ and $K\ge1$, let $\eta_\ast$ be as in~\eqref{eq:eta-star}, fix $\eta\in(0,\eta_\ast]$.
	Define $y:=\eta x$.
	We bound $\pr^0(Z<-x)$ by the $\pr^0$-probabilities of the following events, which cover $\{Z<-x\}$.
	Note that $Z<-x$ forces $L>x/\mu$, because $G\ge0$.
	\begin{enumeratea}
		\item $\mathbf{L>Kx.}$ By~\eqref{eq:L-tail} this has probability at most $CK^{-\gb}x^{-\gb}\slv(x^{1-\ga})\le CK^{-\gb}\E^0 L\rhox(x)$ for large $x$.
		\item \textbf{Some $(1-\eps)x$-saving edge lies in the block.} By the cycle formula as in~\eqref{eq:EM},
				the probability is at most $\E^0 L\rhox((1-\eps)x)\sim(1-\eps)^{-\gb}\E^0 L\rhox(x)$.
		\item \textbf{No edge of $\cD_{\ge y}$ lies in the block,
				and $L>x/\mu$.} Since $\eta\le\eta_\ast\le\eta_0/\mu$,
				this is the event in~\eqref{eq:L-tail-nobig} with $u=x/\mu$, of probability at most $Cx^{-\gb-\gd_0}=o(\rhox(x))$.
		\item \textbf{At least two edges of $\cD_{\ge y}$ lie in the block,
				and $L\le Kx$.} By~\eqref{eq:first-big}, with the second edge having left endpoint in $[0,Kx)$,
				the probability is at most $CKx(\sum_{d\ge y}p_d)^2\le CK\eta^{-2\gb}x\rhox(x)^2=o(\rhox(x))$.
		\item \textbf{Exactly one edge $e_1=\{u_1,v_1\}$ of $\cD_{\ge y}$ lies in the block,
				$\varsigma_{e_1}\le(1-\eps)x$, $L\le Kx$,
				and $Z<-x$.} Every edge other than $e_1$ on a block path has its original cost at least its cost in $T_-^y$:
				short edges are unchanged, and other long edges are non-forbidden.
			Define $\delta_y:=C\sum_{d\ge y}d^{-7}$, so $\mu_-^y\ge\mu-\delta_y$ by Lemma~\ref{lem:basic}.(ii), and $Kx\delta_y=o(x)$.
	\end{enumeratea}

	Write the $y$-hubs in $[0,L]$ as $0=\gamma_0<\cdots<\gamma_N=L$, with $N\le L\le Kx$.
	If a geodesic avoids $e_1$, then
	\[
		G\ge T_-^y(0,L)=S_{-,N}^y+\mu_-^yL,
	\]
	so $Z<-x$ implies $S_{-,N}^y<-x+Kx\delta_y$.
	If the geodesic uses $e_1$ from $w$ to $w'$, with $\{w,w'\}=\{u_1,v_1\}$, its remaining segments avoid $e_1$ and hence
	\begin{align*}
		Z\ge [T_-^y(0,w)-\mu w]+[T_-^y(w',L)-\mu(L-w')]
		+t_{e_1}-\mu(w'-w).
	\end{align*}
	The last term is at least $-(1-\eps)x$, whichever orientation is used.
	Thus at least one bracket must be below $-\eps x/2$.

	For the first bracket, take the last $y$-hub $\gamma_k\le w$.
	Pivotality gives
	\[
		T_-^y(0,w)-\mu w\ge S_{-,k}^y-\mu L_k^{<y}-Kx\delta_y.
	\]
	For sufficiently large $x$, this is at least $-\eps x/2$ unless
	\[
		E_x:=\Big\{\max_{k\le Kx}|S_{-,k}^y|\ge\eps x/4\Big\}
		\cup\Big\{\max_{j\le Kx}L_j^{<y}\ge\eps x/(8C_0)\Big\}
	\]
	occurs, since $\mu\le C_0/2$ and $Kx\delta_y=o(x)$.
	The second bracket is treated by reflection of the marked block, as in the lower bound.

	The event $E_x$ belongs to $\mathcal F_{<y}$,
		so it is independent of the long-edge field and~\eqref{eq:first-big} applies.
	It is enough, up to a factor two, to bound $\pr^0(e_1\text{ exists},E_x)$.
	For $A\ge1$, splitting the sum at $u_1=A$ and using~\eqref{eq:y-mono},
		Lemma~\ref{lem:L-tail} for the $y$-hub process and Karamata's theorem,
	\begin{align}\label{eq:A-split}
		\sum_{u_1\ge0}\pr\big(0\text{ is a }y\text{-hub},L^{<y}\ge u_1,E_x\big)
		&\le(A+1)\,\pr(H_0^{<y})\,\pr^{0,y}(E_x)\notag\\
		&\quad+\pr(H_0^{<y})\sum_{u_1>A}\pr^{0,y}(L^{<y}\ge u_1)\notag\\
		&\le C\big[(A+1)(q(x)+q'(x))+A^{1-\gb}\slv(A^{1-\ga})\big],
	\end{align}
	with $q,q'$ as in Lemma~\ref{lem:no-big-dev}.
	Choosing $A=A(x)\to\infty$ with $A(q(x)+q'(x))\to0$, which is possible by~\eqref{eq:no-big-dev}, the right-hand side tends to $0$.
	Multiplying by $\sum_{d\ge y}p_d\le C\eta^{-\gb}\rhox(x)$ shows that event e) has probability $o(\rhox(x))$.

	Events a), b), c), d), e) cover $\{Z<-x\}$.
	Hence $\limsup_{x\to\infty}\pr^0(Z<-x)/(\E^0 L\rhox(x))\le(1-\eps)^{-\gb}+CK^{-\gb}$,
		and letting $K\to\infty$, $\eps\downarrow0$ completes the proof.
\end{proof}

\begin{proof}[Proof of~\eqref{eq:Z-right}]
	By~\eqref{eq:G-bound}, $Z\le(C_0-\mu)L$, so $Z>x$ forces $L>x/C_0$.
	Fix $K\ge1$, let $\eta_\ast$ be as in~\eqref{eq:eta-star} with $\eps=1$, fix $\eta\in(0,\eta_\ast]$ and $y:=\eta x$.
	The event $\{L>Kx\}$ has probability at most $CK^{-\gb}x^{-\gb}\slv(x^{1-\ga})$.
	On $\{x/C_0<L\le Kx\}$, deleting all edges of length at least $y$ increases passage times,
		so by the pivotal property of the $y$-hubs and Lemma~\ref{lem:basic}.(iii) with~\eqref{eq:karamata},
	\begin{align*}
		Z&\le T_+^y(0,L)-\mu L=S_{+,N}^y+(\mu_+^y-\mu)L\\
		&\le S_{+,N}^y+CKx\,(\eta x)^{1-\gb}\slv(x^{1-\ga})+CKxy^{-6}=S_{+,N}^y+o(x),
	\end{align*}
	with $N\le Kx$.
	Hence $Z>x$ forces $\max_{k\le Kx}\abs{S_{+,k}^y}\ge x/4$ for large $x$.
	If no edge of $\cD_{\ge y}$ lies in the block, then since $L>x/C_0$ and $\eta\le\eta_0/C_0$,
		the event~\eqref{eq:L-tail-nobig} occurs, with probability $O(x^{-\gb-\gd_0})=o(\rhox(x))$.
	Otherwise~\eqref{eq:first-big} applies with $E:=\{\max_{k\le Kx}\abs{S_{+,k}^y}\ge x/4\}$,
		and the splitting~\eqref{eq:A-split} with Lemma~\ref{lem:no-big-dev} (for $\eps=1$) shows that the resulting bound is $o(1)\cdot\sum_{d\ge y}p_d=o(\rhox(x))$.
	Altogether $\limsup\pr^0(Z>x)/\rhox(x)\le CK^{-\gb}\E^0 L$, and $K\to\infty$ proves~\eqref{eq:Z-right}.

	The moment statements follow from~\eqref{eq:Z-left}--\eqref{eq:Z-right} and Karamata's theorem,
		and $\E^0 Z=\E^0(G)-\mu\E^0 L=0$ by the definition of $\mu$.
\end{proof}

\section{Stable fluctuations: proof of Theorem~\ref{thm:stable}}\label{sec:stable}

We continue under the assumptions of Section~\ref{sec:tails}.
Let $(Z_i)_{i\ge1}$ be the i.i.d.~block rewards in the stationary enumeration,
	with common law that of $Z$ in Proposition~\ref{prop:tail-equiv}.

\subsection{The i.i.d.~stable limit theorem}
\begin{lem}\label{lem:iid-stable}
	Let $(b_m)$ satisfy $m\,\pr^0(Z<-b_m)\to1$.
	Then $b_m$ is regularly varying with index $1/\gb$ and
	\begin{align*}
		\frac1{b_m}\sum_{i=1}^mZ_i\tod S_{\gb}(\gs_\gb,-1,0),\quad \gs_\gb^{\gb}=\Gamma(1-\gb)\cos(\pi\gb/2).
	\end{align*}
	Moreover, the sequence $(\sum_{i\le m}Z_i/b_m)_m$ satisfies Anscombe's condition.
	For every $\eps>0$ there is $\gd>0$ with
	\begin{align*}
		\limsup_{m\to\infty}\pr\Big(\max_{\abs{k-m}\le\gd m}\Big|\sum_{i=1}^kZ_i-\sum_{i=1}^mZ_i\Big|>\eps b_m\Big)<\eps.
	\end{align*}
\end{lem}

\begin{proof}
	By Proposition~\ref{prop:tail-equiv}, $\pr^0(\abs Z>x)$ is regularly varying with index $-\gb$,
		$\pr^0(Z>x)/\pr^0(\abs Z>x)\to0$, and $\E^0(Z)=0$.
	This is the tail-balance condition for the domain of attraction of a totally left-skewed $\gb$-stable law with $\gb\in(1,2)$.
	By~\cite[Theorem~XVII.5.3]{F71} (see also~\cite[Theorem~8.3.1]{BGT87} and~\cite[Theorem~1.8.1]{ST94}),
		$\sum_{i\le m}Z_i/b_m$ converges in law to the infinitely divisible law with no Gaussian part and L\'evy measure $\nu(dy)=\gb\abs y^{-\gb-1}\1_{\{y<0\}}dy$, centered to have mean zero.
	No additional centering is needed because $\E^0(Z)=0$ and $\gb>1$.
	The characteristic exponent of that law is
	\begin{align*}
		\psi(t)&=\int_{-\infty}^0\big(e^{ity}-1-ity\big)\gb\abs y^{-\gb-1}dy=\gb\Gamma(-\gb)(it)^{\gb}\\
		&=-\Gamma(1-\gb)\cos\big({\pi\gb}/2\big)\abs{t}^{\gb}\Big(1+i\,\sgn(t)\tan(\pi\gb/2)\Big),
	\end{align*}
	using $\int_0^\infty(e^{-sy}-1+sy)y^{-\gb-1}dy=\Gamma(-\gb)s^{\gb}$ for $\Re s\ge0$ and $\Gamma(1-\gb)=-\gb\Gamma(-\gb)$.
	Comparing with~\eqref{eq:stable-cf} identifies the limit as $S_\gb(\gs_\gb,-1,0)$ with $\gs_\gb^{\gb}=\Gamma(1-\gb)\cos(\pi\gb/2)$,
		which is positive because both factors are negative for $\gb\in(1,2)$.
	Regular variation of $b_m$ follows from that of $\pr^0(Z<-x)$ by inversion~\cite[Theorem~1.5.12]{BGT87}.
	Anscombe's condition for partial sums of i.i.d.~variables in the domain of attraction of a stable law,
		with regularly varying normalization, is~\cite[Theorem~1.3.2]{Gut09}.
	It follows from Etemadi's maximal inequality (see~\cite{Etemadi85}) applied to the increments.
\end{proof}

\subsection{Proof of Theorem~\ref{thm:stable}}
\begin{proof}
	\smallskip\noindent\textbf{Normalization.} By Proposition~\ref{prop:tail-equiv},
		$m\pr^0(Z<-b)\sim m\E^0L\rhox(b)$, so we may define $b_m$ by $m\E^0L\rhox(b_m)\to1$.
	With $m_n:=\floor{n/\E^0L}$ this gives $n\rhox(b_{m_n})\to1$, \ie\ $b_{m_n}\sim a_n$,
		where $a_n$ is the sequence of Theorem~\ref{thm:stable}.
	Since $\rhox$ is regularly varying with index $-\gb$ (Lemma~\ref{lem:rho}),
		$a_n=n^{1/\gb}\slv^\ast(n)$ with $\slv^\ast$ slowly varying,
		and $a_n\sim(\mu^{\theta\ga-1}\mathrm B(\theta+1,\gb)c_Fn)^{1/\gb}$ when $\slv\to c_F$.

	\smallskip\noindent\textbf{Decomposition.} By~\eqref{eq:Tn-decomp}, on $\{\hub_1\le n\}$,
	\begin{align}\label{eq:stable-decomp}
		T_n-\mu n=\sum_{i=1}^{M_n}Z_i+R_n^-+R_n^+,
	\end{align}
	with $M_n=(\Nn-1)_+$ and $R_n^\pm$ as in~\eqref{eq:residuals}.
	The probability of $\{\hub_1\le n\}$ tends to one.
	The residual $R_n^-$ does not depend on $n$ and is a.s.~finite, hence $o_P(a_n)$.
	For $R_n^+$ we have $\abs{R_n^+}\le(C_0+\mu)(\Ares_n+B)$,
		and the backward recurrence time $\Ares_n$ has a law that does not depend on $n$ and is a.s.~finite, hence tight.
	So $R_n^+=O_P(1)=o_P(a_n)$.

	\smallskip\noindent\textbf{Random index.} $\Nn/n\to1/\E^0L$ a.s.~by the renewal strong law,
		hence $M_n/m_n\to1$ in probability.
	Lemma~\ref{lem:iid-stable} and Anscombe's theorem~\cite[Theorem~1.3.1]{Gut09} give
	\begin{align*}
		\frac1{b_{m_n}}\sum_{i=1}^{M_n}Z_i\tod S_{\gb}(\gs_\gb,-1,0),
	\end{align*}
	and $b_{m_n}/a_n\to1$.
	Together with the residual bounds and Slutsky's theorem this proves~\eqref{eq:stable-main} with centering $\mu n$.

	\smallskip\noindent\textbf{Centering by the mean.} We show that $\E T_n-\mu n=O(n^{2-\gb+(\gb-1)/p+o(1)})$, which is $o(a_n)$.
	First, $\E(\sum_{i\le M_n}Z_i)=\sum_{i\ge1}\E(Z_i;\hub_{i+1}\le n)$.
	Since $Z_i$ is independent of $(\hub_1,L_1,\ldots,L_{i-1})$,
		conditioning on these variables gives $\E(Z_i;\hub_{i+1}\le n)=\E\big(\phi(n-\hub_i)\big)$ with $\phi(s):=\E^0(Z;L\le s)=-\E^0(Z;L>s)$,
		so that $\abs{\phi(s)}\le(C_0+\mu)\E^0(L;L>s)\le Cs^{1-\gb}\slv(s^{1-\ga})$ for $s\ge1$ by~\eqref{eq:L-tail} and Karamata's theorem, and $\phi(s)=0$ for $s<B$.
	By stationarity the hubs in $(0,n]$ have intensity $1/\E^0L$ per unit length, so
	\begin{align*}
		\Big|\E\Big(\sum_{i\le M_n}Z_i\Big)\Big|\le\E\Big(\sum_{i:\,\hub_i\le n}\abs{\phi(n-\hub_i)}\Big)\le\frac C{\E^0L}\sum_{s\le n}s^{1-\gb}\slv(s^{1-\ga})\le Cn^{2-\gb}\slv(n^{1-\ga}).
	\end{align*}

	Second, $\abs{R_n^-}\le(C_0+\mu)(\hub_1+B)$ and $\abs{R_n^+}\le(C_0+\mu)(\Ares_n+B)$, and $\hub_1$,
		$\Ares_n$ have the forward recurrence law~\eqref{eq:palm-stat},
		whose tail is $\pr(\hub_1\ge v)\le Cv^{1-\gb}\slv(v^{1-\ga})$ by~\eqref{eq:L-tail} and Karamata's theorem.
	Hence
	\[
		\E(\hub_1;\hub_1\le n)+\E(\Ares_n;\hub_1\le n)\le2\sum_{v\le n}\pr(\hub_1\ge v)\le Cn^{2-\gb}\slv(n^{1-\ga}).
	\]

	Third, on $\{\hub_1>n\}$ we use $0\le T_n\le W_n:=\sum_{m=-1}^{\ceil{n/B}}Y_m$,
		$\E(W_n^p)\le Cn^p$ and H\"older's inequality,
		which give
	\[
		\E(\abs{T_n-\mu n};\hub_1>n)\le Cn\pr(\hub_1>n)^{1-1/p}\le Cn^{1+(1-\gb)(1-1/p)}\slv(n^{1-\ga})^{1-1/p}.
	\]

	Combining the three bounds,
	\begin{align*}
		\abs{\E T_n-\mu n}\le Cn^{2-\gb}\slv(n^{1-\ga})+Cn^{2-\gb+(\gb-1)/p}\slv(n^{1-\ga})^{1-1/p}.
	\end{align*}
	Since $2-\gb<1/\gb$ for $\gb\in(1,2)$ (equivalently $(\gb-1)^2>0$) and $(\gb-1)/p<(\gb-1)^2/(2\gb)$ by~\eqref{eq:p-stable},
		the right-hand side is $o(n^{1/\gb-\gd'})$ for some $\gd'>0$, hence $o(a_n)$.
	Therefore~\eqref{eq:stable-main} also holds with $\E T_n$ in place of $\mu n$.
\end{proof}

\subsection{The functional limit: proof of Theorem~\ref{thm:stable-functional}}\label{ssec:functional}
We use the following facts about the Skorokhod topologies, see~\cite[Chapter~12]{Whitt02}.
The $M_1$ topology on $D[0,t_0]$ is metrized by
\begin{align*}
	d_{M_1}(x_1,x_2):=\inf\big\{\norm{u_1-u_2}_{\infty}\vee\norm{r_1-r_2}_{\infty}\big\},
\end{align*}
the infimum over all parametric representations $(u_j,r_j)$ of the completed graphs of $x_j$,
	$j=1,2$, that is, continuous maps from $[0,1]$ onto the completed graph (the graph of $x_j$ with the vertical segments $\{t\}\times[x_j(t-)\wedge x_j(t),x_j(t-)\vee x_j(t)]$ added) that are non-decreasing in the natural order of the graph.
For two paths $x_1,x_2$, one has $d_{M_1}(x_1,x_2)\le\norm{x_1-x_2}_\infty$, and $J_1$ convergence implies $M_1$ convergence.
For either topology, convergence in $D[0,\infty)$ is equivalent to convergence of the restrictions to $[0,t_0]$ for all $t_0$ in a dense set of continuity points of the limit~\cite[Section~12.9]{Whitt02};
	since $X_\gb$ has no fixed jump times, each deterministic $t_0>0$ is almost surely a continuity point. A countable dense set suffices.
We write $e$ for the identity map, $\norm{\cdot}_{t_0}$ for the supremum norm on $[0,t_0]$ and present a self-contained proof.

\begin{lem}\label{lem:M1-approx}
	\begin{enumeratei}
		\item\label{m1a:i} Let $x_n\to x$ in $(D[0,\infty),J_1)$,
				and let $\tau_n:[0,\infty)\to[0,\infty)$ be non-decreasing, right-continuous, with $\tau_n\to e$ uniformly on compacts.
			Then $x_n\circ\tau_n\to x$ in $(D[0,\infty),J_1)$.
		\item\label{m1a:ii} ($M_1$ approximation) Let $w\in D[0,t_0]$ be piecewise constant on a partition with possible jump times $0<s_1<\cdots<s_k\le t_0$,
				define $s_0:=0$, $s_{k+1}:=t_0$, and let $y\in D[0,t_0]$ and $\eps>0$.
			Suppose that for each $j\in\{0,\ldots,k\}$ with $s_j<s_{j+1}$ there is a monotone c\`adl\`ag function $g_j$ on $[s_j,s_{j+1})$ with $g_j(s_j)=w(s_j)$,
				$\abs{g_j(s_{j+1}-)-w(s_{j+1})}\le\eps$ and $\sup_{[s_j,s_{j+1})}\abs{y-g_j}\le\eps$,
				and that $\abs{y(t_0)-w(t_0)}\le\eps$.
			Then $d_{M_1}(y,w)\le2\eps+\max_{j\le k}(s_{j+1}-s_j)$.
			If $s_k=t_0$, the final interval is empty and its conditions are omitted; the terminal value condition still applies.
		\item\label{m1a:iii} If $x_n\to x$ in $(D[0,t_0],J_1)$, then
		$
			\sup_{t\le t_0}\abs{x(t)-x(t-)}\le\liminf_{n\to\infty}\,\sup_{t\le t_0}\abs{x_n(t)-x_n(t-)}.
		$
	\end{enumeratei}
\end{lem}

\begin{proof}
	\begin{enumeratei}
	\item 
Fix a continuity point $t_0$ of $x$, a continuity point $T>t_0+1$ of $x$, and $\eps>0$.
	There are increasing homeomorphisms $\gl_n$ of $[0,T]$ with $\norm{\gl_n-e}_T\to0$ and $\norm{x_n\circ\gl_n-x}_T\to0$.
	For all large $n$, $\tau_n([0,t_0])\subset[0,T]$.
	Define $\rho_n:=\gl_n^{-1}\circ\tau_n$ on $[0,t_0]$, which is non-decreasing,
		right-continuous and satisfies $\norm{\rho_n-e}_{t_0}\to0$.
	Then $x_n\circ\tau_n=(x_n\circ\gl_n)\circ\rho_n$ and $\norm{(x_n\circ\gl_n)\circ\rho_n-x\circ\rho_n}_{t_0}\le\norm{x_n\circ\gl_n-x}_T\to0$,
		so it suffices to show $x\circ\rho_n\to x$ in $(D[0,t_0],J_1)$.
	Since $x$ is c\`adl\`ag and continuous at $t_0$,
		there are $0=r_0<r_1<\cdots<r_m=T$ with $t_0\in(r_{j^\ast},r_{j^\ast+1})$ for some $j^\ast$ and $\sup_{r_j\le s,s'<r_{j+1}}\abs{x(s)-x(s')}<\eps$ for every $j$.
	Set $t^n_0:=0$ and, only for $1\le j\le j^\ast$, define
	\[
		t^n_j:=\inf\{t\in[0,t_0]:\rho_n(t)\ge r_j\}.
	\]
	Uniform convergence to $e$ makes these hitting times finite for large $n$ and gives $t^n_j\to r_j$.
	Thus $0=t^n_0<t^n_1<\cdots<t^n_{j^\ast}<t_0$ for large $n$.
	Let $\gl'_n$ be the increasing homeomorphism of $[0,t_0]$ that is linear on each $[t^n_j,t^n_{j+1}]$,
		$j<j^\ast$, and on $[t^n_{j^\ast},t_0]$, with $\gl'_n(t^n_j)=r_j$ and $\gl'_n(t_0)=t_0$.
	Then $\norm{\gl'_n-e}_{t_0}\to0$.
	For $j<j^\ast$ and $t\in[t^n_j,t^n_{j+1})$, both $\rho_n(t)$ and $\gl'_n(t)$ lie in $[r_j,r_{j+1})$ by right-continuity, monotonicity and the construction of $\gl'_n$.
	On the final interval $[t^n_{j^\ast},t_0]$, both lie in $[r_{j^\ast},r_{j^\ast+1})$ for large $n$, since $\rho_n(t_0)\to t_0<r_{j^\ast+1}$.
	Consequently $\norm{x\circ\rho_n-x\circ\gl'_n}_{t_0}<\eps$, and
		$\limsup_nd_{J_1,[0,t_0]}(x\circ\rho_n,x)\le\eps$.
	Letting $\eps\downarrow0$ proves (i).

	\item Let $\tilde y$ be equal to $g_j$ on $[s_j,s_{j+1})$ for each $j\le k$ and to $w(t_0)$ at $t_0$.
	Then $\norm{y-\tilde y}_{t_0}\le\eps$, so $d_{M_1}(y,\tilde y)\le\eps$,
		and it suffices to show $d_{M_1}(\tilde y,w)\le\eps+\max_j(s_{j+1}-s_j)$.
	If $s_k=t_0$, omit the empty final interval from the following concatenation.
	The completed graph of $w$ is the concatenation over $j=0,\ldots,k$ of the horizontal segment $[s_j,s_{j+1}]\times\{w(s_j)\}$ and the vertical segment $\{s_{j+1}\}\times[w(s_j),w(s_{j+1})]$ (a point if $w$ does not jump at $s_{j+1}$,
		and for $j=k$ simply the point $(t_0,w(t_0))$).
	The completed graph of $\tilde y$ is the concatenation over $j$ of the curve $\cC_j$ (the graph of $g_j$ on $[s_j,s_{j+1})$ with its jumps filled in,
		which is a continuous curve from $(s_j,w(s_j))$ to $(s_{j+1},\ga_j')$, $\ga_j':=g_j(s_{j+1}-)$,
		whose second coordinate moves monotonically) and the vertical segment $\{s_{j+1}\}\times[\ga'_j\wedge w(s_{j+1}),\ga'_j\vee w(s_{j+1})]$, of length at most $\eps$.
	Parametrize the $j$-th pieces of both graphs by $z\in[j,j+1]$ (and rescale $[0,k+1]$ to $[0,1]$ at the end) as follows.
	\begin{itemize}
		\item For $z\in[j,j+\frac13]$, the representation of $w$ traverses the horizontal segment linearly and that of $\tilde y$ rests at $(s_j,w(s_j))$.
		\item For $z\in[j+\frac13,j+\frac23]$, let $(\tilde u,\tilde r)$ be any parametric representation of $\cC_j$;
			its second coordinate $\tilde r$ is continuous and monotone from $\ga_j:=w(s_j)$ to $\ga'_j$.
		Let the representation of $w$ traverse the vertical segment with second coordinate $\psi(z):=\ga_j+(\tilde r(z)-\ga_j)(w(s_{j+1})-\ga_j)/(\ga'_j-\ga_j)$ if $\ga'_j\ne\ga_j$,
			and linearly from $\ga_j$ to $w(s_{j+1})$ if $\ga'_j=\ga_j$;
			in both cases $\abs{\tilde r(z)-\psi(z)}\le\abs{\ga'_j-w(s_{j+1})}\le\eps$,
			because $\abs{\tilde r(z)-\ga_j}\le\abs{\ga'_j-\ga_j}$.
		\item For $z\in[j+\frac23,j+1]$, the representation of $\tilde y$ traverses its vertical segment at $s_{j+1}$ linearly while that of $w$ rests at $(s_{j+1},w(s_{j+1}))$;
			the second coordinates differ by at most $\eps$.
	\end{itemize}
	Throughout the $j$-th piece both first coordinates lie in $[s_j,s_{j+1}]$.
	Both parametrizations are continuous and non-decreasing along their graphs,
		so $d_{M_1}(\tilde y,w)\le\eps\vee\max_j(s_{j+1}-s_j)$.

	\item Let $\gl_n$ be increasing homeomorphisms of $[0,t_0]$ with $\gd_n:=\norm{x_n\circ\gl_n-x}_{t_0}\to0$.
	For $t\in(0,t_0]$, $x_n(\gl_n(t)-)=(x_n\circ\gl_n)(t-)$ and $\abs{(x_n\circ\gl_n)(t-)-x(t-)}\le\gd_n$,
		so
	\[
		\abs{x(t)-x(t-)}\le2\gd_n+\abs{x_n(\gl_n(t))-x_n(\gl_n(t)-)}\le2\gd_n+\sup_{s\le t_0}\abs{x_n(s)-x_n(s-)}.
	\]
	\end{enumeratei}
	This completes the proof.
\end{proof}

The next two lemmas describe the passage time inside a long block.
Recall the $y$-hubs and the two comparison rewards $Z_{\diamond,j}^y$, $\diamond\in\{-,+\}$,
	from Section~\ref{sec:tails}, and write $S_{\diamond,k}^y:=\sum_{j<k}Z_{\diamond,j}^y$.
For a hub block $[0,L]$, write $0=\gamma_0<\cdots<\gamma_N=L$ for its $y$-hubs and $L_j^{<y}:=\gamma_{j+1}-\gamma_j$.
Finally, let $\widehat T^y$ be the original metric with the edges of $\cD_{\ge y}$ deleted.
This last metric is not asserted to be $\mathcal F_{<y}$-measurable.
Instead, the definition of a forbidden edge and the corrected lower comparison metric give the pointwise sandwich
\begin{align}\label{eq:deleted-sandwich}
	T_-^y(a,b)\le \widehat T^y(a,b)\le T_+^y(a,b),\quad a,b\in\dZ.
\end{align}

\begin{lem}[Long blocks]\label{lem:good-block}
	Fix $\eps\in(0,1)$, $J\ge1$ and $\eta\in(0,\eta_\ast\wedge\eta_0]$,
		with $\eta_\ast$ as in~\eqref{eq:eta-star} and $\eta_0$ as in Lemma~\ref{lem:L-tail}.
	For $u>B/\eta$, let $y:=\eta u$, and call the hub block $[0,L]$ \emph{$u$-good} if $u<L\le Ju$,
		exactly one edge $e_1=\{u_1,v_1\}$ of $\cD_{\ge y}$ lies in $[0,L]$, and
	\begin{align}\label{eq:good-block}
		\max_{\diamond\in\{-,+\}}\max_{k\le N}\big|S_{\diamond,k}^y\big|<\frac{\eps u}4,
		\quad \max_{j<N}L^{<y}_j<\frac{\eps u}{8C_0}.
	\end{align}
	Then $\pr^0\big(u<L\le Ju,\ [0,L]\text{ is not $u$-good}\big)=o\big(u^{-\gb}\slv(u^{1-\ga})\big)$ as $u\to\infty$.
\end{lem}

\begin{proof}
	Since $\eta\le\eta_0$, the event that $L>u$ and no edge of $\cD_{\ge y}$ lies in the block has probability $O(u^{-\gb-\gd_0})$ by~\eqref{eq:L-tail-nobig}.
	The event that $L\le Ju$ and at least two edges of $\cD_{\ge y}$ lie in the block has probability $O\big(Ju(\sum_{d\ge y}p_d)^2\big)=o(u^{-\gb}\slv(u^{1-\ga}))$,
		exactly as event (d) in the proof of Proposition~\ref{prop:tail-equiv}.
	If exactly one such edge lies in the block, $L\le Ju$ and~\eqref{eq:good-block} fails, then,
		since $N\le L/B\le Ju$, at least one of the two comparison-reward maximal deviations in Lemma~\ref{lem:no-big-dev},
		or the displayed maximal-gap event there, occurs with $x=u$.
	By~\eqref{eq:first-big}, the estimate~\eqref{eq:A-split} and Lemma~\ref{lem:no-big-dev},
		the union of these events has probability $o(1)\sum_{d\ge y}p_d=o(u^{-\gb}\slv(u^{1-\ga}))$, by~\eqref{eq:karamata}.
\end{proof}

\begin{lem}\label{lem:shape}
	In the setting of Lemma~\ref{lem:good-block} there is $u_0=u_0(\eps,J,\eta)$ such that on every $u$-good block with $u\ge u_0$,
		writing $\varsigma:=\varsigma_{e_1}$ and
	\begin{align}\label{eq:g-shape}
		g(x):=\min\big\{0,\,-\varsigma+2\mu(v_1-x)_+\big\},\quad x\in[0,L],
	\end{align}
	one has $\sup_{x\in[0,L]\cap\dZ}\abs{T(0,x)-\mu x-g(x)}\le5\eps u$.
	The function $g$ is non-increasing and $2\mu$-Lipschitz, with $g(0)=0$ and $g(L)=-\max(\varsigma,0)$.
\end{lem}

\begin{proof}
	Use the convention $L^{<y}_N:=0$, and let
	\begin{align*}
		\Delta:=2\max_{\diamond\in\{-,+\}}\max_{k\le N}\big|S_{\diamond,k}^y\big|
		+2C_0\Big(\max_{j<N}L^{<y}_j+B\Big)
		+L\max_{\diamond\in\{-,+\}}\big|\mu_\diamond^y-\mu\big| .
	\end{align*}
	We first show
	\begin{align}\label{eq:Tless-linear}
		\abs{\widehat T^y(a,b)-\mu(b-a)}\le\Delta\text{ for all }0\le a<b\le L.
	\end{align}
	For each $\diamond\in\{-,+\}$, the usual decomposition at the $y$-hubs and the linear-control bounds give
	\begin{align*}
		\abs{T_\diamond^y(a,b)-\mu(b-a)}\le\Delta .
	\end{align*}
	Indeed, if $j:=\max\{i:\gamma_i\le a\}$ and $k:=\max\{i:\gamma_i\le b\}$ are equal,
		use non-negativity for the lower bound and the cell route for the upper bound.
	If $j<k$, then
	\begin{align*}
		T_\diamond^y(\gamma_{j+1},\gamma_k)
		=S_{\diamond,k}^y-S_{\diamond,j+1}^y
		+\mu_\diamond^y(\gamma_k-\gamma_{j+1}),
	\end{align*}
	and the two end pieces cost at most $C_0(L_j^{<y}+B)$ and $C_0(L_k^{<y}+B)$.
	The definition of $\Delta$ then gives the displayed estimate.
	The sandwich~\eqref{eq:deleted-sandwich} proves~\eqref{eq:Tless-linear};
		symmetry gives the same bound with the arguments reversed.

	Now fix $x\in[0,L]\cap\dZ$.
	By Lemma~\ref{lem:pivotal}.\ref{pivot:ii}, a geodesic from the hub $0$ to $x$ stays in $[0,L]$. It is a simple path, so it uses $e_1$ at most once, and its remaining edges are not deleted.
	Conversely, every path in the deleted environment,
		and every concatenation of two such paths with $e_1$, is an admissible path.
	Therefore
	\begin{align*}
		T(0,x)=\min\Big\{\widehat T^y(0,x),\ \widehat T^y(0,u_1)+t_{e_1}+\widehat T^y(v_1,x),\ \widehat T^y(0,v_1)+t_{e_1}+\widehat T^y(u_1,x)\Big\}.
	\end{align*}
	Since $t_{e_1}=\mu(v_1-u_1)-\varsigma$,
		by~\eqref{eq:Tless-linear} the three terms are within $\Delta$, $2\Delta$ and $2\Delta$, respectively, of
	\begin{align*}
		\mu x,\quad \mu u_1+\mu(v_1-u_1)-\varsigma+\mu\abs{x-v_1}=\mu x-\varsigma+2\mu(v_1-x)_+,\\
		\text{and}\quad \mu v_1+\mu(v_1-u_1)-\varsigma+\mu\abs{x-u_1},
	\end{align*}
	and the third of these values is at least the second, because $\abs{x-u_1}-\abs{x-v_1}\ge-(v_1-u_1)$.
	Hence $\abs{T(0,x)-\mu x-g(x)}\le2\Delta$.
	On a $u$-good block, $\Delta\le\eps u/2+\eps u/4+2C_0B+Ju\max_\diamond|\mu_\diamond^y-\mu|$,
		and Lemma~\ref{lem:basic}.(ii)--(iii) and~\eqref{eq:karamata} give $Ju\max_\diamond|\mu_\diamond^y-\mu|=o(u)$.
	Thus $2\Delta\le5\eps u$ once $u\ge u_0(\eps,J,\eta)$.
	The properties of $g$ are immediate from~\eqref{eq:g-shape},
		since $x\mapsto2\mu(v_1-x)_+$ is non-increasing and $2\mu$-Lipschitz,
		and $\varsigma\le\mu(v_1-u_1)<2\mu v_1$ gives $g(0)=0$.
\end{proof}

\begin{proof}[Proof of Theorem~\ref{thm:stable-functional}]
	Let $Y_n(t):=(T(0,\floor{nt})-\mu nt)/a_n$, and with $M_n=(\Nn-1)_+$,
		$U_k:=\sum_{i\le k}Z_i$ for $k\ge1$ and $U_k:=0$ for $k\le0$, let $V_n(t):=U_{M_{\floor{nt}}}/a_n$.
	Fix $t_0>0$.
	By the remarks preceding Lemma~\ref{lem:M1-approx}, it suffices to prove $$Y_n\tod X_\gb\text{ in }(D[0,t_0],M_1).$$

	\smallskip\noindent\textbf{Step 1: $V_n\tod X_\gb$ in $J_1$.} Let $b_m$ and $m_n$ be as in the proof of Theorem~\ref{thm:stable}.
	By Proposition~\ref{prop:tail-equiv} the i.i.d.~variables $Z_i$ are centered and lie in the domain of attraction of a totally left-skewed $\gb$-stable law,
		so the functional limit theorem for i.i.d.~sums~\cite[Chapter~4]{Whitt02} gives $x_n:=(U_{\floor{m_nt}}/b_{m_n})_{t\ge0}\tod X_\gb$ in $(D[0,\infty),J_1)$,
		where $X_\gb$ is a L\'evy process whose marginal at time $1$ is identified by Lemma~\ref{lem:iid-stable};
		this is the process described in the theorem.
	Let $\tau_n(t):=\max(M_{\floor{nt}},0)/m_n$,
		a non-decreasing right-continuous step function with $\floor{m_n\tau_n(t)}=\max(M_{\floor{nt}},0)$,
		so that $V_n=(b_{m_n}/a_n)\,x_n\circ\tau_n$.
	By the renewal strong law $\tau_n(t)\to t$ a.s.\ for every $t$,
		hence uniformly on compacts because $\tau_n$ is monotone and the limit is continuous.
	Thus $(x_n,\tau_n)\tod(X_\gb,e)$ jointly~\cite[Section~11.4]{Whitt02},
		and since the composition map is measurable~\cite[Section~13.2]{Whitt02} and continuous at every point of the form $(x,e)$ by Lemma~\ref{lem:M1-approx}.\ref{m1a:i},
		the continuous-mapping theorem and $b_{m_n}/a_n\to1$ give $V_n\tod X_\gb$ in $(D[0,\infty),J_1)$, hence in $(D[0,t_0],M_1)$.

	\smallskip\noindent\textbf{Step 2: reduction.} Let $W_n(t):=Y_n(\hub_{N_{\floor{nt}}}/n)$ if $nt\ge\hub_1$ and $W_n(t):=0$ otherwise.
	By~\eqref{eq:additivity-hubs}, $Y_n(\hub_i/n)=(R_n^-+U_{i-1})/a_n$ for $i\ge1$,
		so $\norm{W_n-V_n}_{t_0}\le\abs{R_n^-}/a_n\to0$ a.s.,
		and $W_n$ is piecewise constant on $[0,t_0]$ with possible jumps at the partition points $s_i:=\hub_i/n\in(0,t_0]$.
	By Step 1 and the convergence-together theorem~\cite[Section~11.4]{Whitt02},
		it suffices to show that $d_{M_1}(Y_n,W_n)\to0$ in probability on $[0,t_0]$.
	For a block $i\ge1$ and $x\in[\hub_i,\hub_{i+1}]$ define $r_i(x):=T(\hub_i,x)-\mu(x-\hub_i)$,
		so that $r_i(\hub_{i+1})=Z_i$ and, again by~\eqref{eq:additivity-hubs},
	\begin{align}\label{eq:Y-W}
		Y_n(t)-W_n(t)=\frac{r_i(\floor{nt})-\mu(nt-\floor{nt})}{a_n},\quad t\in[s_i,s_{i+1}).
	\end{align}
	Recall also that $\abs{r_i(x)}\le(C_0+\mu)(L_i+B)$ by Proposition~\ref{prop:regeneration}.(ii).

	\smallskip\noindent\textbf{Step 3: the good event.} Fix $\eps\in(0,1)$ and $J\ge1$,
		let $\eta:=\eta_\ast\wedge\eta_0$ with $\eta_\ast$ given by~\eqref{eq:eta-star},
		and define $u:=\eps a_n$ and $K_n:=\ceil{2nt_0/\E^0L}$.
	Let $A_n$ be the event that
	\begin{enumeratei}
		\item $\hub_1\le\eps a_n$ and $N_{\floor{nt_0}}\le K_n$;
		\item $L_i\le Ju$ for all $1\le i\le K_n$;
		\item every block $i\in\{1,\ldots,K_n\}$ with $L_i>u$ is $u$-good in the sense of Lemma~\ref{lem:good-block};
		\item the block containing $\floor{nt_0}$ has length $L_{N_{\floor{nt_0}}}\le u$;
		\item $\max_{0\le i\le K_n}L_i\le\gd_nn$, where $\gd_n\to0$ is chosen below;
		\item $u\ge u_0(\eps,J,\eta)$ and $B\le u$.
	\end{enumeratei}
	We claim $\limsup_n\pr(A_n^c)\le Ct_0(J\eps)^{-\gb}$.
	Indeed, $\hub_1$ is a.s.\ finite and $N_{\floor{nt_0}}/n\to t_0/\E^0L$ a.s., so (i) fails with vanishing probability.
	The blocks $i\ge1$ are i.i.d.~with the $\pr^0$-law of $L$, so by~\eqref{eq:L-tail},
	\[
		\pr(\text{(ii) fails})\le K_n\pr^0(L>J\eps a_n)\le CK_n(J\eps a_n)^{-\gb}\slv(a_n^{1-\ga})\le Ct_0(J\eps)^{-\gb}
	\]
	for large $n$,
		because $na_n^{-\gb}\slv(a_n^{1-\ga})$ converges by Lemma~\ref{lem:rho} and $n\rhox(a_n)\to1$.
	By Lemma~\ref{lem:good-block}, $\pr(\text{(iii) fails, (ii) holds})\le K_n\,o(u^{-\gb}\slv(u^{1-\ga}))\to0$.
	The block containing $\floor{nt_0}$ has the size-biased law,
		so $\pr(\text{(iv) fails})\le C\E^0(L;L>u)\le Cu^{1-\gb}\slv(u^{1-\ga})\to0$ by~\eqref{eq:L-tail} and Karamata's theorem.
	For (v),
	\[
		\pr(\max_{1\le i\le K_n}L_i>\gd n)\le K_n\pr^0(L>\gd n)\le C\gd^{-\gb}n^{1-\gb}\slv(n^{1-\ga})\to0
	\]
	for every fixed $\gd>0$,
		and $L_0$ is a.s.\ finite, so a sequence $\gd_n\to0$ with $\pr(\text{(v) fails})\to0$ exists.
	Finally (vi) holds for all large $n$.

	\smallskip\noindent\textbf{Step 4: $M_1$ approximation on $A_n$.} On $A_n$ we apply Lemma~\ref{lem:M1-approx}.\ref{m1a:ii} with $w=W_n$,
		$y=Y_n$, the hub partition points $s_1<s_{2}<\cdots<s_k$ of $W_n$ in $(0,t_0]$,$k=N_{\floor{nt_0}}$,
		and the tolerance $\eps':=6(1+C_\ast)\eps$, where $C_\ast:=C_0+\mu$, as follows.
	\begin{itemize}
		\item On $[0,s_1)$: for $x<\hub_1$, the linear control to the left of the hub $\hub_1$ gives $T(0,x)\le T(0,\hub_1)+T(x,\hub_1)\le2C_0(\hub_1+B)$,
			so $\abs{Y_n}\le(2C_0+\mu)(\hub_1+B)/a_n\le3C_\ast\eps$ for large $n$,
			and $\abs{W_n(s_1)}=\abs{R_n^-}/a_n\le C_\ast(\hub_1+B)/a_n\le2C_\ast\eps$; take $g_0$ linear from $0$ to $W_n(s_1)$.
		\item On $[s_i,s_{i+1})$ with $1\le i<k$ and $L_i\le u$ (a short block): by~\eqref{eq:Y-W},
			$\abs{Y_n-W_n(s_i)}\le(C_\ast(u+B)+\mu)/a_n\le3C_\ast\eps$ and $\abs{W_n(s_{i+1})-W_n(s_i)}=\abs{Z_i}/a_n\le2C_\ast\eps$;
			take $g_i$ linear from $W_n(s_i)$ to $W_n(s_{i+1})$.
		\item On $[s_i,s_{i+1})$ with $1\le i<k$ and $L_i>u$ (a long block, hence $u$-good by (ii)--(iii)):
			let $g$ be the function~\eqref{eq:g-shape} of the block $[\hub_i,\hub_{i+1}]$ translated to $[0,L_i]$,
			and define $g_i(t):=W_n(s_i)+g(nt-\hub_i)/a_n$.
		Then $g_i$ is monotone, $g_i(s_i)=W_n(s_i)$, and by Lemma~\ref{lem:shape},
			the Lipschitz property of $g$ and~\eqref{eq:Y-W},
			$\sup_{[s_i,s_{i+1})}\abs{Y_n-g_i}\le(5\eps u+3\mu)/a_n\le6\eps$ and $\abs{g_i(s_{i+1}-)-W_n(s_{i+1})}=\abs{g(L_i)-Z_i}/a_n\le5\eps u/a_n=5\eps^2$.
		\item On $[s_k,t_0]$: the block $k=N_{\floor{nt_0}}$ is short by (iv),
			so $\abs{Y_n-W_n(s_k)}\le3C_\ast\eps$ there; take $g_k\equiv W_n(s_k)=W_n(t_0)$.
	\end{itemize}
	All hypotheses of Lemma~\ref{lem:M1-approx}.\ref{m1a:ii} hold with tolerance $\eps'$,
		and the spacings satisfy $s_1\le\eps a_n/n$ and $s_{j+1}-s_j\le L_j/n\le\gd_n$ by (v).
	Hence, on $A_n$ and for large $n$,
	\begin{align*}
		d_{M_1}(Y_n,W_n)\le2\eps'+\gd_n+\eps a_n/n .
	\end{align*}

	For fixed $\eps$ and $J$,
		Steps 3 and 4 give $\limsup_n\pr\big(d_{M_1}(Y_n,W_n)>13(1+C_\ast)\eps\big)\le Ct_0(J\eps)^{-\gb}$.
	Letting $J\to\infty$ and then $\eps\downarrow0$ shows $d_{M_1}(Y_n,W_n)\to0$ in probability on $[0,t_0]$, which by Step 2 completes the proof.
\end{proof}

\section{Numerical experiments}\label{sec:sim}
We illustrate both regimes numerically.
The passage time $T(0,n)$ was computed by Dijkstra's algorithm on the complete graph with vertex set $[-m,n+m]\cap\dZ$,
	$m=100$, with $1000$ independent replicates for each $n\in\{250,500,1000,2000\}$ (and $n=4000$ for configuration (2)).
Two edge laws with $\theta=1$ were used: law~A, $\go\sim\mathrm{Uniform}(0,1)$,
	for which $F(t)=t$ on $[0,1]$ and Assumption~\ref{ass:rv} holds with $\slv\equiv1$; and law~B,
	$\go=U(1-U)^{-2}$ with $U$ uniform,
	for which $F(t)\sim t$ at $0$ and $\pr(\go>t)\sim t^{-1/2}$ at infinity, so that $\E\go=\infty$.
For $\theta=1$ the thresholds are $\ga=3$ and $\ga=4$, and $\gb=\ga-2$ in the stable window.
The four configurations and the summary statistics are listed in Table~\ref{tab:sim}.

\begin{table}[htbp]
	\centering\footnotesize
	\begin{tabular}{llcccccc}
		\hline
		config. & law, $\ga$ & largest $n$ & $\hat\mu=\bar T_n/n$ & slope of $\widehat{\var}(T_n)$ vs $n$ & predicted   & skewness & ex.\ kurt. \\
		\hline
		(1)     & A, $5$     & 2000        & $0.4726\pm0.0002$    & $0.98$                                & $1$         & $-0.12$  & $-0.01$    \\
		(2)     & A, $3.5$   & 4000        & $0.3530\pm0.0005$    & $1.35$                                & $2/\gb=4/3$ & $-1.56$  & $4.9$      \\
		(3)     & A, $3.25$  & 2000        & $0.2884\pm0.0010$    & $1.55$                                & $2/\gb=8/5$ & $-1.72$  & $5.6$      \\
		(4)     & B, $5$     & 2000        & $4.810\pm0.010$      & $1.06$                                & $1$         & $+0.26$  & $+0.27$    \\
		\hline
	\end{tabular}
	\smallskip
	\caption{Summary of the simulations.
		In the stable regime, $2/\gb$ denotes twice the fluctuation-scale exponent.}\label{tab:sim}
\end{table}

\begin{figure}[htbp]
	\centering
	\includegraphics[width=.7\textwidth]{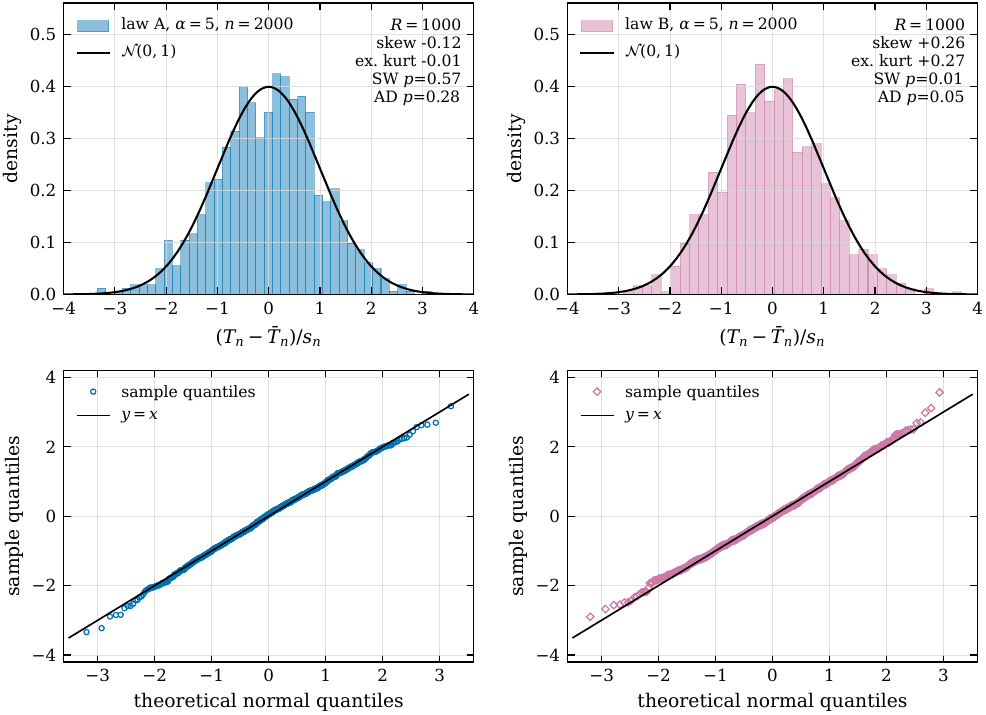}
	\caption{Gaussian regime, $\ga=5$, $n=2000$.
		Left: law~A. Right: law~B ($\E\go=\infty$).
		Top: histograms of the standardized passage time.
		Bottom: normal Q-Q plots.
        }\label{fig:gauss}
\end{figure}

\begin{figure}[htbp]
	\centering
	\includegraphics[width=.7\textwidth]{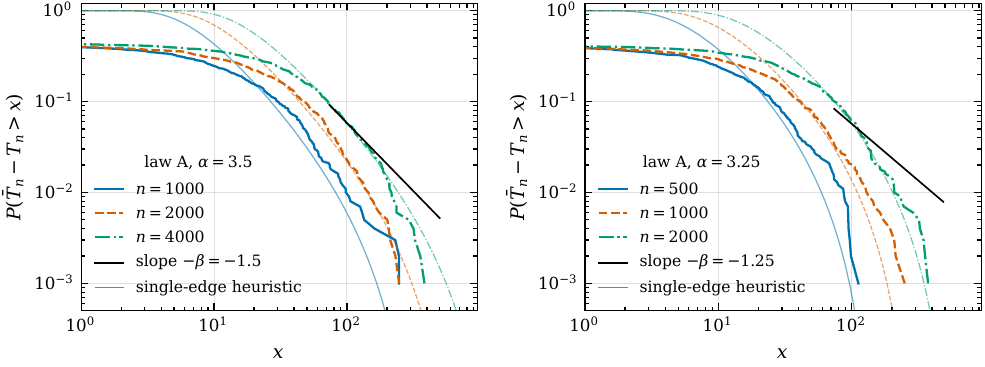}
	\caption{Left tail $\pr(\bar T_n-T_n>x)$ on logarithmic axes for $\ga=3.5$ (left) and $\ga=3.25$ (right) and several $n$,
			with the reference slope $-\gb$ and the single-edge prediction $1-\exp\{-\sum_{d\le n}(n-d+1)F((\hat\mu d-x)_+d^{-\ga})\}$ (thin lines),
			the finite-$n$ version of $n\rhox(x)$.
            }\label{fig:stable-tail}
\end{figure}


\subsection*{Gaussian regime}
Configurations (1) and (4) are consistent with Theorem~\ref{thm:main}:
	the mean and the variance grow linearly (fitted log--log variance slopes $0.98$ and $1.06$),
	and the standardized passage times are Gaussian to statistical accuracy for law~A at every $n$ (Figure~\ref{fig:gauss}).
For law~B, whose edge weights have infinite mean,
	the simulations show a finite empirical time constant and fluctuations on the order of $\sqrt n$;
	the residual right skew ($0.65$ at $n=250$, $0.26$ at $n=2000$) decreases with $n$.
In both cases the sampled geodesics use edges of length at most about $20$,
	with roughly $0.94$ (law~A) or $0.67$ (law~B) edges per unit length.

\subsection*{Stable regime}
Configurations (2) and (3) illustrate the qualitative predictions of Theorems~\ref{thm:stable} and~\ref{thm:stable-functional}.
The finite-sample fluctuation diagnostic grows near the predicted order $n^{1/\gb}$ (fitted sample-variance slopes $1.35$ and $1.55$ against the squared fluctuation-scale exponents $2/\gb=4/3$ and $8/5$).
The fitted laws are skewed to the left:
	every reported maximum-likelihood stable fit returns the skewness parameter $-1$,
	and the right tail of $T_n-\bar T_n$ has an estimated index of about $6$--$7$.
The finite-window single-edge prediction matches the displayed left-tail data for $x\gtrsim n^{1/\gb}$ (Figure~\ref{fig:stable-tail}),
	and the sampled geodesics contain long edges (mean longest edge $113$ and maximum $1286$ at $n=2000$ for $\ga=3.5$).
The convergence of the bulk to the proposed stable law appears slow:
	the freely fitted stable index is about $1.6$--$1.7$ for both $\gb=1.5$ and $\gb=1.25$ up to $n=4000$.
A possible explanation is finite-window truncation.
An edge of length $d$ saves at most $\hat\mu d$ and the simulated window contains no edge longer than $n+2m$,
	so the tail of $\bar T_n-T_n$ is cut off at $x\approx\hat\mu n$.

\appendix
\section{Nearest-neighbor hubs and sharp moment thresholds}\label{app:nn}

In this appendix we retain Assumption~\ref{ass:law} and $\ga>1+2/\theta$, and additionally assume $\E\go<\infty$ and describe the hub construction based on the nearest-neighbor random walk.
It is the special case $\ell=1$, $B=1$, $Y_m=\go_{m,m+1}$ of Section~\ref{sec:spread},
	without the portal, and it has the advantage that the moment thresholds for the hub gaps can be characterized exactly.
Write $\mu_1:=\E\go$, choose $\mu_1<c<2^{\ga-1}\operatorname{ess\,sup}\go$ (with no finite upper restriction for unbounded support),
	and define $S[x,y]:=\sum_{i=x}^{y-1}\go_{i,i+1}$ for $x<y$, $S[x,x]=0$.
This choice exists since $\ga>1$ and ensures $F(cd^{1-\ga})<1$ for every $d\ge2$,
	so that no finite-length factor of the crossing product vanishes.
For $x\in\dZ$ define
\begin{align*}
	A_x^{++}:=\bigcap_{k\ge1}\{S[x,x+k]\le ck\},\ 
    A_x^{--}:=\bigcap_{k\ge1}\{S[x-k,x]\le ck\},\ 
	A_x^{-+}:=\bigcap_{j,k\ge1}\{t_{x-j,x+k}>c(j+k)\}.
\end{align*}
The first two events say that the nearest-neighbor walk started at $x$ stays below the line of slope $c$ in both directions.
The third says that every edge crossing $x$ costs more than $c$ times its length.
The three events are independent, and $x$ is a \emph{nearest-neighbor hub} if $A_x:=A_x^{++}\cap A_x^{--}\cap A_x^{-+}$ occurs.

Every geodesic from the left of a nearest-neighbor hub to its right passes through it,
	and geodesics between consecutive hubs stay in the corresponding interval.
We use the notation of the main text:
	$\hub_0\le0<\hub_1<\hub_2<\cdots$ are the hubs in the stationary enumeration,
	$\pr^0$ is the Palm law, and $L$ is the Palm gap.
Let $\mvgs:=\{x:A^{-+}_x\}$ be the set of points crossed by no cheap edge,
	with Palm law $\pr^{0,\gs}$ and first point $\gs_1>0$ after $0$,
	and let $\gL_{\mathrm{nn}}:=\{x:A_x^{++}\cap A_x^{--}\}$ be the set of two-sided ladder points of the walk $S[0,\cdot]-c\,\cdot$.

\begin{lem}[Existence]\label{lem:nn-hub}
	If $\ga>1+2/\theta$ and $\E\go<\infty$, then $\pr(A_x)>0$.
	Consequently the nearest-neighbor hubs form a stationary ergodic point process of positive density,
		with the regeneration structure of Proposition~\ref{prop:regeneration}.
\end{lem}

\begin{proof}
	For $N\ge1$ let $D_N:=\bigcap_{k>N}\{S[x,x+k]\le ck\}$.
	Then $D_N\uparrow$ and $\pr(\bigcup_ND_N)=1$ by the strong law of large numbers, so $\pr(D_{N_0})>1/2$ for some $N_0$.
	The event $\bigcap_{k\le N_0}\{S[x,x+k]\le ck\}$ has positive probability,
		since $0$ is in the support of $\go$, and both events are decreasing in the weights,
		so by Harris' inequality $\pr(A_x^{++})\ge\pr(\bigcap_{k\le N_0}\{\cdots\})\pr(D_{N_0})>0$.
	Similarly $\pr(A_x^{--})>0$.
	By independence of the crossing edges,
	\begin{align*}
		\pr(A_x^{-+})=\prod_{j,k\ge1}\bar F\big(c(j+k)^{1-\ga}\big)=\exp\Big(\sum_{m\ge2}(m-1)\log\bar F(cm^{1-\ga})\Big),
	\end{align*}
	and since $F(cm^{1-\ga})\le Cm^{-(\theta-\eps)(\ga-1)}$ for large $m$ by~\eqref{eq:two-sided-F},
		with $\log(1-v)\ge-2v$ for $v\le1/2$,
		the tail of the exponent is bounded below by a convergent series when $\theta(\ga-1)>2$.
	The finite initial factors are positive by the choice of $c$.
	No assertion at critical equality is made.
	The hub event has the cut structure of Lemma~\ref{lem:cut} with $H^\pm_x=A_x^{\pm\pm}$ and $H^\times_x=A_x^{-+}$,
		and the localization property holds by the same computation,
		so the proof of Proposition~\ref{prop:regeneration}.(i) applies.
\end{proof}

\begin{thm}[Sharp moment thresholds]\label{thm:momhub}
	Assume $\ga>1+2/\theta$, $\E\go<\infty$, and let $q>0$.
	\begin{enumeratei}
		\item\label{momhub:i} If $q>\theta(\ga-1)-1$, then $\E^0(L^q)=\infty$, and $\E(\hub_1^{q-1})=\infty$ if $q>1$.
		\item\label{momhub:ii} If $q<\theta(\ga-1)-1$ and $\E(\go^p)<\infty$ for some $p>2q+6$, then $\E^0(L^q)<\infty$,
				and $\E(\hub_1^{q-1})<\infty$ if $q>1$.
	\end{enumeratei}
	In particular, when $\go$ has all moments, $\E^0(L^q)<\infty$ if and only if $q<\theta(\ga-1)-1$, up to the boundary case.
	For integer $k$, excluding the critical equality $\ga=1+(k+1)/\theta$,
		this reads $\E^0(L^k)<\infty\iff\ga>1+(k+1)/\theta$.
\end{thm}

The statements about $\hub_1$ follow from those about $L$ by~\eqref{eq:palm-stat}, since for $q>1$,
	$\E(\hub_1^{q-1})<\infty$ if and only if $\E^0(L^q)<\infty$.
The proof of~\ref{momhub:i} uses a union over long edges from zero,
	and the proof of~\ref{momhub:ii} is the spacing argument of Lemma~\ref{lem:second-moment}.

\begin{proof}[Proof of~\ref{momhub:i}]
	Since hubs are points of $\mvgs$, $\hub_1\ge\gs_1$.
	The events $A_0^{++}$, $A_0^{--}$ are independent of the crossing edges,
		so under $\pr^0$ the set $\mvgs\cap(0,\infty)$ has its Palm law $\pr^{0,\gs}$ and $\E^0(L^q)\ge\E^{0,\gs}(\gs_1^q)$.
	If some edge $\{0,j\}$ with $j\ge m$ satisfies $t_{0,j}\le cj$,
		every point of $(0,j)$ is crossed by a cheap edge, so $\gs_1\ge m$.
	The edges with left endpoint $0$ are independent of $A_0^{-+}$ and of each other. Hence, using $1-\prod_j(1-x_j)\ge\frac12\min(1,\sum_jx_j)$ and~\eqref{eq:two-sided-F},
	\begin{align*}
		\pr^{0,\gs}(\gs_1\ge m)\ge1-\prod_{j\ge m}\bar F(cj^{1-\ga})\ge\frac12\min\Big(1,\sum_{j\ge m}F(cj^{1-\ga})\Big)\ge c\,m^{1-(\theta+\eps)(\ga-1)}
	\end{align*}
	for large $m$.
	If $q>\theta(\ga-1)-1$, choose $\eps$ with $q>(\theta+\eps)(\ga-1)-1$.
	Then $\sum_mm^{q-1}\pr^{0,\gs}(\gs_1\ge m)=\infty$, so $\E^{0,\gs}(\gs_1^q)=\infty$.
\end{proof}

\begin{proof}[Proof of~\ref{momhub:ii}]
	For $0<q\le1$, finiteness follows from $\E^0L<\infty$.
	In the following tail-sum argument assume $q>1$.
	The set $\gL_{\mathrm{nn}}$ is a stationary renewal set determined by the i.i.d.~sequence $(\go_{m,m+1})$.
	The proof of Lemma~\ref{lem:loc}, with $Y_q=\go_{q,q+1}$,
		$a=c$ and the two envelope conditions omitted,
		shows that its Palm gap has finite moments of every order $r<(p-2)/2$,
		so that $\pr(\gL_{\mathrm{nn}}\cap(0,D]=\eset)\le CD^{1-r}$.
	The crossing structure is a Bernoulli field,
		an edge of length $d\ge2$ is \emph{cheap} independently with probability $p'_d:=F(cd^{1-\ga})\le C_\eps d^{-(\theta-\eps)(\ga-1)}$,
		and a point is a hub if and only if it lies in $\gL_{\mathrm{nn}}$ and is crossed by no cheap edge.
	Define $\gb':=\theta(\ga-1)-2>q-1$ and choose $\eps$ so small that $\sum_{d\ge D}d\,p'_d\le CD^{-\gb'+\eps'}$ with $\gb'-\eps'>q-1$.
	Fix $\kappa\in(0,1)$ so close to $1$ that $\kappa(1+\gb'-\eps')>q$,
		and $r<(p-2)/2$ with $\kappa r>q+1$, which is possible because $p>2q+6$.
	With $D=\floor{u^\kappa}$ and $cu^{1-\kappa}\le K'\le C(u/D+1)$ spaced points of $\gL_{\mathrm{nn}}\cap(0,u]$ as in the proof of Lemma~\ref{lem:second-moment},
	\begin{align*}
		\pr\big(\text{no hub in }(0,u]\big)\le Cu^{1-\kappa r}+K'\sum_{d\ge D}d\,p'_d+\big(1-\pr(A_0^{-+})\big)^{K'}\le Cu^{-(q-1)-\gd}
	\end{align*}
	for some $\gd>0$.
	By~\eqref{eq:palm-stat}, $\pr(\hub_1\ge v)=\E^0((L-v+1)_+)/\E^0(L)$,
		and $\sum_vv^{q-2}\E^0((L-v+1)_+)\asymp\E^0(L^q)$,
		so the bound $\pr(\hub_1\ge v)\le Cv^{-(q-1)-\gd}$ gives $\E^0(L^q)<\infty$.
\end{proof}

\bibliographystyle{alpha}
\bibliography{lrfpp}
\end{document}